\documentclass[a4paper,10pt]{article}
\usepackage[utf8]{inputenc}
\usepackage{setspace}
\usepackage{amsmath}
\numberwithin{equation}{section}
\usepackage{amsfonts}
\usepackage{amssymb}
\usepackage{amsthm}
\usepackage{enumerate}
\usepackage{enumitem}
\usepackage{mathtools}
\usepackage{multicol}
\usepackage{chngcntr}
\usepackage{booktabs}

\usepackage[noblocks]{authblk} % for author affilitations

\usepackage[section]{algorithm}
\usepackage{algpseudocode}

\usepackage[colorlinks=true, allcolors=blue, linktocpage=true]{hyperref}
\usepackage[left=2.5cm,right=2.5cm,top=3cm,bottom=3cm]{geometry}
\usepackage[table]{xcolor}
\usepackage{longtable}

\usepackage{subcaption}

\usepackage{sidecap}
\sidecaptionvpos{figure}{t}

\definecolor{mygreen}{RGB}{47,130,100}

\newcount\Comments  % 0 suppresses notes to selves in text
\newcommand{\kibitz}[2]{\ifnum\Comments=1\textcolor{#1}{#2}\fi}

\newcommand{\revcomm}[1]{\ifnum\Comments=1\textcolor{red}{#1}\else #1\fi}
\newcommand{\revcommblock}[1]{\ifnum\Comments=1{\color{red}#1}\else #1\fi}
\def\N{\mathbb{N}}
\def\R{\mathbb{R}}
\def\C{\mathbb{C}}
\def\Gr{\mathrm{Gr}}
\def\St{\mathrm{St}}
\def\SO{\mathrm{SO}}
\def\so{\mathfrak{so}}
\def\O{\mathrm{O}}
\def\U{\mathrm{U}}
\def\GL{\mathrm{GL}}
\def\RP{\mathbb{RP}}
\def\Sym{\mathrm{Sym}}
\def\Vectorfields{\mathfrak{X}}
\def\D{\mathrm{d}}
\def\OS{\mathrm{OS}}
\def\SG{\mathrm{SG}}
\def\OG{\mathrm{OG}}

\newcommand{\mcM}{\mathcal{M}}

\newcommand{\Space}[1]{\mathcal{#1}}

\DeclareMathOperator{\Exp}{Exp}
\DeclareMathOperator{\Log}{Log}
\DeclareMathOperator{\expm}{exp_m}

\DeclareMathOperator*{\argmin}{arg\,min}
\DeclareMathOperator{\tr}{tr}
\DeclareMathOperator{\Span}{span}
\DeclareMathOperator{\rank}{rank}
\DeclareMathOperator{\diag}{diag}
\DeclareMathOperator{\Hor}{\mathsf{Hor}}
\newcommand{\hor}{\mathsf{hor}}
\DeclareMathOperator{\Ver}{\mathsf{Ver}}
\DeclareMathOperator{\dist}{dist}

\DeclareMathOperator{\grad}{grad}
\DeclareMathOperator{\Cut}{Cut}
\newcommand{\cut}{\mathrm{cut}}
\DeclareMathOperator{\TCL}{TCL}
\DeclareMathOperator{\ID}{ID}
\DeclareMathOperator{\Conj}{Conj}
\DeclareMathOperator{\inj}{inj}
\DeclareMathOperator{\eucl}{eucl}
\DeclareMathOperator{\id}{id}

\DeclareMathOperator{\colspan}{colspan}

\newcommand{\SP}[1]{\left\langle #1\right\rangle}

\DeclarePairedDelimiterX{\norm}[1]{\lVert}{\rVert}{#1}

\makeatletter
\renewcommand*\env@matrix[1][*\c@MaxMatrixCols c]{%
  \hskip -\arraycolsep
  \let\@ifnextchar\new@ifnextchar
  \array{#1}}
\makeatother

\counterwithin{figure}{section}
\counterwithin{table}{section}

\newtheorem{theorem}[algorithm]{Theorem}
\newtheorem{proposition}[algorithm]{Proposition}
\newtheorem*{remark}{Remark}

\title{A Grassmann Manifold Handbook: Basic Geometry and Computational Aspects}
\author{Thomas Bendokat}
\author{Ralf Zimmermann}
\affil{Department of Mathematics and Computer Science, University of Southern Denmark (SDU), Odense, Denmark (Current address \href{mailto:bendokat@mpi-magdeburg.mpg.de}{bendokat@mpi-magdeburg.mpg.de}, \href{mailto:zimmermann@imada.sdu.dk}{zimmermann@imada.sdu.dk})}
\author{P.-A. Absil}
\affil{ICTEAM Institute, UCLouvain, 1348 Louvain-la-Neuve, Belgium (\href{https://sites.uclouvain.be/absil/}{https://sites.uclouvain.be/absil/})}
\date{}

\ifpdf
\hypersetup{
  pdftitle={A Grassmann Manifold Handbook: Basic Geometry and Computational Aspects},
  pdfauthor={T. Bendokat, R. Zimmermann, and P.-A. Absil}
}
\fi

\begin{document}

\maketitle

\begin{abstract}
  The Grassmann manifold of linear subspaces is important for the mathematical modelling of a multitude of applications, ranging from problems in machine learning, computer vision and image processing to low-rank matrix optimization problems, dynamic low-rank decompositions and model reduction. With this \revcomm{mostly expository} work, we aim to provide a  collection of the essential facts and formulae on the geometry of the Grassmann manifold in a fashion that is fit for tackling the aforementioned problems with matrix-based algorithms.
 Moreover, we expose the Grassmann geometry both from the approach of representing subspaces with orthogonal projectors and when viewed as a quotient space of the orthogonal group, where subspaces are identified as equivalence classes of (orthogonal) bases. This bridges the associated research tracks and allows for an easy transition between these two approaches.
 
 Original contributions include a modified algorithm for computing the Riemannian logarithm map on the Grassmannian that is advantageous numerically but also allows for a more elementary, yet more complete description of the cut locus and the conjugate points. We also derive a formula for parallel transport along geodesics in the orthogonal projector perspective, formulae for the derivative of the exponential map, as well as a formula for Jacobi fields vanishing at one point. 
\\

\textbf{Keywords:} Grassmann manifold, Stiefel manifold, orthogonal group, Riemannian exponential, geodesic, Riemannian logarithm, cut locus, conjugate locus, curvature, parallel transport, quotient manifold, horizontal lift, subspace, singular value decomposition

\textbf{AMS subject classifications:}
  15-02, 15A16, 15A18, 15B10,  22E70, 51F25, 53C80, 53Z99
  
\end{abstract}

\section*{Notation}
\begin{small}
\renewcommand{\arraystretch}{1.5}
\begin{center}
\rowcolors{2}{gray!15}{white}
\begin{longtable}{c|c|p{70mm}}
 \rowcolor{gray!25}
 \textbf{Symbol} & \textbf{Matrix Definition} & \textbf{Name} \\
 \hline
 $I_p$ & $\diag(1, \dots, 1)\in \R^{p \times p}$  & Identity matrix\\
 $I_{n,p}$ & $\begin{psmallmatrix} I_p \\ 0\end{psmallmatrix}\in \R^{n \times p}$ & \\
 $\Sym_n$ & $\{X \in \R^{n \times n}\mid X=X^T \}$ & Space of symmetric matrices\\
 $\O(n)$ & $\left\lbrace Q \in \R^{n \times n} \ \middle|\ Q^TQ=I_n=QQ^T\right\rbrace$ & Orthogonal group \\
 $T_Q\O(n)$ & $\{ Q\Omega \in \R^{n \times n} \mid \Omega^T=-\Omega\}$ & Tangent space of $\O(n)$ at $Q$\\
 $\St(n,p)$ & $\{U \in \R^{n \times p} \mid U^TU=I_p\}$ & Stiefel manifold\\
 $T_U\St(n,p)$ & $\{ D \in \R^{n \times p} \mid U^TD=-D^T U\}$ & Tangent space of $\St(n,p)$ at $U$\\
 $\Gr(n,p)$ & $\{ P \in \Sym_n \mid P^2 = P,\ \rank(P)=p\}$ & Grassmann manifold\\
 $T_P\Gr(n,p)$ & $\{ \Delta \in \Sym_n \mid \Delta=P\Delta + \Delta P\}$ & Tangent space of $\Gr(n,p)$ at $P$\\
 $U_\perp$ & $(U\ U_\perp) \in \O(n)$ & Orthogonal completion of $U \in \St(n,p)$\\
 $g^\St_U(D_1,D_2)$ & $\tr(D_1^T(I_n - \frac{1}{2}UU^T)D_2)$ & (Quotient) metric in $T_U\St(n,p)$\\
 $g^\Gr_P(\Delta_1,\Delta_2)$ & $\frac{1}{2}\tr(\Delta_1^T\Delta_2)$ & Riemannian metric in $T_P\Gr(n,p)$\\
 $\pi^\OS(Q)$ & $QI_{n,p}$ & Projection from $\O(n)$ to $\St(n,p)$\\
 $\pi^\SG(U)$ & $UU^T$ & Projection from $\St(n,p)$ to $\Gr(n,p)$\\
 $\pi^\OG(Q)$ & $Q I_{n,p}I_{n,p}^TQ^T$ & Projection from $\O(n)$ to $\Gr(n,p)$\\
 $\Ver_U\St(n,p)$ & $\{ UA \in \R^{n \times p} \mid A^T = -A \in \R^{p \times p} \}$ & Vertical space w.r.t. $\pi^\SG$\\
 $\Hor_U\St(n,p)$ & $\{U_\perp B \in \R^{n \times p} \mid B \in \R^{(n-p) \times p}\}$ & Horizontal space w.r.t. $\pi^\SG$\\
 $\Delta^\hor_U$ & $\Delta U \in \R^{n \times p}$ & Horizontal lift of $\Delta \in T_P\Gr(n,p)$ to $\Hor_U\St(n,p)$\\
 $[U]$ & $\{UR \in \St(n,p) \mid R \in \O(p) \}$ & Equivalence class representing a point in $\Gr(n,p)$\\
 $\Exp^\Gr_P(t\Delta)$ & $\pi^\SG(UV\cos(t\Sigma) + \hat{Q}\sin(t\Sigma))$ & Riemannian exponential for $\Delta^\hor_U \overset{\text{\tiny SVD}}{=} \hat{Q}\Sigma V^T \in \Hor_U\St(n,p)$\\
 $\Log^\Gr_P(F)$ & $\Delta \in T_P\Gr(n,p)$ s.t. $\Exp^\Gr_P(\Delta)=F$ & Riemannian logarithm in $\Gr(n,p)$\\
 $K_P(\Delta_1,\Delta_2)$ & $4\frac{\tr\mathopen{}\left(\Delta_1^2\Delta_2^2\right) - \tr\mathopen{}\left((\Delta_1\Delta_2)^2\right)}{\tr(\Delta_1^2)\tr(\Delta_2^2)-(\tr(\Delta_1\Delta_2))^2}$ & Sectional curvature of $\Gr(n,p)$
\end{longtable}
\end{center}
\end{small}

\section{Introduction}
\label{sec:intro}
The collection of all linear subspaces of fixed dimension $p$ of the Euclidean space $\R^{n}$ forms the Grassmann manifold $\Gr(n,p)$, also termed the Grassmannian.
Subspaces and thus Grassmann manifolds play an important role in a large variety of applications.
These include, but are not limited to, data analysis and signal processing \cite{Gallivan_etal2003, Rahman_etal2005, Rentmeesters2013}, 
subspace estimation and subspace tracking \cite{boumal2015rtrmcextended,Balzano2015, ZhangBalzano2016},
structured matrix optimization problems \cite{EdelmanAriasSmith1999,AbsilMahonySepulchre2004, AbsilMahonySepulchre2008}, dynamic low-rank decompositions
\cite{Hairer06gni,KochLubich2007}, projection-based parametric model reduction \cite{AmsallemFarhat2008,Nguyen2012,Zimmermann2014,
SonStykel2015,ZimmermannPeherstorferWillcox_2017} and computer vision \cite{Lui2012}, see also the collections \cite{Minh:2016:AAR:3029338, RiemannInComputerVision}.
Moreover, Grassmannians are extensively studied for their purely mathematical aspects \cite{Leichtweiss1961, Wong1967, Wong1968, Wong1968b, Sakai1977, MachadoSalavessa1985, Kozlov} and often serve as illustrating examples in the differential geometry literature
\cite{KobayashiNomizu1996, helgason2001differential}.

In this work, we approach the Grassmannian from a matrix-analytic perspective.
The focus is on the computational aspects as well as on geometric concepts that directly or indirectly feature in
matrix-based algorithmic applications.
The most prominent approaches of representing points on Grassmann manifolds with matrices in computational algorithms are 
\begin{itemize}
     \item {\em the basis perspective}: A subspace $\Space{U}\in \Gr(n,p)$ is identified with a (non-unique) matrix $U\in \R^{n\times p}$ whose columns form a basis of $\Space{U}$. In this way, a subspace is identified with the equivalence class of all rank-$p$ matrices whose columns span $\Space{U}$. For an overview of this approach, see for example the survey \cite{AbsilMahonySepulchre2004}. A brief introduction is given in \cite{HelmkeMoore1994}.
     \item {\em the ONB perspective}: In analogy to the basis perspective above, a subspace $\Space{U}$ may be identified with the equivalence class of matrices whose columns form an orthonormal basis (ONB) of $\Space{U}$. This is often advantageous in numerical computations. This approach is surveyed in \cite{EdelmanAriasSmith1999}.
    \item {\em the projector perspective}: A subspace $\Space{U}\in \Gr(n,p)$ is identified with the (unique) orthogonal projector $P\in\R^{n\times n}$ onto $\Space{U}$, which in turn is uniquely represented by $P=UU^T$, with $U$ from the ONB perspective above. For an approach without explicit matrices see \cite{MachadoSalavessa1985}, and for the approach with matrices see for example \cite{HelmkeHueperTrumpf2007, BatziesHueperMachadoLeite2015, HelmkeMoore1994}.
    \item {\em the Lie group perspective}: A subspace $\Space{U} \in \Gr(n,p)$ is identified with an equivalence class of orthogonal $n \times n$ matrices. This perspective is for example taken in \cite{Gallivan_etal2003,Srivastavaetal2005}.
\end{itemize}
These approaches are closely related and all of them rely on Lie group theory to some extent. 
Yet, the research literature on the basis/ONB perspective and the projector perspective is rather disjoint. The recent preprint \cite{lai2020} proposes yet another perspective, namely representing $p$-dimensional subspaces as symmetric orthogonal matrices of trace $2p-n$. This approach corresponds to a scaling and translation of the projector matrices in the vector space of symmetric matrices, hence it yields very similar formulae.

\revcomm{
    There are at least two other important perspectives on the Grassmann manifold, which are however not treated further in this work, as they are mainly connected to the field of algebraic geometry. The first are \emph{Plücker embeddings}, where $\Gr(n,p)$ is embedded into the projective space $\mathbb{P}^{\binom{n}{p}-1}$, which is done by representing every point, i.e., subspace, in $\Gr(n,p)$ by the determinants of all $p \times p$ submatrices of a matrix spanning that subspace. The second perspective are \emph{Schubert varieties}, where the Grassmannian is partitioned into so called \emph{Schubert cells}. For details on both of those approaches, see for example \cite{Gillespie2019} and several references therein.}

\revcomm{
    The Grassmann manifold can also be defined for the complex case, which features less often in applications, as far as the authors are aware. Here, complex $p$-dimensional subspaces of $\C^n$ are studied. Most of the formulas in this handbook can be transferred to the complex case with analogous derivations, by replacing the orthogonal group $\O(n)$ with the unitary group $\U(n)$, and the transpose with the conjugate transpose. For a study of complex Grasmannians, see for example \cite[Section 5]{MachadoSalavessa1985} and \cite{Berceanu1997}.
}

\paragraph{Raison d'\^etre and original contributions}
We treat the Lie group approach, the ONB perspective and the projector perspective simultaneously.
This may serve as a bridge between the corresponding research tracks.
Moreover, we collect the essential facts and concepts that feature as generic tools and building blocks in Riemannian computing problems on the Grassmann manifold % such as optimization, interpolation, regression, data averaging
in terms of matrix formulae, fit for algorithmic calculations.
This includes, among others, the Grassmannian's quotient space structure (Subsection~\ref{subsec:QuotientStructure}), the Riemannian metric (Subsection~\ref{subsec:RiemannianMetric}) and distance, the Riemannian connection (Subsection~\ref{subsec:RiemmanianConnection}), the Riemannian exponential (Subsection~\ref{subsec:ExponentialMap}) and its inverse, the Riemannian logarithm (Subsection~\ref{subsec:RiemannianLogarithm}), as well as the associated Riemannian normal coordinates (Section~\ref{sec:Parameterizations}), parallel transport of tangent vectors (Subsection~\ref{subsec:ParallelTransport}) and the sectional curvature (Subsection~\ref{subsec:SectionalCurvature}).
Wherever possible, we provide self-contained and elementary derivations of the sometimes classical results. Here, the term elementary is to be understood as ``via tools from linear algebra and matrix analysis'' rather than ``via tools from abstract differential geometry''. Care has been taken that the quantities that are most relevant for algorithmic applications are stated in a form that allows calculations that scale in $\mathcal{O}(np^2)$.

As novel research results, we provide a modified algorithm (Algorithm~\ref{alg:modgrasslog}) for computing the Riemannian logarithm map on the Grassmannian that has favorable numerical features and 
additionally allows to (non-uniquely) map points from the cut locus of a point to its tangent space. Therefore any set of points on the Grassmannian can be mapped to a single tangent space (Theorem~\ref{thm:GrLog} and Theorem~\ref{thm:multiple_shortest_geodesics}). In particular, we give explicit formulae for the (possibly multiple) shortest curves between any two points on the Grassmannian as well as the corresponding tangent vectors. Furthermore, we present a more elementary, yet more complete description of the conjugate locus of a point on the Grassmannian, which is derived in terms of principal angles between subspaces (Theorem~\ref{thm:conjugatelocus}).
We also derive a formula for parallel transport along geodesics in the orthogonal projector perspective (Proposition~\ref{prop:parallel_transport_projector}), formulae for the derivative of the exponential map (Subsection~\ref{sec:DifferentialOfExp}), as well as a formula for Jacobi fields vanishing at one point (Proposition~\ref{prop:Jacobifields}).

\paragraph{Organization}
Section~\ref{sec:GrassQuotintro} introduces the manifold structure of the Grassmann manifold and provides basic formulae for representing Grassmann points and tangent vectors via matrices. Section~\ref{sec:RiemannStruct} recaps the essential Riemann-geometric aspects of the Grassmann manifold, including the Riemannian exponential, its derivative and parallel transport. In Section~\ref{sec:symcurve}, the Grassmannian's symmetric space structure is established by elementary means and used to explore the sectional curvature and its bounds. In Section~\ref{sec:CutLocusLogarithm}, the (tangent) cut locus is described and a new algorithm is proposed to calculate the pre-image of the exponential map, i.e. the Riemannian logarithm where the pre-image is unique. Section~\ref{sec:Parameterizations} addresses normal coordinates and local parameterizations for the Grassmannian. In Section~\ref{sec:conjugatepoints}, questions on Jacobi fields and the conjugate locus of a point are considered. Section~\ref{sec:Conclusion} concludes the paper.

\section{The Manifold Structure of the Grassmann Manifold}
\label{sec:GrassQuotintro}
    In this section, we recap the definition of the Grassmann manifold and connect results from \cite{EdelmanAriasSmith1999,BatziesHueperMachadoLeite2015,MachadoSalavessa1985,HelmkeHueperTrumpf2007}. Tools from Lie group theory establish the quotient space structure of the Grassmannian, which
    gives rise to efficient representations. The required Lie group background can be found in the appendix and in~\cite{Hall_Lie2015}, \cite[\revcomm{Chapters 7 \& 21}]{Lee2012smooth}.
    
    The \emph{Grassmann manifold} (also called \emph{Grassmannian}) is defined as the set of all $p$-dimensional subspaces of the Euclidean space $\R^n$. 
    This set can be identified with the set of orthogonal rank-$p$  projectors,
    \begin{equation}
    \label{eq:GrassmProj}
        \Gr(n,p) := \left\lbrace P \in \R^{n \times n}\ \middle|\ P^T=P,\ P^2=P,\ \rank{P}=p \right\rbrace,
    \end{equation}
    as is for example done in \cite{HelmkeHueperTrumpf2007,BatziesHueperMachadoLeite2015}. Note that a projector $P$ is symmetric as a matrix (namely, $P^T=P$) if and only if it is orthogonal as a projection operation (its range and null space are mutually orthogonal)~\cite[\S 3]{Eigenvalue_Problems_Saad}. The identification in~\eqref{eq:GrassmProj} associates $P$ with the subspace $\Space{U}:=\mathrm{range}(P)$. Every $P \in \Gr(n,p)$ can in turn be identified with an equivalence class of orthonormal basis matrices spanning the same subspace; an approach that is for example chosen in \cite{EdelmanAriasSmith1999}. These ONB matrices are elements of the so called \emph{Stiefel manifold}
    \begin{equation*}
        \St(n,p):=\left\lbrace U \in \R^{n \times p} \ \middle|\ U^TU=I_p \right\rbrace.
    \end{equation*}
    The link between these two sets is via the projection
    \begin{equation*}
        \pi^\SG\colon \St(n,p) \to \Gr(n,p),\ U \mapsto UU^T.
    \end{equation*}
    
    To obtain a manifold structure on $\Gr(n,p)$ and $\St(n,p)$, \revcomm{i.e., endow those sets with coordinate patches that overlap smoothly,} we recognize these matrix sets
    as quotients of the \emph{orthogonal group}
    \begin{equation}
    \label{eq:On}
        \O(n):=\left\lbrace Q \in \R^{n \times n} \ \middle|\ Q^TQ=I_n=QQ^T\right\rbrace,
    \end{equation}
    which is a compact Lie group\revcomm{, i.e., a group that also is a compact manifold, for which the multiplication and inversion operation are smooth maps, respectively. Quotients of Lie groups identify sets of group elements as equivalent. It is a standard construction that quotients of Lie groups are themselves manifolds under certain assumptions, c.f. \cite[Chapter 21]{Lee2012smooth}, so called \emph{homogeneous spaces}. When the Lie group is compact, many constructions for homogeneous spaces simplify, as this guaranteesthe existence of a \emph{bi-invariant Riemannian metric} on the Lie group \cite[Corollary 3.15]{Lee2018riemannian}, associating the (manifold) curvature of the Lie group with its algebraic structure.}
    For a brief introduction to Lie groups and their actions, see Appendix~\ref{app:LieGroups}. The link from $\O(n)$ to $\St(n,p)$ and $\Gr(n,p)$ is given by the projections
    \begin{equation*}
        \pi^\OS\colon \O(n) \to \St(n,p),\ Q \mapsto Q(:,1:p),
    \end{equation*}
    where $Q(:,1:p)$ is the matrix formed by the $p$ first columns of $Q$, and 
    \begin{equation*}
        \pi^\OG:=\pi^\SG \circ \pi^\OS\colon \O(n) \to \Gr(n,p),\ Q \mapsto Q(:,1:p)Q(:,1:p)^T,
    \end{equation*}
    respectively. We can consider the following hierarchy of quotient structures:
    \begin{quote}
        Two square orthogonal matrices $Q,\tilde{Q}\in \O(n)$ determine the same rectangular, column-orthonormal matrix $U\in \R^{n\times p}$, if both $Q$ and $\tilde{Q}$ feature $U$ as their first $p$ columns.
        Two column-orthonormal matrices $U,\tilde{U}\in \R^{n\times p}$ determine the same subspace, if they differ by an orthogonal coordinate change.
    \end{quote}
    This hierarchy is visualized in Figure \ref{fig:Quotient}. In anticipation of the upcoming discussion, the figure already indicates the lifting of tangent vectors according to the quotient hierarchy.
    \begin{figure}[ht]
    \centering
        \includegraphics{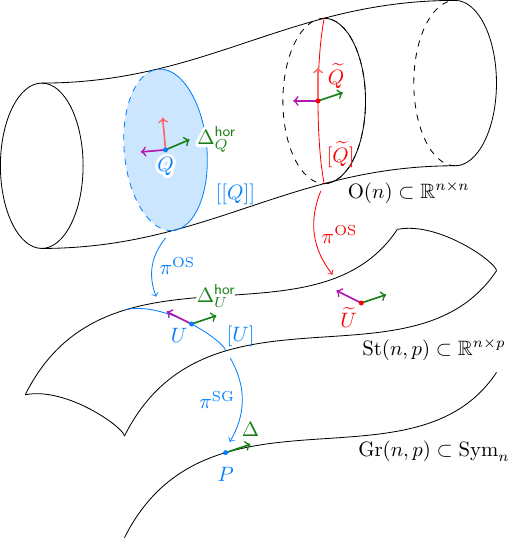}
        \caption{Conceptual visualization of the quotient structure of the Grassmann manifold. The double brackets $[[\cdot]]$ denote an equivalence class with respect to $\pi^\OG=\pi^\SG \circ \pi^\OS$, while the single brackets $[\cdot]$ denote an equivalence class with respect $\pi^\OS$ or $\pi^\SG$, depending on the element inside the brackets. The tangent vectors along an equivalence class for a projection are vertical with respect to that projection, while the directions orthogonal to the vertical space are horizontal. Correspondingly, the horizontal lift of a tangent vector $\Delta \in T_P\Gr(n,p)$ to $T_U\St(n,p)$ or $T_Q\O(n)$ is orthogonal to all vertical tangent vectors at that point. With respect to the projection $\pi^\SG$ from the Stiefel to the Grassmann manifold, the green tangent vector $\Delta^\hor_U$ is horizontal and the magenta tangent vector (along the equivalence class) is vertical. On the other hand, the magenta tangent vectors in $\O(n)$ (pointing to the left) are horizontal with respect to $\pi^\OS$ but vertical with respect to $\pi^\OG$.
        }
         \label{fig:Quotient}
    \end{figure}

\subsection{The Embedded Manifold Structure of the Grassmannian}
    In order to obtain a smooth manifold structure on the set of orthogonal projectors $\Gr(n,p)$, we can advance as in \cite[Proposition 2.1.1]{HelmkeMoore1994}. Define an isometric group action of the orthogonal group $\O(n)$ on the symmetric $n \times n$ matrices $\Sym_n$ by
    \begin{equation*}
        \Phi\colon \O(n)\times \Sym_n \to \Sym_n,\ (Q,S) \mapsto QSQ^T.
    \end{equation*}
    Introduce
    \begin{equation*}
        P_0 := \begin{pmatrix}
            I_p & 0\\
            0   & 0
            \end{pmatrix} \in \Gr(n,p),
    \end{equation*}
    which is the matrix representation of the canonical projection onto the first $p$ coordinates with respect to the Cartesian standard basis.
    The set of orthogonal projectors $\Gr(n,p)$ is the orbit $\Phi(\O(n), P_0)$ of the element $P_0$ under the group action $\Phi$: Any matrix $QP_0Q^T$ obviously satisfies the defining properties of $\Gr(n,p)$ as stated in \eqref{eq:GrassmProj}.
    Conversely, if $P\in \Gr(n,p)$, then $P$ is real, symmetric and positive semidefinite with $p$ eigenvalues equal to one and $n-p$ eigenvalues equal to zero. %\TB{I am not sure about mentioning [33] here either.}
    Hence, the eigenvalue decomposition (EVD) $P=Q\Lambda Q^T = QP_0Q^T$ establishes $P$ as a point in the orbit of $P_0$. In other words, we have confirmed that 
    \begin{equation}
    \label{eq:projectionOGr}
        \pi^\OG= \Phi\vert_{O(n)\times \{P_0\}} \colon \O(n) \to \Gr(n,p),\ Q \mapsto QP_0Q^T,
    \end{equation}
    maps into $\Gr(n,p)$ and is surjective. Since $\O(n)$ is compact, the first part of Proposition~\ref{prop:orbitmanifold} in the appendix shows that $\Gr(n,p) = \Phi(\O(n),P_0)$ is an embedded submanifold of $\Sym_n$. 
    
    \revcomm{This construction also shows that the Grassmannian is connected and even \emph{path-connected}, i.e., between any two points $P_1$, $P_2 \in \Gr(n,p)$, there is a path in $\Gr(n,p)$ joining the two locations: Let $P_1=Q_1P_0Q_1^T$ and $P_2=Q_2P_0Q_2^T$ be the EVDs of $P_1$ and $P_2$. If $Q_1$ or $Q_2$ have determinant $-1$, multiply it from the right with the diagonal matrix $\diag(1,\dots,1,-1)$, which does not change the EVD. As the special orthogonal group $\SO(n) =\{Q\in \O(n)\mid \det Q = 1\}$ is path connected, there is a path between $P_1$ and $P_2$ in $\Gr(n,p)$.}

\subsection{The Quotient Structure of the Grassmannian}
\label{subsec:QuotientStructure}
    To formally introduce the quotient structure of the Grassmannian, we make use of the second part of Proposition~\ref{prop:orbitmanifold}. The objects of interest are the orthogonal group $\O(n)$, which is the domain of $\pi^\OS$ and $\pi^\OG$, and the Cartesian product $\O(p)\times \O(n-p)$, which can be identified with a subgroup of $\O(n)$.
    
    The stabilizer of $\Phi$ at $P_0$, i.e., \revcomm{the set of matrices $Q \in \O(n)$ leaving $P_0$ invariant under} $\pi^\OG$, is given by $H=\{\begin{psmallmatrix}R_1 & 0\\ 0 & R_2 \end{psmallmatrix} \in \O(n) \mid R_1 \in \O(p),\ R_2 \in \O(n-p)\}\cong \O(p) \times \O(n-p)$. This is readily seen by noticing that $Q\in\O(n)$ fulfills  $\pi^\OG(Q) = Q P_0 Q^T = P_0$ if and only if $Q=\begin{psmallmatrix}R_1 & 0\\ 0 & R_2 \end{psmallmatrix}$. An equivalence relation on $\O(n)$ is defined by $\widetilde{Q} \sim Q$ if and only if $\pi^\OG(\widetilde{Q})=\pi^\OG(Q)$. This equivalence relation collects all orthogonal matrices whose first $p$ columns span the same subspace into an equivalence class. In other words, the equivalence classes of $\O(n)/H$ are
    \begin{equation}
    \begin{split}
     \label{eq:doublebracket}
     [[Q]]&=(\pi^\OG)^{-1}(\pi^\OG(Q))\\
     &=\left\lbrace \widetilde{Q} \in \O(n) \ \middle|\ \widetilde{Q}=Q\begin{pmatrix}R_1 & 0\\ 0 & R_2 \end{pmatrix},\ \begin{pmatrix}R_1 & 0\\ 0 & R_2 \end{pmatrix} \in H  \right\rbrace,
    \end{split}
    \end{equation} 
    which corresponds to \cite[Eq. (2.28)]{EdelmanAriasSmith1999}. The manifold structure on $\O(n)/H \cong \O(n)/(\O(p)\times\O(n-p))$ is by definition the unique one that makes the quotient map
    \begin{equation*}
        \O(n) \to \O(n)/(\O(p)\times \O(n-p))\colon Q \mapsto [[Q]]
    \end{equation*} 
    a smooth submersion, i.e., a smooth map with surjective differential at every point. The second part of Proposition~\ref{prop:orbitmanifold} shows that, as $\Gr(n,p)$ is the orbit of $P_0$ under $\Phi$, it holds that 
    \begin{equation*}
        \Gr(n,p) \cong \O(n)/(\O(p) \times \O(n-p)).
    \end{equation*}
    Therefore $\pi^\OG$ is also a smooth submersion.
    Furthermore, we have the well known result
    \begin{equation*}
        \dim(\Gr(n,p))= \dim(\O(n))-\dim(\O(p) \times \O(n-p))=(n-p)p.
    \end{equation*}

\subsection{The Tangent Spaces of the Grassmannian}
    The quotient structure of the Grassmannian allows to split every tangent space of $\O(n)$ into a vertical and (after choosing a Riemannian metric) horizontal part, and to identify every tangent space of $\Gr(n,p)$ with such a horizontal space as in \cite{EdelmanAriasSmith1999}.
    
    As the Lie algebra of $\O(n)$ is the set of skew-symmetric matrices
    \begin{equation*}
        \so(n):=T_I\O(n) = \left\lbrace \Omega \in \R^{n \times n} \ \middle|\ \Omega^T=-\Omega \right\rbrace,
    \end{equation*} 
    the tangent space at an arbitrary $Q \in \O(n)$ is given by the left translates
    \begin{equation*}
        T_Q\O(n)=\left\lbrace Q\Omega \ \middle|\ \Omega \in \so(n) \right\rbrace.
    \end{equation*}
    Restricting the Euclidean matrix space metric $\langle A, B\rangle_0 =\tr(A^TB)$ 
    to the tangent spaces turns the manifold $\O(n)$ into a Riemannian manifold. We include a factor of $\frac{1}{2}$ to obtain Riemannian metrics on the Stiefel and Grassmann manifold, in Subsections \ref{subsec:horizontalliftStiefel} and \ref{subsec:RiemannianMetric}, respectively, that comply with common conventions. This yields the Riemannian metric (termed here \emph{metric} for short)  $g_Q^\O\colon T_Q\O(n) \times T_Q\O(n) \to \R$,
    \begin{equation*}
        g_Q^\O(Q\Omega,Q\widetilde{\Omega}):=\SP{Q\Omega,Q\widetilde{\Omega}}_Q:= \frac{1}{2}\tr\left( (Q\Omega)^TQ\widetilde{\Omega} \right)=\frac{1}{2}\tr\left( \Omega^T\widetilde{\Omega} \right).
    \end{equation*}
    The differential of the projection $\pi^\OG$ at $Q \in \O(n)$ is a linear map $\D\pi^\OG_Q\colon T_Q\O(n) \to T_{\pi^\OG(Q)}\Gr(n,p)$,
    where $T_Q\O(n)$ and $T_{\pi^\OG(Q)}\Gr(n,p)$ are the tangent spaces of $\O(n)$ and $\Gr(n,p)$ at $Q$ and $\pi^\OG(Q)$, respectively.
    The directional derivative of $\pi^\OG$ at $Q \in \O(n)$ in the tangent direction $Q\Omega=Q\begin{psmallmatrix}A & \smash{-B^T}\\ B & C\end{psmallmatrix} \in T_Q\O(n)$ is given by
    \begin{equation}
    \label{eq:dPi}
        \D \pi^\OG_Q(Q\Omega)=\frac{\D}{\D t}\Big\vert_{t=0}(\pi^\OG(\gamma(t))) =
        \frac{\D}{\D t}\Big\vert_{t=0}(\gamma(t)P_0\gamma(t)^T) = Q\begin{pmatrix}0 & B^T\\ B & 0\end{pmatrix}Q^T,
    \end{equation}
    where $\gamma\colon t\mapsto \gamma(t)\in \O(n)$ is an arbitrary differentiable curve with $\gamma(0)=Q$, $\dot \gamma(0)=Q\Omega$. Since $\pi^\OG$ is a submersion, this spans the entire tangent space, i.e.,
    \begin{equation*}
        T_{\pi^\OG(Q)}\Gr(n,p)=\left\lbrace Q\begin{pmatrix} 0 & B^T\\ B & 0\end{pmatrix}Q^T \ \middle|\ B \in \R^{(n-p)\times p} \right\rbrace.
    \end{equation*}
    
    In combination with the metric $g^\O_Q$, the smooth submersion $\pi^\OG$ allows to decompose every tangent space $T_Q\O(n)$ into a vertical and horizontal part, c.f. \cite[Chapter 2]{Lee2018riemannian}. The vertical part is the kernel of the differential $\D\pi^\OG_Q$, and the horizontal part is the orthogonal complement with respect to the metric $g^\O_Q$. We therefore have
    \begin{equation*}
        T_Q\O(n) = \Ver^{\pi^\OG}_Q \O(n) \oplus \Hor^{\pi^\OG}_Q \O(n),
    \end{equation*}
    where
    \begin{equation*}
        \Ver^{\pi^\OG}_Q \O(n) = \left\lbrace Q\begin{pmatrix} A & 0\\ 0 & C\end{pmatrix} \ \middle|\ A \in \so(p),\ C \in \so(n-p) \right\rbrace
    \end{equation*}
    and
    \begin{equation}
    \label{eq:HorQGr}
        \Hor^{\pi^\OG}_Q \O(n) = \left\lbrace Q\begin{pmatrix} 0 & -B^T\\ B & 0\end{pmatrix} \ \middle|\ B \in \R^{(n-p)\times p} \right\rbrace,
    \end{equation}
    c.f. \cite[Eq. (2.29) and (2.30)]{EdelmanAriasSmith1999}. 
    The tangent space of the Grassmann manifold at $P=\pi^\OG(Q)$ can be identified with the horizontal space at any representative $Q \in (\pi^\OG)^{-1}(P)\subset \O(n)$,
    \begin{equation*}
        T_P\Gr(n,p)\cong \Hor^{\pi^\OG}_Q\O(n).
    \end{equation*} 
    
    In \cite{BatziesHueperMachadoLeite2015}, the tangent space $T_P\Gr(n,p)$ is given by matrices of the form $[\Omega,P]$, where $[\cdot,\cdot]$ denotes the matrix commutator, and $\Omega \in \so_P(n)$ fulfilling
    \begin{equation}
        \label{eq:soPn}
         \so_P(n):= \left\lbrace \Omega \in \so(n) \ \middle|\ \Omega=\Omega P+P\Omega \right\rbrace.
    \end{equation}
    Writing $P=Q P_0 Q^T$ and making use of \eqref{eq:dPi} shows that every $\Delta \in T_P\Gr(n,p)$ is of the form
    \begin{equation}
    \label{eq:GrassmannTangentvectorB}
        \Delta = Q\begin{pmatrix}0 & B^T\\ B & 0\end{pmatrix}Q^T = \revcomm{\Bigg[}Q\begin{pmatrix}0 & -B^T \\ B & 0 \end{pmatrix}Q^T,P\revcomm{\Bigg]}.
    \end{equation}
    Since $\Omega \in \so_P(n)$ is equivalent to $Q^T\Omega Q \in \so_{Q^TPQ}(n)$ and $Q^TPQ = P_0$ it follows that every $\Omega \in \so_P(n)$ is of the form
    \begin{equation*}
        \Omega= Q\begin{pmatrix}0 & -B^T \\ B & 0 \end{pmatrix}Q^T.
    \end{equation*}
    Note that for $\Delta \in T_P\Gr(n,p)$, there is $\Omega \in \so_P(n)$ such that $\Delta = [\Omega,P] $. This $\Omega$ can be calculated via $\Omega = [\Delta,P] \in \so_P(n)$.    
    
    \begin{proposition}[Tangent vector characterization]
    \label{prop:TangentVectorCharacterizations}
        Let $P \in \Gr(n,p)$ be the orthogonal projector onto the subspace $\Space{U}$. For every symmetric $\Delta = \Delta^T \in \R^{n \times n}$, the following conditions are equivalent:
        %\begin{multicols}{2}
        \begin{enumerate}[label=\alph*)]
            \item \label{prop:TangentVectorChar_a}$\Delta \in T_P \Gr(n,p)$,
            \item \label{prop:TangentVectorChar_b} $\Delta(\Space{U}) \subset \Space{U}^\perp$ and $\Delta(\Space{U}^\perp)\subset \Space{U}$,
            \item \label{eq:DP+PD=D}$\Delta P + P \Delta = \Delta$,
            \item \label{prop:TangentVectorChar_d} $\Delta = [\Omega,P]$, where $\Omega := [\Delta,P] \in \so_P(n)$.
        \end{enumerate}
        %\end{multicols}
        Here, $\Delta(\Space{U}) :=\{\Delta x \in \R^n \mid x \in \Space{U}\}$ and the orthogonal complement $\Space{U}^\perp$ is taken with respect to the Euclidean metric in $\R^{n}$.
        \begin{proof}
          The equivalence of \ref{prop:TangentVectorChar_a}, \ref{prop:TangentVectorChar_b} and \ref{eq:DP+PD=D} is from \cite[Result 3.7]{MachadoSalavessa1985}. To show \ref{eq:DP+PD=D} implies \ref{prop:TangentVectorChar_d}, note that $\Delta P + P\Delta = \Delta$ implies $P\Delta P=0$ and therefore $[[\Delta,P],P]=\Delta P + P \Delta - 2 P \Delta P = \Delta$. On the other hand, if \ref{prop:TangentVectorChar_d} holds then $\Delta = \Delta P +P\Delta -2P \Delta P$, which also implies $P \Delta P = 0$ by multiplication with $P$ from one side. Inserting $P \Delta P = 0$ into the equation shows that \ref{eq:DP+PD=D} holds. The statement that $\Omega \in \so_P(n)$ is automatically true.
        \end{proof}

    \end{proposition}
    
\subsection{Horizontal Lift to the Stiefel Manifold}
\label{subsec:horizontalliftStiefel}
    The elements of $\Gr(n,p)$ are $n \times n$ matrices (see the bottom level of \autoref{fig:Quotient}). The map $\pi^\OG$ makes it possible to (non uniquely) represent elements of $\Gr(n,p)$ as elements of $\O(n)$---the top level of \autoref{fig:Quotient}---which are also $n \times n$ matrices. In practical computations, however, it is often not feasible to work with $n \times n$ matrices, especially if $n$ is large when compared to the subspace dimension $p$. A remedy is to resort to the middle level of \autoref{fig:Quotient}, namely the Stiefel manifold $\St(n,p)$ \cite{EdelmanAriasSmith1999}. By making use of the map $\pi^\SG$, elements of $\Gr(n,p)$ can be (non uniquely) represented as elements of $\St(n,p)$, which are $n\times p$ matrices. 

    The Stiefel manifold can be obtained analogously to the Grassmann manifold by means of a group action of $\O(n)$ on $\R^{n \times p}$, defined by left multiplication. It is the orbit of
    \begin{equation*}
        I_{n,p}:=\begin{pmatrix}I_p \\ 0\end{pmatrix} \in \R^{n \times p}
    \end{equation*}
    under this group action with stabilizer $\O(n-p)\cong  \left\{\begin{psmallmatrix}I_p & 0\\ 0 & R \end{psmallmatrix}\ \middle|\ R \in \O(n-p)\right\}$. By Proposition~\ref{prop:orbitmanifold}, $\St(n,p) \cong \O(n)/\O(n-p)$ is an embedded submanifold of $\R^{n \times p}$ and the projection from the orthogonal group onto the Stiefel manifold is given by
    \begin{equation*}
        \pi^\OS\colon \O(n) \to \St(n,p),\ Q\mapsto Q I_{n,p},
    \end{equation*}
    the projection onto the first $p$ columns. It defines an equivalence relation on $\O(n)$ by collecting all orthogonal matrices that share the same first $p$ column vectors into an equivalence class. As above, 
    \begin{equation*}
        \dim(\St(n,p))= \dim(\O(n))-\dim(\O(n-p))=np -\frac{1}{2}p(p+1),
    \end{equation*}
    and $\pi^\OS$ is a smooth submersion, which \revcomm{admits a decomposition of} every tangent space $T_Q\O(n)$ into a vertical and horizontal part with respect to the metric $g_Q^\O$. We therefore have
    \begin{equation*}
        T_Q\O(n) = \Ver^{\pi^\OS}_Q \O(n) \oplus \Hor^{\pi^\OS}_Q\O(n),
    \end{equation*}
    where
    \begin{equation*}
        \Ver^{\pi^\OS}_Q \O(n) = \left\lbrace Q\begin{pmatrix} 0 & 0\\ 0 & C\end{pmatrix} \ \middle|\ C \in \so(n-p) \right\rbrace
    \end{equation*}
    and
    \begin{equation*}
        \Hor^{\pi^\OS}_Q \O(n) = \left\lbrace Q\begin{pmatrix} A & -B^T\\ B & 0\end{pmatrix} \ \middle|\ A \in \so(p),\ B \in \R^{(n-p)\times p} \right\rbrace.
    \end{equation*}
    By the identification 
    \begin{equation*}
        T_U\St(n,p) \cong \Hor^{\pi^\OS}_Q \O(n),
    \end{equation*} 
    see \cite{EdelmanAriasSmith1999}, and orthogonal completion $U_\perp \in \R^{n\times (n-p)}$  of $U$, i.e., such that $\begin{pmatrix} U & U_\perp \end{pmatrix}\in \O(n)$, the tangent spaces of the Stiefel manifold are explicitly given by either of the following expressions
    \begin{equation}
    \label{eq:TangSt}
     \begin{split}
        T_U\St(n,p) &= \left\lbrace UA+U_\perp B \in \R^{n \times p}\ \middle|\ A \in \so(p),\ B \in \R^{(n-p)\times p} \right\rbrace\\
        &=\left\lbrace UA + (I_n-UU^T)T \ \middle|\ A \in \so(p),\ T \in \R^{n \times p} \right\rbrace\\
        &=\left\lbrace\Omega U \ \middle|\ \Omega \in \so(n)\right\rbrace \\
        &= \left\lbrace D \in \R^{n \times p} \ \middle|\ U^TD=-D^T U \right\rbrace.
     \end{split}
    \end{equation}
    Note that $U^TU_\perp=0$ and $U_\perp^T U_\perp=I_{n-p}$, as well as $  I_n=UU^T+U_\perp U_\perp^T$.
    
    The \emph{canonical metric} $g_U^\St(\cdot,\cdot)$ on the Stiefel manifold is given via the horizontal lift. That means that for any two tangent vectors in $D_1=UA_1+U_\perp B_1,\ D_2=UA_2+U_\perp B_2 \in T_U\St(n,p)$, we take a total space representative $Q \in \O(n)$ of $U \in \St(n,p)$ and 
    `lift' the tangent vectors $D_1, D_2 \in T_U\St(n,p)$ to tangent vectors 
    $D_{1,Q}^\hor, D_{2,Q}^\hor \in \Hor^{\pi^\OS}_Q\O(n)\subset T_Q\O(n)$, defined by 
    $\D(\pi^\OS)_Q(D_{i,Q}^\hor)=D_i$, $i=1,2$.
    The inner product between $D_{1,Q}^\hor, D_{2,Q}^\hor$
    is now computed according to the metric of $\O(n)$. In practice, this leads to
    \begin{equation*}
    \begin{split}
            g_U^\St(D_1,D_2)&:=g_Q^\O(D_{1,Q}^\hor,D_{2,Q}^\hor)
            =\frac{1}{2}\tr(A_1^TA_2)+\tr(B^T_1B_2)\\
            &=\tr\left(D_1^T(I_n-\frac{1}{2}UU^T)D_2\right),
    \end{split}
    \end{equation*}
    c.f. \cite{EdelmanAriasSmith1999}. The last equality shows that it does not matter which base point $Q \in (\pi^\OS)^{-1}(U)$ is chosen for the lift.
    
    In order to make the transition from column-orthogonal matrices $U$ to the associated subspaces
    $\Space{U} =\Span(U)$, another equivalence relation, this time on the Stiefel manifold, is required:
    Identify any matrices $U \in \St(n,p)$, whose column vectors span the same subspace $\Space{U}$. For any two Stiefel matrices $U,\tilde{U}$ that span the same subspace, it holds that $\tilde{U}=UU^T\tilde{U}$. As a consequence, $I_p = (\tilde{U}^TU)(U^T\tilde{U})$, so that 
    $R = (U^T\tilde{U})\in \O(p)$.
    Hence, any two such Stiefel matrices differ by a rotation/reflection $R \in \O(p)$. Define a smooth right action of $\O(p)$ on $\St(n,p)$ by multiplication from the right. Every equivalence class
    \begin{equation}
    \label{eq:singlebracket}
        \Space{U}\cong [U]:=\left\lbrace \widetilde{U} \in \St(n,p) \ \middle|\ \widetilde{U}=UR,\ R \in \O(p) \right\rbrace
    \end{equation}
    under this group action can be identified with a projector $UU^T$ and vice versa. Therefore, according to \cite[Thm 21.10, p. 544]{Lee2012smooth}, the set of equivalence classes $[U]$, denoted by $\St(n,p)/\O(p)$, is a smooth manifold with a manifold structure for which the quotient map is a smooth submersion. To show that the manifold structure is indeed the same as the one on $\Gr(n,p)$ (which we can identify as a set with $\St(n,p)/\O(p)$), we show directly that the projection from $\St(n,p)$ to $\Gr(n,p)$,
    \begin{equation}
        \label{eq:piSG}
        \pi^\SG\colon \St(n,p) \to \Gr(n,p),\ U \mapsto UU^T,
    \end{equation}
    is a smooth submersion. Indeed, the derivative $\D(\pi^\SG)_U(D)= DU^T+UD^T$ is surjective, since every tangent vector $\Delta \in T_{\pi^\OG(Q)}\Gr(n,p)$ can be written as 
    \begin{equation}
        \label{eq:GrassmannTangentvectorUBUT}
        \Delta = U_\perp B U^T + U B^T U_\perp^T,
    \end{equation}
    by making use of \eqref{eq:GrassmannTangentvectorB}. This shows surjectivity, since for every $\Delta \in T_{\pi^\OG(Q)}\Gr(n,p)$ we can choose $U_\perp B \in T_U\St(n,p)$, such that $\D (\pi^\SG)_U(U_\perp B) = \Delta$.
    
    Again, we split every tangent space $T_U\St(n,p)$ with respect to the projection $\pi^\SG$ and the metric $g^\St_U(\cdot,\cdot)$ on the Stiefel manifold.
    Defining the kernel of $\D (\pi^\SG)_U$ as the vertical space and its orthogonal complement (with respect to the metric $g^\St_U$) as the horizontal space leads to the direct sum decomposition
    \begin{equation*}
        T_U\St(n,p) = \Ver_U \St(n,p) \oplus \Hor_U \St(n,p),
    \end{equation*}
    where
    \begin{equation*}
        \Ver_U \St(n,p)= \ker \D (\pi^\SG)_U = \left\lbrace UA \ \middle|\ A \in \so(p) \right\rbrace
    \end{equation*}
    and
    \begin{equation}
    \label{eq:HorUSt}
        \begin{split}
        \Hor_U\St(n,p)&= (\ker \D (\pi^\SG)_U)^{\perp} =  \left\lbrace U_\perp B \ \middle|\ B \in \R^{(n-p) \times p} \right\rbrace =\left\lbrace (I_n-UU^T)T \ \middle|\ T \in \R^{n \times p} \right\rbrace\\
        &= \left\lbrace D \in \R^{n \times p} \ \middle|\ U^TD=0 \right\rbrace.
     \end{split}
    \end{equation}
  Since $\pi^\SG$ is the only projection that we use on the Stiefel manifold, the dependence of the splitting on the projection is omitted in the notation.
  
  The tangent space $T_P \Gr(n,p)$ of the Grassmannian can be identified with the horizontal space $\Hor_U\St(n,p)$. Therefore, for every tangent vector $\Delta \in T_P\Gr(n,p)$, there is a unique $\Delta^\hor_U \in \Hor_U\St(n,p)$, called the \emph{horizontal lift of $\Delta$ to $U$}. By (\ref{eq:HorUSt}), there are matrices $T \in \R^{n \times p}$ and $B \in \R^{(n-p) \times p}$ such that
    \begin{equation*}
        \Delta^\hor_U = U_\perp B = (I_n-UU^T)T \in \Hor_U\St(n,p).
    \end{equation*}     
    Note that $\Delta^\hor_U$ depends only on the chosen representative $U$ of $P$, while $B$ depends on the chosen orthogonal completion $U_\perp$ as well.
    
    Multiplication of \eqref{eq:GrassmannTangentvectorUBUT} from the right with $U$ shows that the horizontal lift of $\Delta \in T_P\Gr(n,p)$ to $U \in \St(n,p)$ can be calculated by
    \begin{equation}
    \label{eq:HorLiftToStiefel}
        \Delta^\hor_U=\Delta U.
    \end{equation}
    Therefore, the horizontal lifts of $\Delta$ to two different representatives $U$ and $UR$ are connected by
    \begin{equation}
    \label{eq:LiftDifferentBase}
        \Delta^\hor_{UR}=\Delta^\hor_U R,
    \end{equation}
    which relates to \cite[Prop. 3.6.1]{AbsilMahonySepulchre2008}. The lift of $\Delta \in T_P\Gr(n,p)$ to $Q=\begin{pmatrix} U & U_\perp \end{pmatrix} \in \O(n)$ can also be calculated explicitly. By \eqref{eq:dPi}, \eqref{eq:HorQGr} and \eqref{eq:GrassmannTangentvectorB}, it is given by
    \begin{equation*}
        \Delta^\hor_Q=[\Delta,P]Q=Q\begin{pmatrix}0 & -B^T \\ B & 0 \end{pmatrix} \in \Hor^{\pi^\OG}_Q\O(n).
    \end{equation*}
    
    In conclusion, the Grassmann manifold is placed at the end of the following quotient space hierarchy with equivalence classes $[\cdot]$ from \eqref{eq:singlebracket} and $[[\cdot]]$ from \eqref{eq:doublebracket}:
    \begin{eqnarray*}
     \label{eq:GrassmQuotient}
         \Gr(n,p)        
         \cong&\ \St(n,p)/\O(p)
         &=\left\lbrace [U] \ \middle|\ U \in \St(n,p)\right\rbrace\\         
         \cong&\ \O(n)/(\O(p)\times \O(n-p))\ 
         &= \left\lbrace [[Q]] \ \middle|\ Q \in \O(n)\right\rbrace.
    \end{eqnarray*}
    
    \begin{remark}
     It should be noted that there is yet another way of viewing the Grassmann manifold as a quotient. Instead of taking equivalence classes in $\O(n)$, one can take the quotient of the noncompact Stiefel manifold by the general linear group $\GL(p)$. This introduces a factor of the form $(Y^TY)^{-1}$ into many formulae, where $Y \in \R^{n \times p}$ is a rank $p$ matrix with (not necessarily orthogonal) column vectors spanning the desired subspace. For this approach see for example \cite{AbsilMahonySepulchre2004}.
    \end{remark}

\section{Riemannian Structure}
\label{sec:RiemannStruct}
    In this section, we study the basic Riemannian structure of the Grassmannian. We introduce the canonical metric coming from the quotient structure---which coincides with the Euclidean metric---and the Riemannian connection. The Riemannian exponential mapping for geodesics is derived in the formulation as projectors as well as with Stiefel representatives. Lastly, we study the concept of parallel transport on the Grassmannian. Many of those results have been studied before for the projector or the ONB perspective. For the metric and the exponential see for example \cite{EdelmanAriasSmith1999,AbsilMahonySepulchre2004} (Stiefel perspective) and \cite{BatziesHueperMachadoLeite2015} (projector perspective). For the horizontal lift of the Riemannian connection see \cite{AbsilMahonySepulchre2004}. A formula for parallel transport in the ONB perspective was given in \cite{EdelmanAriasSmith1999}. 
    Here we combine the approaches and provide some modifications and additions. 
    We derive formulae for all mentioned concepts in both perspectives and also study the derivative of the exponential mapping.

\subsection{Riemannian Metric}
\label{subsec:RiemannianMetric}
    The Riemannian metric on the Grassmann manifold that is induced by  the quotient structure coincides with (one half times) the Euclidean metric. 
    To see this, let $\Delta_1, \Delta_2 \in T_{P}\Gr(n,p)$ be two tangent vectors at $P \in \Gr(n,p)$ and let $Q=\begin{pmatrix} U & U_\perp \end{pmatrix} \in \O(n)$ such that $\pi^\OG(Q)=P$. The metric on the Grassmann manifold is then inherited from the metric on $\O(n)$ applied to the horizontal lifts, i.e.
    \begin{equation}
    \label{eq:Grassmann_metric_On}
        g_P^\Gr(\Delta_1,\Delta_2):= g_Q^\O(\Delta^\hor_{1,Q},\Delta^\hor_{2,Q}).
    \end{equation}
    Let $\Delta_i=[\Omega_i,P]$, where $\Omega_i \in \so_P(n)$, as well as $\Delta^\hor_{i,Q}=Q\begin{psmallmatrix}0 & \smash{-B_i^T} \\ B_i & 0 \end{psmallmatrix}$ and $\Delta^\hor_{i,U}=U_\perp B_i$. We immediately see that
    \begin{equation}
    \label{eq:GrassmannMetricAlternatives}
     \begin{split}
        g_P^\Gr(\Delta_1,\Delta_2)&= \frac{1}{2}\tr\left((\Delta^\hor_{1,Q})^T\Delta^\hor_{2,Q}\right)
        = \tr\left(\Delta^{\hor^T}_{1,U}\Delta^\hor_{2,U}\right)=\tr(U^T\Delta_1\Delta_2 U)\\
        &= \tr(B_1^T B_2)
        = \frac{1}{2}\tr(\Delta_1\Delta_2)
        =\frac{1}{2}\tr(\Omega_1^T\Omega_2).
     \end{split}
    \end{equation} 
    The last equality can be seen by noticing $[\Omega_i,P]=(I_n-2P)\Omega_i$ for $\Omega_i \in \so_P(n)$ and $(I_n-2P)^2=I_n$. Although the formulae in (\ref{eq:GrassmannMetricAlternatives}) all look similar, notice that $\Delta_i,\Omega_i,\Delta^\hor_{i,Q} \in \R^{n \times n}$, but $\Delta^\hor_{i,U} \in \R^{n \times p}$ and $B_i \in \R^{(n-p) \times p}$.
    
    The metric does not depend on the point to which we lift: Lifting to a different $UR \in \St(n,p)$ results in a postmultiplication of $\Delta^\hor_{i,U}$ with $R$ according to (\ref{eq:LiftDifferentBase}). By the invariance properties of the trace, this does not change the metric. An analogous argument holds for the lift to $\O(n)$.
    
    With the Riemannian metric we can define the induced norm of a tangent vector $\Delta \in T_P\Gr(n,p)$ by
    \begin{equation*}
        \norm{\Delta}:=\sqrt{g_P^\Gr(\Delta,\Delta)}=\frac{1}{\sqrt{2}}\sqrt{\tr(\Delta^2)}.
    \end{equation*} 
    
\subsection{Riemannian Connection}
\label{subsec:RiemmanianConnection}
    The disjoint collection of all tangent spaces of a manifold $M$ is called the \emph{tangent bundle} $TM = \dot\cup_{p\in M} T_pM$\revcomm{, which is itself a smooth manifold, c.f. \cite[Proposition 3.18]{Lee2012smooth}}.
    A \emph{smooth vector field} on $M$ is a smooth map $X$ from $M$ to $TM$ that maps a point $p \in M$ to a tangent vector $X(p) \in T_pM$.
    The set of all smooth vector fields on $M$ is denoted by $\Vectorfields(M)$. 
    Plugging smooth vector fields $Y,Z \in \Vectorfields(M)$ into the metric of a Riemannian manifold gives a smooth function $g(Y,Z)\colon M\to \R$.
    It is not possible to calculate the differential of a vector field in the classical sense, since every tangent space is a separate vector space and the addition of $X(p) \in T_pM$ and $X(q) \in T_q M$ is not defined for $p\neq q$. 
    To this end, the abstract machinery of differential geometry provides special tools called \emph{connections}. A connection acts as the derivative of a vector field in the direction of another vector field. On a Riemannian manifold $(M,g)$, the \emph{Riemannian} or \emph{Levi-Civita connection} is the unique connection $\nabla \colon \Vectorfields(M) \times \Vectorfields(M) \to \Vectorfields(M): (X,Y)\mapsto \nabla_X Y$ that is 
    \begin{itemize}
    \item   {\em compatible with the metric}:
    for all vector fields $X,Y,Z \in \Vectorfields(M)$, we have the product rule
    \begin{equation*}
        \nabla_X g(Y,Z)=g(\nabla_X Y,Z)+g(Y,\nabla_X Z).
    \end{equation*}
    \item  {\em torsion free}: for all $X,Y \in \Vectorfields(M)$,
    $
        \nabla_XY-\nabla_YX=[X,Y],
    $\\
    where $[X,Y] = X(Y) - Y(X)$ denotes the Lie bracket of two vector fields.
    \end{itemize}
    The Riemannian connection can be explicitly calculated in the case of embedded submanifolds: It is the projection of the Levi-Civita connection of the ambient manifold onto the tangent space of the embedded submanifold. For details see for example \cite[\revcomm{Chapter 5 \& Chapter 8, Proposition 8.6}]{Lee2018riemannian}.
    
    The Euclidean space $\R^{n \times p}$ is a vector space, which implies that every tangent space of $\R^{n \times p}$ can be identified with $\R^{n \times p}$ itself. Therefore, the Riemannian connection of the Euclidean space $\R^{n\times p}$ with the Euclidean metric (\ref{eq:EuclideanMetric}) is the usual directional derivative: Let $F \colon \R^{n \times p} \to \R^{n \times p}$ and $X,Y \in \R^{n \times p}$. The directional derivative of $F$ at $X$ in direction $Y$ is then
    \begin{equation*}
        \D F_X(Y)=\frac{\D}{\D t}\Big\vert_{t=0}F(X+tY).
    \end{equation*}
    The same holds for the space of symmetric matrices $\Sym_n$. When considered as the set of orthogonal projectors, the Grassmann manifold $\Gr(n,p)$ is an embedded submanifold of $\Sym_n$. In this case, the projection onto the tangent space is  
    \begin{equation}
    \label{eq:projSymTPGr}
        \Pi_{T_P\Gr} \colon \Sym_n \to T_P\Gr(n,p), \quad
        S \mapsto (I_n-P)SP + PS(I_n-P),
    \end{equation}
    see also \cite{MachadoSalavessa1985}. 
    In order to restrict calculations to $n\times p$ matrices, 
    we can lift to the Stiefel manifold and use  the projection onto the horizontal space, which is
    \begin{equation}
    \label{eq:projRHorUSt}
        \Pi_{\Hor_U\St} \colon \R^{n \times p} \to \Hor_U\St(n,p), \quad
        Z \mapsto (I_n-UU^T)Z,
    \end{equation}
    see also  \cite{EdelmanAriasSmith1999,AbsilMahonySepulchre2004}. Note that $\Hor_U\St(n,p) \cong T_{\pi^\SG(U)}\Gr(n,p)$ as described in Subsection \ref{subsec:horizontalliftStiefel}. The Riemannian connection on $\Gr(n,p)$ is now obtained via the following proposition.
    \begin{proposition}[Riemannian Connection]
            Let $X \in \Vectorfields(\Gr(n,p))$ be a smooth vector field on $\Gr(n,p)$, i.e., $X(P) \in T_P\Gr(n,p)$, with a smooth extension to an open set in the symmetric $n \times n$ matrices, again denoted by $X$. Let $Y \in T_P\Gr(n,p)$. The Riemannian connection on $\Gr(n,p)$ is then given by
            \begin{equation}
                \label{eq:connectionGr}
                \nabla_Y(X)=\Pi_{T_P\Gr}(\D X_P(Y))=\Pi_{T_P\Gr}\left(\frac{\D}{\D t}\Big\vert_{t=0} X(P + tY)\right).
            \end{equation}
            It can also be calculated via the horizontal lift,
            \begin{equation}
                \label{eq:connectionGrlifted}
                (\nabla_Y(X))^\hor_U=\Pi_{\Hor_U\St}(\D (U \mapsto X^\hor_U)_U(Y^\hor_U))
                =(I_n -UU^T)\frac{\D}{\D t}\Big\vert_{t=0} X^\hor_{U+tY^\hor_U}.
            \end{equation}
            Here, $\R^{n \times p} \ni U \mapsto X^\hor_U \in \R^{n \times p}$ is to be understood as a smooth extension to an open subset of $\R^{n \times p}$ of the actual horizontal lift $U \mapsto X^\hor_U:=(X(UU^T))^\hor_U$ from \eqref{eq:HorLiftToStiefel}, i.e., fulfilling $\D (\pi^\SG)_UX^\hor_U=X(UU^T)$, where $X$ is the vector field $P \mapsto X(P)$.
            \begin{proof}
                Equation \eqref{eq:connectionGr} follows directly from the preceding discussion. It can be checked that \eqref{eq:connectionGrlifted} is the horizontal lift of \eqref{eq:connectionGr}. Alternatively, \eqref{eq:connectionGrlifted} can be deduced from \cite[Lemma 7.45]{ONeill1983} by noticing that the horizontal space of the Stiefel manifold is the same for the Euclidean and the canonical metric. Furthermore, \eqref{eq:connectionGrlifted} coincides with \cite[Theorem 3.4]{AbsilMahonySepulchre2004}.
            \end{proof}
    \end{proposition}
    
\subsection{Gradient}
    The gradient of a real-valued function on the Grassmannian for the canonical metric was computed in \cite{EdelmanAriasSmith1999} for the Grassmannian with Stiefel representatives, in \cite{HelmkeHueperTrumpf2007} for the projector perspective and in \cite{AbsilMahonySepulchre2004} for the Grassmannian as a quotient of the noncompact Stiefel manifold. For the sake of completeness, we introduce it here as well.
    The gradient is dual to the differential of a function in the following sense:
    For a function $f \colon \Gr(n,p) \to \R$, the gradient at $P$ is defined as the unique tangent vector $(\grad f)_P \in T_P\Gr(n,p)$ fulfilling
    \begin{equation*}
     \D f_P(\Delta)= g_P^\Gr((\grad f)_P, \Delta)
    \end{equation*}
    for all $\Delta \in T_P\Gr(n,p)$, where $\D f_P$ denotes the differential of $f$ at $P$. 

    It is well known that the gradient for the induced Euclidean metric on a manifold is the projection of the Euclidean gradient $\grad^{\eucl}$ to the tangent space. For the Euclidean gradient to be well-defined, $f$ is to be understood as a smooth extension of the actual function $f$ to an open subset of $\Sym_n$. Therefore
    \begin{equation*}
        (\grad f)_P = \Pi_{T_P\Gr}((\grad^{\eucl} f)_P).
    \end{equation*}
    The function $f$ on $\Gr(n,p)$ can be lifted to the function $\bar{f}:=f \circ \pi^\SG$ on the Stiefel manifold. Again, when necessary, we identify $\bar f$ with a suitable differentiable extension. These two functions are linked by
    \begin{equation*}
        ((\grad f)_P)^\hor_U=(\grad \bar{f})_U=\Pi_{T_U\St}((\grad^{\eucl} \bar{f})_U)=\Pi_{\Hor_U\St}((\grad^{\eucl} \bar{f})_U),
    \end{equation*}
    where $\Pi_{T_U\St}(X)=X-\frac{1}{2}U(X^TU+U^TX)$ is the projection of $X \in \R^{n \times p}$ to $T_U\St(n,p)$. The first equality is~\cite[Equation (3.39)]{AbsilMahonySepulchre2004}, while the second equality uses the same argument as above. The last equality is due to the fact that the gradient of $\bar{f}$ has no vertical component.
    For further details see \cite{EdelmanAriasSmith1999,HelmkeHueperTrumpf2007,AbsilMahonySepulchre2004}.

\subsection{Exponential Map}
\label{subsec:ExponentialMap}
    The \emph{exponential map} $\exp_p \colon T_pM \to M$ on a Riemannian manifold $M$ maps a tangent vector $\Delta \in T_pM$ to 
    the endpoint $\gamma(1) \in M$ of the unique geodesic $\gamma$ that emanates from $p$ in the direction $\Delta$.
    Thus, geodesics and the Riemannian exponential are related by $\gamma(t)=\exp_p(t\Delta)$. 
    Under a Riemannian submersion $\pi \colon M \to N$, geodesics with horizontal tangent vectors in $M$ are mapped to geodesics in $N$, cf. \cite[Corollary 7.46]{ONeill1983}. Since the projection $\pi^\OG \colon \O(n) \to \Gr(n,p)$ defined in (\ref{eq:projectionOGr}) is a Riemannian submersion by construction, this observation may be used to obtain the Grassmann geodesics.
    
    We start with the geodesics of the orthogonal group.
    For any Lie group with bi-invariant metric, the geodesics are the one-parameter subgroups, \cite[\S 2]{Alexandrino2015}.
    Therefore, the geodesic from $Q \in \O(n)$ in direction $Q\Omega \in T_Q\O(n)$ is calculated via
    \begin{equation*}
        \Exp_Q^\O(tQ\Omega)=Q\expm(t\Omega),
    \end{equation*}
    where $\expm$ denotes the matrix exponential, see \eqref{eq:MatrxiExpLog}. If $\pi^\OG(Q)=P \in \Gr(n,p)$ and $\Delta \in T_P\Gr(n,p)$ with $\Delta^\hor_Q=Q\begin{psmallmatrix} 0 & -\smash{B^T}\\ B & 0\end{psmallmatrix} \in \Hor^{\pi^\OG}_Q \O(n)$, the geodesic in the Grassmannian is therefore
    \begin{equation}
    \label{eq:GrassmannExpO}
        \Exp_P^\Gr(t\Delta)=\pi^\OG\left(Q\expm \left(t\begin{pmatrix} 0 & -B^T\\ B & 0\end{pmatrix}\right)\right).
    \end{equation}
    This formula, while simple, is not useful for applications with large $n$, since it involves the matrix exponential of an $n \times n$ matrix. Evaluating the projection $\pi^\OG$ leads to the geodesic formula from \cite{BatziesHueperMachadoLeite2015}:
    \begin{proposition}[Grassmann Exponential:  Projector Perspective]
        Let $P \in \Gr(n,p)$ be a point in the Grassmannian and $\Delta \in T_P\Gr(n,p)$. The exponential map is given by
        \begin{equation*}
            \Exp_P^\Gr(\Delta)=\expm([\Delta,P])P\expm(-[\Delta,P]).
        \end{equation*}
        \begin{proof}
          With $\Omega = [\Delta,P] = Q\widetilde{\Omega}Q^T \in \so_P(n)$ and $\widetilde{\Omega} = \begin{psmallmatrix} 0 & -\smash{B^T}\\ B & 0\end{psmallmatrix}$,
          the horizontal lift of the tangent vector $\Delta=[\Omega,P] \in T_P\Gr(n,p)$
            is given by $\Omega Q \in \Hor^{\pi^\OG}_Q\O(n)$, see \eqref{eq:soPn}. Then
            \begin{align*}
                \Exp_P^\Gr([\Omega,P])&= Q \expm(\widetilde{\Omega}) I_{n,p} I_{n,p}^T \expm(\widetilde{\Omega}^T)Q^T\\
                &=\expm(Q\widetilde{\Omega}Q^T)QI_{n,p} I_{n,p}^T Q^T \expm(Q\widetilde{\Omega}^TQ^T)=\expm(\Omega) P \expm(\Omega^T).
            \end{align*} 
        \end{proof}
    \end{proposition}
    
    If $n\gg p$, then working with Stiefel representatives reduces the computational effort immensely.
    The corresponding  geodesic formula appears in \cite{AbsilMahonySepulchre2004,EdelmanAriasSmith1999}
    and is restated in the following proposition. 
    The bracket $[\cdot]$ denotes the equivalence classes from (\ref{eq:singlebracket}).

    \begin{proposition}[Grassmann Exponential: ONB Perspective]
    \label{prop:GrassExp_Stiefel}
        For a point $P=UU^T \in \Gr(n,p)$ and a tangent vector  $\Delta \in T_P\Gr(n,p)$, let $\Delta^\hor_U \in \Hor_U\St(n,p)$ be the horizontal lift of $\Delta$ to $\Hor_U\St(n,p)$. Let $r\leq \min(p,n-p)$ be the number of non-zero singular values of $\Delta^\hor_U$. Denote the thin singular value decomposition (SVD) of $\Delta^\hor_U$ by
        \begin{equation*}
            \Delta^\hor_U=\hat{Q}\Sigma V^T,
        \end{equation*}
        i.e., $\hat{Q} \in \St(n,r), \Sigma=\diag(\sigma_1, \dots, \sigma_r)$ and $V \in \St(p,r)$. The Grassmann exponential for the geodesic from $P$ in direction $\Delta$ is given by
        \begin{equation}
        \label{eq:GrassmannExpStthin}
        \begin{split}
            \Exp_P^\Gr(t\Delta)&=[UV\cos(t\Sigma)V^T + \hat{Q}\sin(t\Sigma)V^T + UV_\perp V_\perp^T]\\
            &=[\begin{pmatrix}[c|c] UV\cos(t\Sigma) + \hat{Q}\sin(t\Sigma) & UV_\perp\end{pmatrix}],
        \end{split}
        \end{equation}
        which does not depend on the chosen orthogonal completion $V_\perp$.
    \begin{proof}
    This is essentially \cite[Theorem 2.3]{EdelmanAriasSmith1999} with a reduced storage requirement for $\hat{Q}$ in case of rank-deficient tangent velocity vectors.
        The thin SVD of $B$ is given by
        \begin{equation*}
            B=U_\perp^T\Delta^\hor_U=U_\perp^T\hat{Q}\Sigma V^T
        \end{equation*}
        with $W:=U_\perp^T\hat{Q} \in \St(n-p,r)$, $\Sigma\in\R^{r\times r}$, $V\in St(p,r)$. Let $W_\bot, V_\bot$ be suitable orthogonal completions. Then,
        \begin{equation*}
            \expm \begin{psmallmatrix} 0 & \smash{-B^T}\\ B & 0\end{psmallmatrix}=\begin{psmallmatrix} V & V_\perp & 0 & 0\\ 0 & 0 & W & W_\perp\end{psmallmatrix}\begin{psmallmatrix} \cos(\Sigma) & 0 & -\sin(\Sigma) & 0\\ 0 & I_{p-r} & 0 & 0\\ \sin(\Sigma) & 0 & \cos(\Sigma) & 0 \\ 0 & 0 & 0 & I_{n-p-r}\end{psmallmatrix} \begin{psmallmatrix} V^T & 0 \\ V_\perp^T & 0 \\ 0 & W^T\\ 0 & W_\perp^T\end{psmallmatrix},
        \end{equation*}
        which leads to the desired result when inserted into (\ref{eq:GrassmannExpO}). The second equality in \eqref{eq:GrassmannExpStthin} is given by a postmultiplication by $\begin{pmatrix}V & V_\perp\end{pmatrix} \in \O(p)$, which does not change the equivalence class. This postmultiplication does however change the Stiefel representative, so $\begin{pmatrix}[c|c] UV\cos(t\Sigma)+ \hat{Q}\sin(t\Sigma) & UV_\perp\end{pmatrix}$ is the Stiefel geodesic from $\begin{pmatrix} UV & UV_\perp\end{pmatrix}$ in direction $\begin{pmatrix}\hat{Q}\Sigma & 0 \end{pmatrix}$. A different orthogonal completion of $V$ does not change the second expression in~\eqref{eq:GrassmannExpStthin} and results in a different representative of the same equivalence class in the third expression.
    \end{proof}
    \end{proposition}
    The formula established in \cite{EdelmanAriasSmith1999} uses the compact SVD $\Delta^\hor_U=\tilde{Q}\tilde{\Sigma}\tilde{V}^T$ with $\tilde{Q} \in \St(n,p),\ \tilde{\Sigma}=\diag(\sigma_1,...,\sigma_p)$ and $\tilde{V} \in \O(p)$. Then
    \begin{equation}
        \label{eq:GrassmannExpSt}
        \Exp_P^\Gr(t\Delta)=[U\tilde{V}\cos(t\tilde{\Sigma})\tilde{V}^T + \tilde{Q}\sin(t\tilde{\Sigma})\tilde{V}^T].
    \end{equation}
    By a slight abuse of notation we also define
    \begin{equation}
        \label{eq:GrassmannExpStRep}
        \Exp_U^\Gr(t\Delta^\hor_U)=U\tilde{V}\cos(t\tilde{\Sigma})\tilde{V}^T + \tilde{Q}\sin(t\tilde{\Sigma})\tilde{V}^T
    \end{equation}
    to be the Grassmann exponential on the level of Stiefel representatives.
    
\subsection{Differentiating the Grassmann Exponential}
\label{sec:DifferentialOfExp}
    In this section, we compute explicit expressions for the differential $\D(\Exp_P^{\Gr})_\Delta$ of the Grassmann exponential
    at a tangent location $\Delta\in T_P\Gr(n,p)$. One possible motivation is the computation of Jacobi fields vanishing at a point in Subsection \ref{subsec:jacobifields}. Another motivation is, e.g., Hermite manifold interpolation as in \cite{ZimmermannHermite_2020}.

    Formally, the differential at $\Delta$ is the linear map
    \begin{equation}
    \label{eq:diffmap_Grassmann}
    \D(\Exp_P^{\Gr})_{\Delta}\colon T_\Delta(T_P\Gr(n,p)) \to T_{\Exp_P^{\Gr}(\Delta)}\Gr(n,p).
    \end{equation}
    The tangent space to a linear space can be identified with the linear space itself, so that
    $T_\Delta(T_P\Gr(n,p)) \cong T_P\Gr(n,p)$. We also exploit this principle in practical computations.
    We consider the exponential in the form of \eqref{eq:GrassmannExpSt}.
    The task boils down to computing the directional derivatives
    \begin{equation}
    \label{eq:diffGrassmann}
    \D(\Exp_P^{\Gr})_{\Delta}(\tilde{\Delta}) =
    \frac{\D}{\D t}\Big\vert_{t=0} \Exp_P^{\Gr}(\Delta + t\tilde{\Delta}),  %\in T_{\Exp_P^{\Gr}(\Delta)}\Gr(n,p)
    \end{equation}
    where $\Delta, \tilde{\Delta}\in T_P\Gr(n,p)$. A classical result in Riemannian geometry~\cite[Prop. 5.19]{Lee2018riemannian} ensures that for $\Delta = 0 \in T_P\Gr(n,p)$ the derivative is the identity
    $\D(\Exp_P^{\Gr})_{0}(\tilde{\Delta}) = \tilde{\Delta}$.
    For $\Delta \neq 0$, we can proceed as follows:

    \begin{proposition}[Derivative of the Grassmann Exponential]
        \label{prop:derivative_Grassmannexp}
        Let $P=UU^T \in \Gr(n,p)$ and $\Delta, \tilde\Delta \in T_P\Gr(n,p)$ such that $\Delta^\hor_U$ has mutually distinct, non-zero singular values. Furthermore let $\Delta^\hor_U=Q\Sigma V^T$ and $(\Delta+t\tilde\Delta)^\hor_U=Q(t)\Sigma(t)V(t)^T$ be the compact SVDs of the horizontal lifts of $\Delta$ and $\Delta + t\tilde \Delta$, respectively. Denote the derivative of $Q(t)$ evaluated at $t=0$ by $\dot Q = \frac{\D}{\D t}\big\vert_{t=0}Q(t)$ and likewise for $\Sigma(t)$ and $V(t)$.\footnote{The matrices $\dot Q, \dot \Sigma$ and $\dot V$ can be calculated via Algorithm~\ref{alg:SVDdiff}.} Let 
        \begin{equation*}
            Y:= UV\cos(\Sigma) + Q \sin(\Sigma)\in \St(n,p)
        \end{equation*}
        and 
        \begin{equation*}
            \Gamma:=U\dot V \cos(\Sigma) - U V \sin(\Sigma) \dot\Sigma + \dot Q \sin(\Sigma) + Q \cos(\Sigma)\dot \Sigma \in T_Y\St(n,p).
        \end{equation*}
        Then the derivative of the Grassmann exponential is given by
        \begin{equation}
        \label{eq:GrassExpDiffStiefel}
            \D(\Exp_P^{\Gr})_{\Delta}(\tilde{\Delta}) = \Gamma Y^T + Y \Gamma^T \in T_{\Exp^\Gr_P(\Delta)}\Gr(n,p) \subseteq \R^{n \times n}.
         \end{equation}
        The horizontal lift to $Y$ is accordingly 
        \begin{equation}
        \label{eq:derivative_Grassmannexp}
            \left( \D(\Exp_P^{\Gr})_{\Delta}(\tilde{\Delta}) \right)^\hor_{Y} = (I_n-YY^T)\Gamma = \Gamma + Y \Gamma^T Y  \in \R^{n \times p}.
        \end{equation}
        \end{proposition}
        \begin{proof}
            The curve $\gamma(t):=\Exp^\Gr_P(\Delta+t\widetilde{\Delta})$ on the Grassmannian is given by
            \begin{equation*}
                \gamma(t)=\pi^\SG\left( U V(t) \cos(\Sigma(t)) V(t)^T + Q(t) \sin(\Sigma(t))V(t)^T \right),
            \end{equation*}
            according to (\ref{eq:GrassmannExpSt}). Note that this is in general not a geodesic in $\Gr(n,p)$ but merely a curve through the endpoints of the geodesics from $P$ in direction $\Delta + t \widetilde{\Delta}$. That is to say, it is the mapping of the (non-radial) straight line $\Delta + t \widetilde{\Delta}$ in $T_P\Gr(n,p)$ to $\Gr(n,p)$ via the exponential map. The projection $\pi^\SG$ is not affected by the postmultiplication of $V(t)^T \in \O(p)$, because of the nature of the equivalence classes in $\St(n,p)$. Therefore we set
            \begin{equation*}
                \mu\colon [0,1] \to \St(n,p),\quad \mu(t):=U V(t) \cos(\Sigma(t)) + Q(t) \sin(\Sigma(t))
            \end{equation*}
            and have $\gamma(t)=\pi^\SG(\mu(t))$. The derivative of $\gamma$ with respect to $t$ evaluated at $t=0$ is then given by
            \begin{equation}
            \label{eq:GrassExpDiffStiefelproof}
             \begin{split}
                \frac{\D}{\D t}\Big\vert_{t=0} \gamma(t) = \frac{\D}{\D t}\Big\vert_{t=0} \pi^\SG(\mu(t)) =\D \pi^\SG_{\mu(0)}\left(\dot\mu(0)\right)
                = \dot\mu(0) \mu(0)^T + \mu(0) \dot\mu(0)^T.
             \end{split}
            \end{equation}                
            But with the definitions above, $Y=\mu(0)$ and $\Gamma=\dot\mu(0)$, so (\ref{eq:GrassExpDiffStiefelproof}) is equivalent to (\ref{eq:GrassExpDiffStiefel}).
            The horizontal lift of (\ref{eq:GrassExpDiffStiefelproof}) to $Y$ is according to (\ref{eq:HorLiftToStiefel}) given by a postmultiplication of $Y$, which shows (\ref{eq:derivative_Grassmannexp}).
            Note however that $\Gamma \in T_Y\St(n,p)$ is not necessarily horizontal, so $0 \neq\Gamma^TY \in \so(p)$.
        \end{proof}
   
%-----------------------Mathias-------------------------
In order to remove the ``mutually distinct singular values'' assumption of Proposition~\ref{prop:derivative_Grassmannexp} and to remedy the numerical instability of the SVD in the presence of clusters of singular values, we introduce an alternative computational approach that relies on the derivative of the QR-decomposition rather than that of the SVD. Yet in this case, the ``non-zero singular values'' assumption is retained, and instabilities may arise for matrices that are close to being rank-deficient.

Let $U, \Delta^\hor_U, \tilde\Delta^\hor_U$ be as introduced in Prop. \ref{prop:derivative_Grassmannexp} (now with possibly repeated singular values of $\Delta^\hor_U$)
and consider the $t$-dependent QR-decomposition of the matrix curve $(\Delta+t\tilde\Delta)^\hor_U=Q(t)R(t)$.
The starting point is \eqref{eq:GrassmannExpO}, which can be transformed to 
\[
    \gamma(t) = \pi^\SG\left((U, Q(t)) \exp_m\begin{pmatrix}
    0 & -R(t)^T\\
    R(t) & 0
    \end{pmatrix}
    \begin{pmatrix}
    I_{p}\\
    0
    \end{pmatrix}
    \right)
    =: \pi^\SG(\tilde{\gamma}(t))
\]
by means of elementary matrix operations.
Write $M(t) = \begin{psmallmatrix}0 & -R(t)^T\\ R(t) & 0\end{psmallmatrix}$.
By the product rule,
\begin{equation}
\label{eq:diffGr_aux}
 \frac{\D}{\D t}\Big\vert_{t=0} \tilde{\gamma}(t)
 =(0, \dot Q(0)) \exp_m\left(M(0)\right)\begin{pmatrix}I_p\\ 0\end{pmatrix}
  + (U,Q(0))\frac{\D}{\D t}\Big\vert_{t=0} \exp_m\left(M(t)\right)\begin{pmatrix}I_p\\ 0\end{pmatrix}.
\end{equation}
The derivative $\frac{\D}{\D t}\big\vert_{t=0} \exp_m\left(M(t)\right) = \D(\exp_m)_{M(0)}(\dot M(0))$
can be computed according to Mathias' Theorem \cite[Thm 3.6, p. 58]{Higham:2008:FM} from
\begin{align*}
 \label{eq:Mathias}
 \exp_m \begin{pmatrix} M(0) & \dot M(0)\\ 0 & M(0)\end{pmatrix}
 &=  \begin{pmatrix} \exp_m(M(0)) & \frac{\D}{\D t}\big\vert_{t=0} \exp_m(M(0) + t\dot M(0))\\ 0 & \exp_m(M(0))\end{pmatrix}\\
 &= \begin{pmatrix}
   \begin{pmatrix}
   E_{11} & E_{12}\\
   E_{21} & E_{22}
   \end{pmatrix}
   &
      \begin{pmatrix}
   D_{11} & D_{12}\\
   D_{21} & D_{22}
   \end{pmatrix}\\ 
   \mathbf{0}
   &
      \begin{pmatrix}
   E_{11} & E_{12}\\
   E_{21} & E_{22}
   \end{pmatrix}
  \end{pmatrix}
\end{align*}
which is a $(4p\times 4p)$-matrix exponential written in sub-blocks of size $(p\times p)$.
Substituting in \eqref{eq:diffGr_aux} gives the $\mathcal{O}(np^2)$-formula
\begin{equation}
 \label{eq:Grdiff_Mathias}
  \frac{\D}{\D t}\Big\vert_{t=0} \tilde{\gamma}(t)
  = \dot Q(0) E_{21} + UD_{11} + Q(0)D_{21}.
 \end{equation}
This corresponds to \cite[Lemma 5]{ZimmermannHermite_2020}, which addresses the Stiefel case. The derivative matrices $\dot Q(0), \dot R(0)$ can be obtained from Alg. \ref{alg:QRdiff} in Appendix \ref{app:diffQR}.
The final formula is obtained by taking the projection into account as in \eqref{eq:GrassExpDiffStiefelproof},
where $\mu$ is to be replaced by $\tilde{\gamma}$. The horizontal lift is computed accordingly.

    The derivative of the Grassmann exponential can also be computed directly in $\Gr(n,p)$ without using horizontal lifts, at the cost of a higher computational complexity, but without restrictions with regard to the singular values. The key is again to apply Mathias' Theorem to evaluate the derivative of the matrix exponential. Let $P \in \Gr(n,p)$ and $\Delta = [\Omega, P],\ \widetilde{\Delta}=[\widetilde{\Omega},P] \in T_P\Gr(n,p)$ with $\Omega=(I_n-2P)\Delta,\ \widetilde{\Omega}=(I_n-2P)\widetilde{\Delta} \in \so_P(n)$. Denote $Q:= \expm(\Omega) \in \O(n)$ and $\Psi Q =\frac{\D}{\D t}\big\vert_{t=0}\expm(\Omega + t\widetilde{\Omega})$. Here, $\Psi \in \so(n)$, since $\expm(\Omega + t\widetilde{\Omega})$ is a curve in $\O(n)$ through $Q$ at $t=0$. Then a computation shows that the derivative of
    \begin{equation*}
        \Exp^\Gr_P(\Delta + t\widetilde{\Delta})=\expm(\Omega+t\widetilde{\Omega})P\expm(-\Omega-t\widetilde{\Omega})
    \end{equation*}
    is given by
    \begin{equation*}
        \frac{\D}{\D t}\Big\vert_{t=0}\Exp^\Gr_P(\Delta + t\widetilde{\Delta})=\Psi Q P Q^T + Q P (\Psi Q)^T \in T_{QPQ^T}\Gr(n,p).
    \end{equation*}
    The matrices $Q$ and $\Psi Q$ can be obtained in one calculation by evaluating the left side of
    \begin{equation*}
        \expm\begin{pmatrix}
            \Omega & \widetilde{\Omega}\\
            0 & \Omega
        \end{pmatrix}
        =
        \begin{pmatrix}
        \expm(\Omega) & \frac{\D}{\D t}\Big\vert_{t=0}\expm(\Omega + t\widetilde{\Omega}) \\
        0 & \expm(\Omega)
        \end{pmatrix}
        =
        \begin{pmatrix}
         Q & \Psi Q \\
         0 & Q
        \end{pmatrix}
    \end{equation*} 
    according to Mathias' Theorem.
    
\subsection{Parallel Transport}
\label{subsec:ParallelTransport}
    On a Riemannian manifold $(M,g)$, \emph{parallel transport} of a tangent vector $v \in T_pM$ along a smooth curve $\gamma \colon I \to M$ through $p$ gives a smooth vector field $V \in \Vectorfields(\gamma)$ along $\gamma$ that is parallel with respect to the Riemannian connection $\nabla$ and fulfills the initial condition $V(p)=v$. A vector field $V \in \Vectorfields(\gamma)$ along a curve $\gamma$ is a vector field that is defined on the range of the curve, i.e., $V \colon \gamma(I) \to TM$ and $V(\gamma(t))\in T_{\gamma(t)}M$. The term ``parallel'' means that for all $t \in I$, the covariant derivative of $V$ in direction of the tangent vector of $\gamma$ vanishes, i.e.
    \begin{equation*}
        \nabla_{\dot{\gamma}(t)}V=0.
    \end{equation*} 
    
    Parallel transport on the Grassmannian (ONB perspective) was studied in \cite{EdelmanAriasSmith1999}, where an explicit formula for the horizontal lift of the parallel transport of a tangent vector along a geodesic was derived, and in \cite{AbsilMahonySepulchre2004}, where a differential equation for the horizontal lift of parallel transport along general curves was given. In the next proposition, we complete the picture by providing a formula for the parallel transport on the Grassmannian from the projector perspective. Note that this formula is similar to the parallel transport formula in the preprint \cite{lai2020}.
    
    \begin{proposition}[Parallel Transport: Projector Perspective]
    \label{prop:parallel_transport_projector}
        Let $P \in \Gr(n,p)$ and $\Delta, \Gamma \in T_P\Gr(n,p)$. Then the parallel transport $\mathbb{P}_\Delta(\Exp_P^\Gr(t\Gamma))$ of $\Delta$ along the geodesic 
        \begin{equation*}
            \Exp_P^\Gr(t\Gamma)=\expm(t[\Gamma,P])P\expm(-t[\Gamma,P]) 
        \end{equation*}
        is given by
        \begin{equation*}
            \mathbb{P}_\Delta(\Exp_P^\Gr(t\Gamma))=\expm(t[\Gamma,P])\Delta\expm(-t[\Gamma,P]).
        \end{equation*} 
        \begin{proof}
            Denote $\gamma(t):=\Exp_P^\Gr(t\Gamma)$ and note that $\Omega:=[\Gamma,P] \in \so_P(n)$. The fact that $\mathbb{P}_\Delta(\Exp_P^\Gr(t\Gamma)) \in T_{\Exp_P^\Gr(t\Gamma)}\Gr(n,p)$ can be checked with Proposition\nobreakspace\ref{prop:TangentVectorCharacterizations}\nobreakspace\ref{eq:DP+PD=D}. To show that $\mathbb{P}_\Delta$ gives parallel transport, we need to show that $\nabla_{\dot{\gamma}(t)}(\mathbb{P}_\Delta(\gamma(t)))=\Pi_{T_{\gamma(t)}\Gr}\left(\D (\mathbb{P}_\Delta)_{\gamma(t)}(\dot{\gamma}(t))\right)=0$ as in \eqref{eq:connectionGr}. By making use of the chain rule, we have $\D (\mathbb{P}_\Delta)_{\gamma(t)}(\dot{\gamma}(t))= \frac{\D}{\D t} \mathbb{P}_\Delta(\gamma(t))=[\Omega,\mathbb{P}_\Delta(\gamma(t))]$, where $[\cdot,\cdot]$ denotes the matrix commutator. Applying the projection $\Pi_{T_{\gamma(t)}\Gr}$ from  \eqref{eq:projSymTPGr} and making use of the relation \eqref{eq:soPn} and the tangent vector properties from Proposition\nobreakspace\ref{prop:TangentVectorCharacterizations} give the desired result.
        \end{proof}        
    \end{proposition}
    Applying the horizontal lift to the parallel transport equation leads to the formula also found in \cite{EdelmanAriasSmith1999}. Let $Q=\begin{pmatrix} U & U_\perp \end{pmatrix} \in (\pi^\OG)^{-1}(P)$. Then $\Omega=Q\begin{psmallmatrix}0 & \smash{-A^T} \\ A & 0\end{psmallmatrix}Q^T$ and $\Delta=Q\begin{psmallmatrix}0 & \smash{B^T} \\ B & 0\end{psmallmatrix}Q^T$ for some $A,B \in \R^{(n-p)\times p}$.  According to \eqref{eq:HorLiftToStiefel}, the horizontal lift of $\mathbb{P}_\Delta(\Exp_P^\Gr(t\Gamma))$ to the Stiefel geodesic representative $U(t)=Q\expm(tQ^T\Omega Q)I_{n,p}$ is given by 
    a post-multiplication with $U(t)$, 
    \begin{equation*}
        \left(\mathbb{P}_\Delta(\Exp_P^\Gr(t\Gamma))\right)^\hor_{U(t)}=\mathbb{P}_\Delta(\Exp_P^\Gr(t\Gamma))U(t)
        = Q \expm\left(t \begin{pmatrix}0 & \smash{-A^T} \\ A & 0\end{pmatrix} \right) \begin{pmatrix}0 \\ B \end{pmatrix}.
    \end{equation*}
    This formula can be simplified similarly to \cite[Theorem 2.4]{EdelmanAriasSmith1999} by discarding all principal angles equal to zero. With notation as above, $\Gamma^\hor_U = U_\perp A$ and $\Delta^\hor_U = U_\perp B$. Let $r \leq \min(p,n-p)$ be the number of non-zero singular values of $\Gamma^\hor_U$. Denote the thin SVD of $\Gamma^\hor_U$ by $\Gamma^\hor_U=\hat{Q}\Sigma V^T$, where $\hat{Q} \in \St(n,r), \Sigma=\diag(\sigma_1, \dots, \sigma_r)$ and $ V \in \St(p,r)$, which means $\Sigma$ has full rank. Then $A = U_\perp^T\hat{Q}\Sigma V^T$ with $W:=U_\perp^T\hat{Q} \in \St(n-p,r)$. Similarly to the proof of Proposition\nobreakspace\ref{prop:GrassExp_Stiefel}, with $\gamma_\Gamma(t):=\Exp_P^\Gr(t\Gamma)$,
    \begin{equation}
        \label{eq:ParallelTransportHorLift}
    \begin{split}
        \left(\mathbb{P}_\Delta(\gamma_\Gamma(t))\right)^\hor_{U(t)} &= \left(-UV\sin(t\Sigma)W^T + U_\perp W \cos(t\Sigma)W^T + U_\perp(I_{n-p}-WW^T)\right)B\\
        &= (-UV\sin(t\Sigma)\hat{Q}^T + \hat{Q}\cos(t\Sigma)\hat{Q}^T + I_n-\hat{Q}\hat{Q}^T)\Delta^\hor_U.
    \end{split}
    \end{equation}
    The difference between this formula and the one found from \cite[Theorem 2.4]{EdelmanAriasSmith1999} is in the usage of the thin SVD and the therefore smaller matrices $\hat{Q},\Sigma$ and $V$, depending on the problem. But the first line also shows that if $r=n-p$, the term $I_{n-p}-WW^T$ vanishes, and therefore also the term $(I_n-\hat{Q}\hat{Q}^T)\Delta^\hor_U$. This can happen if $p\geq n/2$. \revcomm{For large $n$, \eqref{eq:ParallelTransportHorLift} allows for an $\mathcal{O}(np^2)$-computation of the parallel transport, which is efficient compared to the projector perspective of Proposition~\ref{prop:parallel_transport_projector}.}

\section{Symmetry and Curvature}
\label{sec:symcurve}
In this section, we establish the symmetric space structure of the Grassmann manifold by elementary means. The symmetric structure of the Grassmannian was for example shown in \cite[Vol. II]{KobayashiNomizu1996} and \cite{BorisenkoNikolaevskii1991}.

Exploiting the symmetric space structure, the curvature of the Grassmannian can be calculated explicitly. Curvature formulae for symmetric spaces can be found for example in \cite[\revcomm{Chapter 11, Proposition 11.31}]{ONeill1983} and \cite[Vol. II]{KobayashiNomizu1996}. To the best of the authors' knowledge, a first formula for the sectional curvature of the Grassmannian was given in \cite{Wong1968b}, without making use of the symmetric structure. The bounds were studied in \cite{WuChen1988}. In \cite{Leichtweiss1961}, curvature formulae have been derived in local coordinates via differential forms. Explicit curvature formulae for a generalized version of the Grassmannian as the space of orthogonal projectors were given in \cite{MachadoSalavessa1985}.

Curvature bounds are required for the analysis of Riemannian optimization problems \revcomm{(see, e.g., \cite{Alimisisetal2021,CriscitielloBoumal2023,ZhangSra2016})} and, in particular, for studying the Riemannian centers of mass, see for example \revcomm{\cite{AfsariTronVidal2013,boumal2014thesis}} and \cite{li2019}, and several references therein.
The sectional curvature features also in statistical problems on Riemannian manifolds \cite{Chakraborty2018}, and enables estimates for data processing errors on manifolds \cite{ZimmermannHermite_2020}.

\subsection{Symmetric Space Structure}
\label{sec:SymSpace}
    In differential geometry, a {\em metric symmetry at $q$} is an isometry $\sigma: M \to M$
    of a manifold $M$ that fixes a certain point $\sigma(q) = q$ with the additional property that $d\sigma_q = -\id|_{T_qM}$. This relates to the concept of a point reflection in Euclidean geometry.
    A (metric) {\em symmetric space} is a connected differentiable manifold that has a metric symmetry at every point, \cite[\revcomm{Chapter} 8]{ONeill1983}.  %\cite[\S XIII, 5]{lang2001fundamentals}.
    Below, we execute an explicit construction of symmetries for the Grassmannian, which compares to the 
    abstract course of action in \cite[\revcomm{Chapter} 11, p. 315ff]{ONeill1983}.

    Consider the orthogonal matrix
    $S_0 = \begin{psmallmatrix}
        I_p & 0\\
        0 & -I_{n-p}
    \end{psmallmatrix}\in \O(n)$.
    Then $S_0$ induces a symmetry at $P_0$ via $\sigma^{P_0}\colon P \mapsto P^{S_0} := S_0PS_0^T$,
    which is defined on all of $\Gr(n,p)$.
    Obviously, $\sigma^{P_0}(P_0) = P_0$.
    For any point $P\in \Gr(n,p)$ and any tangent vector $\Delta \in T_P\Gr(n,p)$,
    the differential in direction $\Delta$ can be computed as 
    $\D\sigma^{P_0}_P(\Delta)= \frac{\D}{\D t}|_{t=0} \sigma(P(t))$, where $P(t)$ is any curve on $\Gr(n,p)$ with $P(0) = P$ and $\dot{P}(0)=\Delta$. This gives 
    \[
    \D\sigma_{P_0}^{P_0}\colon T_{P_0}\Gr(n,p)\to T_{P_0}\Gr(n,p), \quad
    \begin{pmatrix}
        0 & B^T\\
        B &0
    \end{pmatrix} \mapsto 
    S_0\begin{pmatrix}
        0 & B^T\\
        B &0
    \end{pmatrix}S_0^T = - 
    \begin{pmatrix}
        0 & B^T\\
        B &0
    \end{pmatrix},
    \]
    so that  $\sigma^{P_0}$ is indeed a symmetry of $\Gr(n,p)$ at $P_0$.

    Given any other point $P\in \Gr(n,p)$, we can compute the EVD $P = QP_0Q^T$
    and define 
    $\sigma^P: \tilde{P}\mapsto (QS_0Q^T) \tilde{P} (QS_0Q^T)$.
    This isometry fixes $P$,  $\sigma^P(P) = P$. 
    Moreover, for any curve with $P(0) = P$, $\dot{P}(0) = \Delta\in T_P\Gr(n,p)$, it holds
    $\Delta = \frac{\D}{\D t}|_{t=0} Q(t)P_0Q^T(t) = \dot{Q}P_0Q^T + Q P_0 \dot{Q}^T$ (evaluated at $t=0$).
    Since $Q(t)$ is a curve on $O(n)$, it holds $Q^T\dot{Q} = - \dot{Q}^TQ$, so that
    $
    Q^T\dot{Q} = 
    \begin{psmallmatrix}
        C_{11} & -C_{21}^T\\
        C_{21} & C_{22}
    \end{psmallmatrix}
    $ is skew.
    As a consequence, we use the transformation $Q^T\Delta Q =  \begin{psmallmatrix}
        0      &  C_{21}^T\\
        C_{21} & 0
    \end{psmallmatrix}$ to move $\Delta$ to the tangent space at $P_0$ and compute
    \[
    \D\sigma^P_P(\Delta) = QS_0(Q^T\Delta Q)S_0Q^T = QS_0  \begin{pmatrix}
        0      &  C_{21}^T\\
        C_{21} & 0
    \end{pmatrix}  S_0Q^T
    = - Q(Q^T\Delta Q) Q^T =-\Delta.
    \]
    Hence, we have constructed metric symmetries at every point of $\Gr(n,p)$.

    The symmetric space structure of $\Gr(n,p)$ implies a number of strong properties.
    First of all, it follows that $\Gr(n,p)$ is {\em geodesically complete} \cite[\revcomm{Chapter} 8, Lemma 20]{ONeill1983}. %\cite[\S XIII, Prop. 5.2]{lang2001fundamentals}.
    This means that the maximal domain of definition for all Grassmann geodesics is the whole real line $\R$.
    As a consequence, all the statements of the Hopf-Rinow Theorem \cite[Chap. 7, Thm 2.8]{DoCarmo2013riemannian}, \cite[Thm 2.9]{Alexandrino2015} hold for the Grassmannian\revcomm{, as it is a connected manifold}:
    \begin{enumerate}
     \item The Riemannian exponential $\Exp_P^\Gr:T_P\Gr(n,p) \to \Gr(n,p)$ is globally defined.
     \item $(\Gr(n,p), \dist(\cdot,\cdot))$ is a complete metric space, where $\dist(\cdot,\cdot)$ is the Riemannian distance function.
     \item Every closed and bounded set in $\Gr(n,p)$ is compact.
    \end{enumerate}
    These statements are equivalent. Any one of them additionally implies
    \begin{enumerate}
    \setcounter{enumi}{3}
     \item For any two points $P_1,P_2\in \Gr(n,p)$, 
     there exists a geodesic $\gamma$ of length $L(\gamma) = \dist(P_1,P_2)$ that joins $P_1$ to $P_2$; hence
     any two points can be joined by a {\em minimal} geodesic segment. 
     \item The exponential map $\Exp_P^\Gr:T_P\Gr(n,p) \to \Gr(n,p)$ is surjective for all $P\in \Gr(n,p)$.
    \end{enumerate}

\subsection{Sectional Curvature}
\label{subsec:SectionalCurvature}   
%\subsection{The Curvature Tensor}
    For  $X,Y,Z \in \R^{(n-p)\times p}$, let $\hat{X}:=\begin{psmallmatrix} 0 & \smash{-}X^T \\ X & 0\end{psmallmatrix} \in 
    %\mathfrak{m} = 
    \Hor_I\O(n)$ and $\hat{Y},\hat{Z} \in \Hor_I\O(n)$ accordingly. Denote the projections to $T_{P_0}\Gr(n,p)$ by $x:=\D \pi^\OG_{P_0}(\hat{X})= \begin{psmallmatrix} 0 & X^T \\ X & 0\end{psmallmatrix} \in T_{P_0}\Gr(n,p)$, etc. Then, by \cite[Proposition 11.31]{ONeill1983}, the curvature tensor at $P_0$ is given by $R_{xy}z=\D \pi^\OG_{P_0}([\hat{Z},[\hat{X},\hat{Y}]])$, since the Grassmannian is symmetric and therefore also reductive homogeneous. This formula coincides with the formula found in \cite{MachadoSalavessa1985}. Explicitly, we can calculate
    \begin{equation*}
        R_{xy}z=\begin{pmatrix} 0 & B^T \\ B & 0\end{pmatrix} \in T_{P_0}\Gr(n,p),
    \end{equation*}
    where $B=Z X^T Y - Z Y^T X - X Y^T Z + Y X^T Z \in \R^{(n-p) \times p}$.
    
%\subsection{Sectional Curvature}
    The sectional curvature of the Grassmannian can be calculated by the following formulae. It depends only on the plane spanned by two given tangent vectors, not the spanning vectors themselves. For a Riemannian manifold, the sectional curvature completely determines the curvature tensor, see for example \cite[Proposition 8.31]{Lee2018riemannian}.
    
    \begin{proposition}
    Let $P\in \Gr(n,p)$ and let $\Delta_1,\Delta_2 \in T_P \Gr(n,p)$ span a non-degenerate plane in $T_P\Gr(n,p)$. The sectional curvature is then given by
    \begin{equation}
    \label{eq:seccurvDelta}
        K_P(\Delta_1,\Delta_2)=4\frac{\tr\mathopen{}\left(\Delta_1^2\Delta_2^2\right) - \tr\mathopen{}\left((\Delta_1\Delta_2)^2\right)}{\tr(\Delta_1^2)\tr(\Delta_2^2)-(\tr(\Delta_1\Delta_2))^2}
        =
        2\frac{\|[\Delta_1, \Delta_2]\|_F^2}{\|\Delta_1\|_F^2\|\Delta_2\|_F^2 - \langle \Delta_1,\Delta_2\rangle_0^2}.
    \end{equation}
    \begin{proof}
    This formula can be derived from the result in \cite{MachadoSalavessa1985}. For a direct proof, we proceed as follows. 
    The tangent vectors can be expressed as $\Delta_1=[\Omega_1,P],\ \Delta_2=[\Omega_2,P] \in T_P\Gr(n,p)$ for some $\Omega_1,\Omega_2 \in \so_P(n)$. Using the fact
    \begin{equation*}
        [\Omega_1,P][\Omega_2,P]=-\Omega_1\Omega_2,
    \end{equation*}
    we see that
    \begin{equation*}
        \tr\mathopen{}\left([\Omega_2,P][\Omega_1,P]^2[\Omega_2,P]\right) - \tr\mathopen{}\left(([\Omega_2,P][\Omega_1,P])^2\right)=\tr\mathopen{}\left(\Omega_2\Omega_1[\Omega_1,\Omega_2]\right).
    \end{equation*}
    The property that for any two $X,Y \in \Hor_I\O(n)
    %\mathfrak{m}
    $ the equality
    \begin{equation*}
        \tr(YX[X,Y])=\SP{[Y,[X,Y]],X}
    \end{equation*}
    holds, shows the claim according to \cite[Proposition 11.31]{ONeill1983}.
    \end{proof}
    \end{proposition}
    
    With (\ref{eq:GrassmannTangentvectorUBUT}) every $\Delta_i \in T_P\Gr(n,p)$ can be written as $\Delta_i=U_\perp B_i U^T+U B_i^T U_\perp^T$ for some $ \begin{pmatrix} U & U_\perp \end{pmatrix} \in (\pi^\OG)^{-1}(P)$ and $B_i \in \R^{(n-p)\times p}$. Since every tangent vector in $T_P\Gr(n,p)$ is uniquely determined by such a $B$ for a chosen representative $\begin{pmatrix} U & U_\perp \end{pmatrix}$, we can insert this into (\ref{eq:seccurvDelta}) and get the simplified formula
    \begin{equation}
    \label{eq:CurvB}
    \begin{split}
        K_P(B_1,B_2)&=\frac{\tr\mathopen{}\left(B_1^TB_2\left(B_2^T B_1-2B_1^TB_2\right)+B_1^TB_1B_2^TB_2\right)}{\tr\mathopen{}\left(B_1^TB_1\right)\tr\mathopen{}\left(B_2^TB_2\right)-\left(\tr\mathopen{}\left(B_1^TB_2\right)\right)^2}\\
        &=
        \frac{\|B_2^TB_1\|_F^2 + \|B_1B_2^T\|_F^2 - 2\langle B_2^TB_1, B_1^TB_2\rangle_0}
        {\|B_1\|_F^2 \|B_2\|_F^2 - \langle B_1,B_2\rangle_0^2}.
    \end{split}
    \end{equation}
    This formula is equivalent to the slightly more extended form in \cite{Wong1968b} and depends only on the factors $B_1^TB_2$, $B_1^TB_1$ and $B_2^TB_2 \in \R^{p \times p}$. It also holds for the horizontal lifts  of $\Delta_i$ by just replacing the symbols $B_i$ by $(\Delta_i)^\hor_U$, which can also be shown by exploiting \eqref{eq:GrassmannTangentvectorUBUT} and $(\Delta_i)^\hor_U=U_\perp B_i$.
    
    In summary, for two orthonormal tangent vectors $\Delta_1=[\Omega_1,P],\ \Delta_2=[\Omega_2,P] \in T_P\Gr(n,p)$ with $\Omega_1,\Omega_2 \in \so_P(n)$, i.e.\revcomm{,}
    \begin{equation*}
        1=\SP{\Delta_i,\Delta_i}=\frac{1}{2}\tr(\Delta_i^T\Delta_i) \text{ and } 0=\SP{\Delta_1,\Delta_2}
    \end{equation*}
    the sectional curvature is given by
    \begin{equation*}
    \begin{split}
        K_P(\Delta_1,\Delta_2)&=\tr(\Omega_2\Omega_1[\Omega_1,\Omega_2])\\
        &= \tr\left(\Delta^{\hor^T}_{2,U}\Delta^\hor_{1,U}\left(\Delta^{\hor^T}_{1,U}\Delta^\hor_{2,U}-2\Delta^{\hor^T}_{2,U}\Delta^\hor_{1,U}\right)+ \Delta^{\hor^T}_{1,U}\Delta^\hor_{1,U}\Delta^{\hor^T}_{2,U}\Delta^\hor_{2,U}\right).
    \end{split}
    \end{equation*}
    Inserting any pair of orthonormal tangent vectors shows that for $n>2$, the sectional curvature of the real projective space $\Gr(n,1)=\RP^{n-1}$ is constant $K_P\equiv 1$, as it is by the same calculation for $\Gr(n,n-1)$, see also \cite{Wong1968b}. The same source also states a list of facts about the sectional curvature on $\Gr(n,p)$ without proof, especially that
    \begin{equation}
    \label{eq:sectionalcurvaturebounds}
        0 \leq K_P(\Delta_1,\Delta_2) \leq 2
    \end{equation}
    for $\min(p,n-p) \geq 2$. Nonnegativity follows directly from \eqref{eq:seccurvDelta}. The upper bound was proven in \cite{WuChen1988}, by proving that for any two matrices $A,B \in \R^{m \times n}$, with $m,n \geq 2$, the inequality
    \begin{equation}
        \norm{AB^T-BA^T}_F^2\leq 2\norm{A}_F^2\norm{B}_F^2
    \end{equation}
    holds. Note that \eqref{eq:CurvB} can be rewritten as
    \begin{equation*}
        K_P(B_1,B_2)=\frac{\frac12\left(\norm{B_1B_2^T-B_2B_1^T}_F^2+\norm{B_1^TB_2-B_2^TB_1}_F^2\right)}{\norm{B_1}_F^2\norm{B_2}_F^2-(\tr(B_1^TB_2))^2}.
    \end{equation*} 
    The bounds of the sectional curvature \eqref{eq:sectionalcurvaturebounds} are sharp for all cases except those mentioned in the next paragraph: The lower bound zero is attained whenever $\Delta_1,\Delta_2$ commute.
    The upper curvature bound is attained, e.g., for 
    $B_1 = \begin{psmallmatrix}
     1&1\\
     -1&1
     \end{psmallmatrix},\
     B_2 = \begin{psmallmatrix}
     -1&1\\
     -1&-1
     \end{psmallmatrix}
    $, or matrices containing $B_1$ and $B_2$ as their top-left block and else only zeros, when $p>2$.

    In \cite{Leichtweiss1961} it was shown that a Grassmannian $\Gr(n,p)$ 
    features a strictly positive sectional curvature $K_P$ only if the sectional curvature is constant throughout. The sectional curvature is constant (and equal to $K_P \equiv 1$) only in the cases $p=1,\ n>2$ or $p=n-1,\ n>2$. In the case of $n=2,\ p=1$, the sectional curvature is not defined, since $\dim(\Gr(2,1))=1$.
    Hence, in this case, there are no non-degenerate two-planes in the tangent space.
    
\section{Cut Locus and Riemannian Logarithm}
\label{sec:CutLocusLogarithm}
    We have seen in Section~\ref{sec:SymSpace} that $\Gr(n,p)$ is a complete Riemannian manifold. On such manifolds, the \emph{cut locus} of a point $P$ consists of those points $F$ beyond which the geodesics starting at $P$ cease to be length-minimizing. It is known~\cite[Ch.~III, Prop. 4.1]{sakai1996riemannian} that $P$ and $F$ are in each other's cut locus if there is more than one shortest geodesic from $P$ to $F$. We will see that, on the Grassmannian, this ``if'' is an ``if and only if'' \revcomm{(in other words, the Grassmannian does not admit singular cut points in the sense of~\cite{Bishop77})}, and moreover ``more than one'' is always either two or infinitely many.

    To get an intuitive idea of the cut locus, think of the earth as an ideal sphere. Then the cut locus of the north pole is the south pole, as it is the only point beyond which the geodesics starting at the north pole cease to be length-minimizing. In the case of the sphere, the ``if and only if'' statement that we just mentioned for the Grassmannian also holds; however, for the sphere, ``more than one'' is always infinitely many.

    Given two points $P, F \in \Gr(n,p)$ that are not in each other's cut locus, the unique smallest \revcomm{norm} tangent vector $\Delta\in T_P \Gr(n,p)$ such that $\Exp^\Gr_P(\Delta)=F$ is called the \emph{Riemannian logarithm} of $F$ at $P$. We propose an algorithm that calculates the Riemannian logarithm. Moreover, in the case of cut points, the algorithm is able to return any of the (two or infinitely many) smallest $\Delta\in T_P \Gr(n,p)$ such that $\Exp^\Gr_P(\Delta)=F$. This ability comes from the indeterminacy of the SVD operation invoked by the algorithm.

    The horizontal lift of the exponential map (\ref{eq:GrassmannExpSt}) depends explicitly on the so called \emph{principal angles} between two points and allows us to give explicit formulae for different geodesics between $P$ and a cut point $F$. We observe that the inherent ambiguity of the SVD, see Appendix \ref{app:MatrixBasics}, corresponds to the different geodesics connecting the same points.

    Our approach allows data processing schemes to explicitly map any given set of points on the Grassmannian to any tangent space $T_P\Gr(n,p)$, with the catch that possibly a subset of the points (namely those that are in the cut locus of $P$), is mapped to a set of tangent vectors each, instead of just a single one.
    
    The cut locus, and the related injectivity radius, play an important role in curve fitting methods on manifolds \cite{gousenbourger2019} and the analysis of Riemannian optimization problems \cite{AfsariTronVidal2013}. The ability to tackle cut points numerically is of special importance for computing so-called \emph{almost gradients}, which enable the computation of Riemannian barycenters for not necessarily localized point sets, see \cite[Section 6.2]{AfsariTronVidal2013}.
\subsection{Cut Locus}
    We can introduce the cut locus of the Grassmannian by applying the definitions of \cite[Chap. 10]{Lee2018riemannian} about cut points to $\Gr(n,p)$. In the following, let $P \in \Gr(n,p)$ and $\Delta \in T_P\Gr(n,p)$ and $\gamma_\Delta \colon t \mapsto \Exp^\Gr_P(t\Delta)$. Then the \emph{cut time of $(P,\Delta)$} is defined as
    \begin{equation*}
        t_\cut(P,\Delta) := \sup\{ b > 0 \mid \text{the restriction of $\gamma_\Delta$ to $[0,b]$ is minimizing}\}.
    \end{equation*}
    The \emph{cut point of $P$ along $\gamma_\Delta$} is given by $\gamma_\Delta(t_\cut(P,\Delta))$ and the \emph{cut locus of $P$} is defined as
    \begin{equation*}
        \Cut_P:=\{F \in \Gr(n,p) \mid F = \gamma_\Delta(t_\cut(P,\Delta)) \text{ for some } \Delta \in T_P\Gr(n,p)\}.
    \end{equation*} 
    In \cite{Wong1967,Sakai1977}, it is shown that the cut locus of $P = UU^T \in \Gr(n,p)$ is the set of all (projectors onto) subspaces with at least one direction orthogonal to all directions in the subspace onto which $P$ projects, i.e.
        \begin{equation}
        \label{eq:Cut_P}
            \Cut_P=\left\lbrace F= YY^T \in \Gr(n,p) \ \middle|\ \rank(U^TY) < p \right\rbrace.
        \end{equation} 
    This means that the cut locus can be described in terms of \emph{principal angles}: The principal angles $\theta_1,\ldots,\theta_p\in[0,\frac{\pi}{2}]$ between two subspaces $\Space{U}$ and $\Space{\widetilde{U}}$ are defined 
    recursively by
    \[
    \cos(\theta_k) := u_k^Tv_k := 
        \max_{\begin{array}{l} u\in \Space{U}, \|u\|=1\\
        u\bot u_1,\ldots, u_{k-1}
        \end{array}
        }
        \max _{\begin{array}{l} v\in \Space{\widetilde{U}}, \|v\|=1\\
          v\bot v_1,\ldots v_{k-1}
         \end{array}
        }
    u^Tv.
    \]
    They can be computed via 
    $\theta_k := \arccos(s_k) \in [0,\frac{\pi}{2}]$,
    where $s_k\leq 1$ is the $k$-largest singular value of $U^T\tilde{U}\in\R^{p\times p}$
    for any two Stiefel representatives $U$ and $\tilde{U}$.
    According to this definition, the principal angles are listed in ascending order: $0\leq\theta_1\leq \ldots\leq\theta_p\leq\frac{\pi}{2}$.
    In other words, the cut locus of $P$ consists of all points $F \in \Gr(n,p)$ with at least one principal angle between $P$ and $F$ being equal to $\frac{\pi}{2}$.
    
    Furthermore, as in \cite[\revcomm{Chapter 10, p. 310}]{Lee2018riemannian}, we introduce the \emph{tangent cut locus of $P$} by
    \begin{equation*}
        \TCL_P:=\{\Delta \in T_P\Gr(n,p) \mid \norm{\Delta}=t_\cut(P,\Delta/\norm{\Delta})\}
    \end{equation*}
    and the \emph{injectivity domain of $P$} by
    \begin{equation*}
        \ID_P:=\{ \Delta \in T_P\Gr(n,p) \mid \norm{\Delta}<t_\cut(P,\Delta/\norm{\Delta})\}.
    \end{equation*}
    The cut time can be explicitly calculated by the following proposition.
    \begin{proposition}
        Let $P=UU^T \in \Gr(n,p)$ and $\Delta \in T_P\Gr(n,p)$. Denote the largest singular value of $\Delta^\hor_U \in \Hor_U\St(n,p)$ by $\sigma_1$. Then
        \begin{equation}
            t_\cut(P,\Delta)=\frac{\pi}{2\sigma_1}.
        \end{equation}
        \begin{proof}
            Since $\gamma_\Delta(t_\cut(P,\Delta))\in \Cut_P$, by \eqref{eq:GrassmannExpSt} we have
            \begin{equation*}
                \rank(U^T(UV\cos(t_\cut(P,\Delta)\Sigma)V^T + \hat{Q}\sin(t_\cut(P,\Delta)\Sigma)V^T))<p,
            \end{equation*}
            which is equivalent to $\cos(t_\cut(P,\Delta)\sigma_1)=0$.
        \end{proof}
    \end{proposition}
    Now we see that the tangent cut locus $\TCL_P$ consists of those tangent vectors for which $\sigma_1$ (the largest singular value of the horizontal lift) fulfills $\sigma_1=\frac{\pi}{2}$ and the injectivity domain $\ID_P$ contains the tangent vectors with $\sigma_1<\frac{\pi}{2}$. 
    
    The \emph{geodesic distance} is a natural notion of distance between two points on a Riemannian manifold. It is defined as the length of the shortest curve(s) between two points as measured with the Riemannian metric, if such a curve exists. On the Grassmannian, it can be calculated as the two-norm of the vector of principal angles between the two subspaces, cf. \cite{Wong1967}, i.e.
    \begin{equation}
    \label{eq:distGrass}
        \dist(\Space{U},\Space{\widetilde{U}}) = \left(\sum_{i=1}^p \sigma_i^2\right)^{\frac{1}{2}}.
    \end{equation}
    This shows that for any two points on the Grassmann manifold $\Gr(n,p)$, the geodesic distance is bounded by 
    \begin{equation*}
        \dist(\Space{U},\Space{\widetilde{U}})\leq \sqrt{p}\frac{\pi}{2},                                                                               
    \end{equation*}
    which was already stated in \cite[Theorem 8]{Wong1967}. 
    
     \begin{remark}
      There are other notions of distance on the Grassmannian that can also be computed from the principal angles, but which are not equal to the geodesic distance, see \cite[\S 4.5]{EdelmanAriasSmith1999}, \cite{QiuZhangLi2005},  \cite[Table 2]{YeLim2016}. In the latter reference, 
     it is also shown that all these distances can be generalized to subspaces of different dimensions by introducing Schubert varieties and adding $\frac{\pi}{2}$ for the ``missing'' angles.
     \end{remark}

    The \emph{injectivity radius} at $P \in \Gr(n,p)$ is defined as the distance from $P$ to its cut locus, or equivalently, as the supremum of the radii $r$ for which $\Exp_P^\Gr$ is a diffeomorphism from the open ball $B_r(0) \subset T_P\Gr(n,p)$ onto its image. The injectivity radius at every $P$ is equal to $\inj(P)=\frac{\pi}{2}$, since there is always a subspace $F$ for which the principal angles between $P$ and $F$ are all equal to zero, except one, which is equal to $\frac{\pi}{2}$. For such an $F$ it holds that $\dist(P,F)=\frac{\pi}{2}$, c.f. \eqref{eq:distGrass}, and $F \in \Cut_P$. For all other points $\widetilde{F}$ with $\dist(P,\widetilde{F})<\frac{\pi}{2}$, all principal angles are strictly smaller than $\frac{\pi}{2}$, and therefore $\widetilde{F} \notin \Cut_P$.

    \begin{proposition}
        Let $P= UU^T \in \Gr(n,p)$ and $\Delta \in T_P\Gr(n,p)$. Consider the geodesic segment $\gamma_\Delta\colon [0,1] \ni t \mapsto \Exp^\Gr_P(t\Delta)$. Let the SVD of the horizontal lift of $\Delta$ be given by $\hat{Q}\Sigma V^T = \Delta^\hor_U \in \Hor_U\St(n,p)$, where $\Sigma = \diag(\sigma_1,\dots,\sigma_p)$.
        \begin{enumerate}[label=\alph*)]
         \item \label{prop:cutlocus_1} If the largest singular value $\sigma_1<\pi/2$, then the geodesic segment $\gamma_\Delta$ is unique minimizing.
         \item \label{prop:cutlocus_2} If the largest singular value $\sigma_1=\pi/2$, then the geodesic segment $\gamma_\Delta$ is non-unique minimizing.
         \item \label{prop:cutlocus_3} If the largest singular value $\sigma_1>\pi/2$, then the geodesic segment $\gamma_\Delta$ is not minimizing.
        \end{enumerate}
        \begin{proof}
            In case of \ref{prop:cutlocus_1}, $\gamma_\Delta$ is minimizing by definition of the cut locus. It is unique by \cite[Thm. 10.34 c)]{Lee2018riemannian}. In case of \ref{prop:cutlocus_2}, $\gamma_\Delta$ is still minimizing by the definition of the cut locus. For non-uniqueness, replace $\sigma_1$ by $-\frac{\pi}{2}$ (instead of $\frac{\pi}{2}$) and observe that we get a different geodesic with the same length and same endpoints. Case \ref{prop:cutlocus_3} holds by definition of the cut locus.
        \end{proof}
    \end{proposition}
    
\subsection{Riemannian Logarithm}
\label{subsec:RiemannianLogarithm}
    For any $P \in \Gr(n,p)$, the restriction of $\Exp_P^\Gr$ to the injectivity domain $\ID_P$ is a diffeomorphism onto $\Gr(n,p)\setminus \Cut_P$ by \cite[Theorem 10.34]{Lee2018riemannian}. This means that for any $F \in \Gr(n,p)\setminus \Cut_P$ there is a unique tangent vector $\Delta \in \ID_P$ such that $\Exp_P^\Gr(\Delta)=F$. The mapping that finds this $\Delta$ is conventionally called the \emph{Riemannian logarithm}. Furthermore, \cite[Thm. 10.34]{Lee2018riemannian} states that the restriction of $\Exp_P^\Gr$ to the union of the injectivity domain and the tangent cut locus $\ID_P \cup \TCL_P$ is surjective. Therefore for any $F \in \Cut_P$ we find a (non-unique) tangent vector which is mapped to $F$ via the exponential map. We propose Algorithm \ref{alg:modgrasslog}, which computes the unique $\Delta \in \ID_P \subset T_P\Gr(n,p)$ in case of $F \in \Gr(n,p)\setminus \Cut_P$ and one possible $\Delta \in \TCL_P\subset \Gr(n,p)$ for $ F\in \Cut_P$. In the latter case, all other possible $\tilde\Delta \in \TCL_P$ such that $\Exp_P^\Gr(\tilde\Delta)=F$ can explicitly be derived from that result.
    
    \begin{algorithm}[H]
    \caption{Extended Grassmann Logarithm with Stiefel representatives}
    \label{alg:modgrasslog}
    \begin{algorithmic}[1]
        \Require $U, Y \in \St(n,p)$ representing $P=UU^T,\ F=YY^T \in \Gr(n,p)$, respectively
        \State $\widetilde{Q} \widetilde{S} \widetilde{R}^T \overset{\text{\tiny SVD}}{:=} Y^T U$  \Comment{SVD}
        \State $Y_* := Y(\widetilde{Q} \widetilde{R}^T)$ \Comment{Procrustes processing}
        \State $\hat{Q} \hat{S} R^T \overset{\text{\tiny SVD}}{:=}(I_n-UU^T)Y_*$ \Comment{compact SVD}
        \State $\Sigma := \arcsin(\hat{S})$ \Comment{element-wise on the diagonal}
        \State $\Delta_U^\hor:= \hat{Q} \Sigma R^T$
        \Ensure smallest $\Delta_U^\hor \in \Hor_U\St(n,p)$ such that $\Exp_P^\Gr(\Delta)=F$
    \end{algorithmic}
    Remark: In Step 1, the expression $\overset{\text{\tiny SVD}}{:=}$ is to be understood as ``is \emph{an} SVD''. In case of $F \in \Cut_P$, i.e. singular values equal to zero, different choices of decompositions lead to different valid output vectors $\Delta^\hor_U$. The non-uniqueness of the compact SVD in Step 3 does not matter, because $\Sigma=\arcsin(\hat{S})$, and $\arcsin$ maps zero to zero and repeated singular values to repeated singular values. Therefore any non-uniqueness cancels out in the definition of $\Delta^\hor_U$.
    \end{algorithm}
    
    Before we prove the claimed properties of Algorithm \ref{alg:modgrasslog}, let us state the following: 
    An algorithm for the Grassmann logarithm with Stiefel representatives only was derived in \cite[\revcomm{Section 3.8}]{AbsilMahonySepulchre2004}. The Stiefel representatives are however not retained in this algorithm, i.e., coupling the exponential map and the logarithm recovers the input subspace
    but produces a different Stiefel representative $\tilde{Y}=\Exp^\Gr_U(\Log^\Gr_U(Y))\neq Y$ as an output. Furthermore, it requires the matrix inverse of $U^TY$, which also means that it only works for points not in the cut locus, see~\eqref{eq:Cut_P}.
    By slightly modifying this algorithm we get Algorithm~\ref{alg:modgrasslog}, which retains the Stiefel representative, does not require the calculation of the matrix inverse $(U^TY)^{-1}$ and works for all pairs of points. The computational procedure of Algorithm~\ref{alg:modgrasslog} was first published in the \revcomm{preprint of the} book chapter \cite{zimmermann2019modelreduction}.

    In the following Theorem \ref{thm:GrLog}, we show that Algorithm \ref{alg:modgrasslog} indeed produces the Grassmann logarithm for points not in the cut locus.
    \begin{theorem}
    \label{thm:GrLog}
        Let $P=UU^T \in \Gr(n,p)$ and $F=YY^T \in \Gr(n,p)\setminus\Cut_P$ be two points on the Grassmannian. Then Algorithm \ref{alg:modgrasslog} computes the horizontal lift of the Grassmann logarithm $\Log_P^\Gr(F)=\Delta \in T_P\Gr(n,p)$ to $\Hor_U\St(n,p)$. It retains the Stiefel representative $Y_*$ when coupled with the Grassmann exponential on the level of Stiefel representatives (\ref{eq:GrassmannExpStRep}), i.e.
        \begin{equation*}
            Y_*=\Exp_{U}^\Gr(\Delta^\hor_U).
        \end{equation*}
        \begin{proof}
            First, Algorithm \ref{alg:modgrasslog} aligns the given subspace representatives $U$ and $Y$ by producing a representative of the equivalence class $[Y]$ that is ``closest'' to $U$. To this end, the Procrustes method is used, cf. \cite[\revcomm{Theorem 8.6}]{Higham:2008:FM}.
            Procrustes
            gives 
            \begin{equation*}
                QR^T = \argmin_{\Phi \in O(p)} \norm{U-Y\Phi}_F,
            \end{equation*}
            by means of the SVD
            \begin{equation}
                \label{eq:SVD-UTU}
                Y^T U = Q S R^T,
            \end{equation}
            chosen here to be with singular values in ascending order from the top left to the bottom right. 
            Therefore $Y_* := YQ R^T$ represents the same subspace $[Y_*]=[Y]$, but 
            \begin{equation*}
            U^TY_*=RSR^T
            \end{equation*}
            is symmetric. Now, we can split $Y_*$ with the projector $ P = UU^T$ onto $\Span(U)$ and the projector $I_n-UU^T$ onto the orthogonal complement of $\Span(U)$ via
            \begin{equation}
                \label{eq:SplitU*}
                Y_*=UU^TY_* + (I_n-UU^T)Y_*=URSR^T+(I_n-UU^T)Y_*.
            \end{equation}
            If we denote the part of $Y_*$ that lies in $\Span(U)^\perp$ by $L:= (I_n-UU^T)Y_*$, we see that
            \begin{equation*}
                L^TL=Y_*^T(I_n-UU^T)Y_*=I_n-RS^2R^T = R(I_n-S^2)R^T.
            \end{equation*}
            That means that $\hat{S}:=\sqrt{(I_n-S^2)}$ is the diagonal matrix of singular values of $L$, with the singular values in descending order. The square root is well-defined, since $(I_n-S^2)$ is diagonal with values between $0$ and $1$.
            Note also that the column vectors of $R$ are a set of orthonormal eigenvectors of $L^TL$, i.e., a compact singular value decomposition of $L$ is of the form
            \begin{equation}
                \label{eq:SVD-L}
                L=(I_n-UU^T)Y_*=\hat{Q}\hat{S}R^T,
            \end{equation}
            where again $\hat{Q} \in \St(n,p)$.
            Define $\Sigma := \arccos (S)$, where the arcus cosine (and sine and cosine in the following) is applied entry-wise on the diagonal. Then $S=\cos(\Sigma)$ and $\hat{S}=\sin(\Sigma)$. Inserting in \eqref{eq:SplitU*} gives
            \begin{equation*}
                Y_*=UR\cos(\Sigma)R^T + \hat{Q}\sin(\Sigma)R^T.
            \end{equation*}
            This is exactly the exponential with Stiefel representatives (\ref{eq:GrassmannExpStRep}), i.e., $\Exp_{U}^\Gr(\Delta^\hor_U)=Y_*$, 
            where $\Delta^\hor_U=\hat{Q}\Sigma R^T$. We also see that the exact matrix representative $Y_*$, and not just any equivalent representative, is computed by plugging $\Delta^\hor_U$ into the exponential $\Exp^\Gr_U$.

            The singular value decomposition in (\ref{eq:SVD-UTU}) differs from the usual SVD -- with singular values in descending order -- only by a permutation of the columns of $Q$ and $R$. But if $Y^T U = Q S R^T$ is an SVD with singular values in ascending order and $Y^T U = \widetilde{Q} \widetilde{S} \widetilde{R}^T$ is an SVD with singular values in descending order, the product $QR^T=\widetilde{Q}\widetilde{R}^T$ does not change, i.e., the computation of $Y_*$ is not affected. Therefore we can compute the usual SVD for an easier implementation and keep in mind that $\widetilde{S}^2+\hat{S}^2\neq I_n$.
            
            It remains to show that $\Delta \in \ID_P$, so that it is actually the Riemannian logarithm. Since $F$ is not in the cut locus $\Cut_P$, we have $\rank(U^TY)=p$, which means that the smallest singular value of $U^TY$ is larger than zero (and smaller than or equal to one). Therefore the entries of $\Sigma=\arccos(S)$ are smaller than $\frac{\pi}{2}$, which shows the claim.
        \end{proof}
    \end{theorem}

    \revcomm{
    \begin{remark}
        It should be noted that the compact SVD of the $n \times p$ matrix in Step 3 of Algorithm~\ref{alg:modgrasslog} does not need to be computed explicitly. As can be seen from the proof of Theorem~\ref{thm:GrLog}, the factors $\hat{S}$ and $R$ can be obtained from the SVD of $Y^TU \in \R^{p \times p}$ in Step 1 of Algorithm~\ref{alg:modgrasslog}, by flipping the order of columns of $\tilde{R}$ to obtain $R$, and by flipping the order of the diagonal $\tilde{S}$ to obtain $S$ and calculate $\hat{S}=\sqrt{(I_n-S^2)}$. In the end, $\hat{Q}$ is obtained by $\hat{Q}=(I_n-UU^T)YQ\hat{S}^{-1}$, where $Q$ is $\tilde{Q}$ with flipped columns. (When $\hat{S}$ has zeros on the diagonal, the resulting $0/0$ ambiguity can be resolved in any way that preserves the orthogonality of $\hat{Q}$; this has no impact on the output of Algorithm~\ref{alg:modgrasslog} in view of the remark therein.) This course of action with just one SVD can be compared to approaches in \cite[Task 2]{Gallivan_etal2003}, using a thin CS-decomposition of a larger matrix, and \cite[Equation~(16)]{AlimisisVandereycken2023}.
    \end{remark}
    }
    
    The next theorem gives an explicit description of the shortest geodesics between a point and another point in its cut locus.
    \begin{theorem}
    \label{thm:multiple_shortest_geodesics}
        For $P=UU^T \in \Gr(n,p)$ and some $F=YY^T \in \Cut_P$, let $r$ denote the number of principal angles between $P$ and $F$ equal to $\frac{\pi}{2}$. Then $\Delta \in \TCL_P \subset T_P\Gr(n,p)$ is a minimizing solution of 
        \begin{equation}
        \label{eq:multiple_shortest_geodesics}
            \Exp_P^\Gr(\Delta)=F
        \end{equation} 
        if and only if the horizontal lift $\Delta^\hor_U$ is an output of Algorithm \ref{alg:modgrasslog}.
                
        Consider the compact SVD $\Delta^\hor_U = \hat{Q} \Sigma R^T$. Then the horizontal lifts of all other minimizing solutions of \eqref{eq:multiple_shortest_geodesics} are given by
        \begin{equation*}
                (\Delta_W)^\hor_U:= \hat{Q}\Sigma\diag(W,I_{p-r})R^T,
        \end{equation*}
        where $W \in \O(r)$ and $\diag(W,I_{p-r})=\begin{psmallmatrix}\smash{W} & 0 \\ 0 & I_{p-r} \end{psmallmatrix}$ denotes a block diagonal matrix. The shortest geodesics between $P$ and $F$ are given by
        \begin{equation*}
            \gamma_W(t):= \Exp^\Gr_P(t\Delta_W)=[UR\diag(W^T,I_{p-r}) \cos(t\Sigma) + \hat{Q}\sin(t\Sigma)].
        \end{equation*}
        \begin{proof}
            Algorithm \ref{alg:modgrasslog} continues to work for points in the cut locus, but the result is not unique. With an SVD of $Y^T U$ with singular values in ascending order, the first $r$ singular values are zero. By Proposition~\ref{prop:ambiguitySVD},
            \begin{equation*}
                Y^TU=QSR^T=Q\diag(W_1,D)S\diag(W_2,D^T)R^T,
            \end{equation*}
            where $D \in \O(p-r)$ with $(D)_{ij}=0 $ for $ s_i \neq s_j$ and $W_1,W_2 \in \O(r)$ arbitrary. Then $Y_*$ is not unique anymore, but is given as the set of matrices
            \begin{equation*}
                \left\{Y_{*,W_1,W_2}:=YQ\diag(W_1W_2,I_{p-r})R^T|\hspace{0.1cm} W_1,W_2 \in O(r)\right\}.
            \end{equation*}
            Define $W:=W_1W_2$ and $\hat{W}:= \diag(W,I_{p-r})$. Then
            \begin{equation*}
                \begin{split} 
                    (I_n-UU^T)Y_{*,W}&=(I_n-UU^T)YQ\hat{W}R^T\\
                    &=\underbrace{(I_n-UU^T)YQR^T}_{\hat{Q}\hat{S}R^T}R\hat{W}R^T=\hat{Q}\hat{S}\hat{W}R^T.
                \end{split}
            \end{equation*}
            With $\Sigma=\arcsin(\hat{S})=\arccos(S)$, every matrix
            \begin{equation*}
                (\Delta_W)^\hor_U:=\hat{Q}\Sigma\hat{W}R^T
            \end{equation*}
            is the horizontal lift of a tangent vector at $P$ of a geodesic towards $F$: For the exponential, it holds that
            \begin{equation*}
                \begin{split}
                    \Exp_{P}^\Gr(\Delta_W)&=[UR\hat{W}^T\cos(\Sigma)\hat{W}R^T+\hat{Q}\sin(\Sigma)\hat{W}R^T]\\
                    &=[UR\hat{W}^T\cos(\Sigma)+\hat{Q}\sin(\Sigma)]
                    =[UR\cos(\Sigma)+\hat{Q}\sin(\Sigma)]
                    =[Y],
                \end{split}
            \end{equation*} 
            where the third equality holds, since $\Sigma=\diag(\frac{\pi}{2},\dots, \frac{\pi}{2}, \sigma_{r+1}, \dots, \sigma_{p})$. But the geodesics $\gamma_W$ starting at $[U]$ in the directions $\Delta_W$ differ, i.e.
            \begin{equation*}
                \begin{split}
                    \gamma_W(t)&=[UR\hat{W}^T\cos(t\Sigma)\hat{W}R^T+\hat{Q}\sin(t\Sigma)\hat{W}R^T]\\
                    &=[UR\hat{W}^T\cos(t\Sigma)+\hat{Q}\sin(t\Sigma)].
                \end{split}
            \end{equation*}
            Hence, the ambiguity factor $\hat{W}^T=\diag(W^T,I_{p-r})$ does not vanish for $0<t<1$. The geodesics are all of the same (minimal) length, since the singular values do not change and $\gamma_W(1)=\Exp^\Gr_{P}(\Delta_W)$.
            
            To show that there are no other solutions, let $\bar\Delta \in \TCL_P$ fulfill $\Exp^\Gr_P(\bar\Delta)=F$ and $\bar\Delta^\hor_U=\bar Q \bar \Sigma \bar R^T$. Then by \eqref{eq:GrassmannExpStRep} there is some $M \in \O(p)$ such that
            \begin{equation*}
                YM = U\bar R\cos(\bar\Sigma)\bar R^T + \bar Q\sin(\bar\Sigma)\bar R^T,
            \end{equation*}
            which implies $U^TYM = \bar R \cos(\bar\Sigma)\bar R^T$, which is an SVD of $U^TYM$. Therefore $\bar Y_*:=YM\bar R \bar R^T=YM$ fulfills the properties of $Y_*$ of Algorithm~\ref{alg:modgrasslog}. Now
            \begin{equation*}
                (I_n-UU^T)\bar Y_* = (I_n-UU^T)(U\bar R\cos(\bar\Sigma)\bar R^T + \bar Q\sin(\bar\Sigma)\bar R^T)= \bar Q\sin(\bar\Sigma)\bar R^T,
            \end{equation*}
            which means that $\bar Q\sin(\bar\Sigma)\bar R^T$ is a compact SVD of $(I_n-UU^T)\bar Y_*$. Therefore $\bar \Delta^\hor_U$ is an output of Algorithm~\ref{alg:modgrasslog} and the claim is shown.
        \end{proof}
    \end{theorem}
    \begin{figure}[ht]
        \centering
        \includegraphics{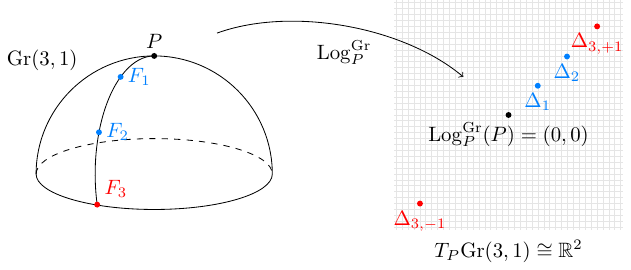}
        \caption{The manifold of one-dimensional subspaces of $\R^3$, i.e., $\Gr(3,1)$, can be seen as the upper half sphere with half of the equator removed. For points in the cut locus of a point $P \in \Gr(3,1)$ (like $F_3$ in the figure), there is no unique velocity vector in $T_P\Gr(3,1)$
        that sends a geodesic from $P$ to the point in question, but instead a set of two starting velocities ($\Delta_{3,+1}$ and $\Delta_{3,-1}$) that can be calculated according to Theorem\nobreakspace\ref{thm:multiple_shortest_geodesics}. Since the points actually mark one dimensional subspaces through the origin, $F_3$ is identical to its antipode on the equator.}
        \label{fig:cutlocus}
    \end{figure}
    Together, Theorem~\ref{thm:GrLog} and Theorem~\ref{thm:multiple_shortest_geodesics} allow to map any set of points on $\Gr(n,p)$ to a single tangent space. The situation of multiple tangent vectors that correspond to one and the same point in the cut locus is visualized in \autoref{fig:cutlocus}. Notice that if $r=1$ in Theorem~\ref{thm:multiple_shortest_geodesics}, there are only two possible geodesics $\gamma_{\pm 1}(t)$. For $r>1$ there is a smooth variation of geodesics.
    
    In \cite[Theorem 3.3]{BatziesHueperMachadoLeite2015} a closed formula for the logarithm for Grassmann locations represented by orthogonal projectors was derived.
    We recast this result in form of the following proposition. 
    \begin{proposition}[Grassmann Logarithm: Projector perspective]
        Let a point $P \in \Gr(n,p)$ and $F \in \Gr(n,p) \setminus\Cut_P$. Then $\Delta=[\Omega,P] \in \ID_P\subset T_P\Gr(n,p)$ such that $\Exp^\Gr_P([\Omega,P])=F$ is determined by
        \begin{equation*}
            \Omega= \frac{1}{2}\log_m\left((I_n-2F)(I_n-2P)\right) \in \so_P(n).
        \end{equation*}
        Consequently $\Log^\Gr_P(F)=[\Omega,P]$.
    \end{proposition}
    This proposition gives the logarithm explicitly, but it relies on 
    $n \times n$ matrices.
    Lifting the problem to the Stiefel manifold reduces the computational complexity.
    A method to compute the logarithm that uses an orthogonal completion of the Stiefel representative $U$ and the CS decomposition was proposed in \cite{Gallivan_etal2003}. 

\subsection{Numerical Performance of the Logarithm}
    In this section, we assess the numerical accuracy of Algorithm~\ref{alg:modgrasslog} as opposed to the algorithm introduced in \cite[Section 3.8]{AbsilMahonySepulchre2004}, for brevity hereafter referred to as the \emph{new log algorithm} and the \emph{standard log algorithm}, respectively.
    \begin{figure}
        \centering
         \includegraphics[width=.6\textwidth]{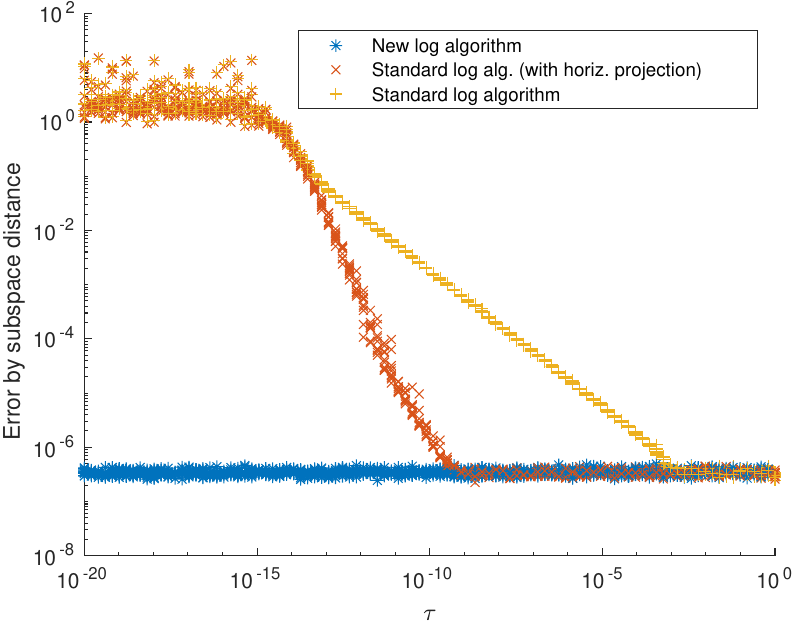}
        \caption{The error of the new log algorithm (blue stars) versus standard log algorithm with horizontal projection (red crosses) by subspace distance over $\tau$. For comparison, the error of the standard log algorithm without projection onto the horizontal space is also displayed (yellow plus). The cut locus is approached \revcomm{as $\tau$ goes to zero}.  It can be observed that the new log algorithm still produces reliable results close to the cut locus.}
        \label{fig:numericalLog_error}
    \end{figure}
    
    For a random subspace representative $U \in \St(1000,200)$ and a random horizontal tangent vector $\Delta^\hor_U \in \Hor_U\St(1000,200)$ with largest singular value set to $\frac\pi2$, the subspace representative
    \begin{equation*}
        U_1(\tau) = \Exp^\Gr_U\left((1-\tau)\Delta^\hor_U\right), \quad \tau \in [10^{-20},10^0],
    \end{equation*}
    is calculated. Observe that $U_1(0)U_1(0)^T$ is in the cut locus of $UU^T$. Then the logarithm $(\tilde{\Delta}(\tau))^\hor_U = \Log^\hor_U(U,U_1(\tau))$ is calculated according to the new log algorithm and the standard log algorithm, respectively. In the latter case, $(\tilde{\Delta}(\tau))^\hor_U$ is projected to the horizontal space $\Hor_U\St(1000,200)$ by \eqref{eq:projRHorUSt} to ensure $U^T(\tilde{\Delta}(\tau))^\hor_U = 0$. For $\tilde{U}_1(\tau)=\Exp^\Gr_U((\tilde{\Delta}(\tau))^\hor_U)$, the error is then calculated according to \eqref{eq:distGrass} as 
    \begin{equation*}
        \dist\left(U_1(\tau),\tilde{U}_1(\tau)\right) = \norm{\arccos(S)}_F,
    \end{equation*}
    where $QSR^T = U_1(\tau)^T\tilde{U}_1(\tau)$ is an SVD. Even though theoretically impossible, entries of values larger than one may arise in $S$ due to the finite machine precision. In order to catch such numerical errors, the real part $\Re(\arccos(S))$ is used in the actual calculations of the subspace distance.
    
    In Figure~\ref{fig:numericalLog_error}, the subspace distance between $U_1(\tau)$ and $\tilde{U}_1(\tau)$ is displayed for $100$ logarithmically spaced values $\tau$ between $10^{-20}$ and $10^0$. The Stiefel representative $\tilde{U}_1(\tau)$ is here calculated with the new log algorithm, with the standard log algorithm, and with the standard log algorithm with projection onto the horizontal space. This is repeated for $10$ random subspace representatives $U$ with random horizontal tangent vectors $\Delta^\hor_U$, and the results are plotted individually.
    As expected, Algorithm~\ref{alg:modgrasslog} shows favorable behaviour when approaching the cut locus. When the result of the standard log algorithm is not projected onto the horizontal space, it can be seen that its subspace error starts to increase already at $\tau \approx 10^{-3}$. The baseline error (in Figure~\ref{fig:numericalLog_error}) is due to the numerical accuracy of the subspace distance calculation procedure. \revcomm{The code to reproduce Figure~\ref{fig:numericalLog_error} can be found at \href{https://github.com/RalfZimmermannSDU/RiemannGrassmannLog}{github.com/RalfZimmermannSDU/RiemannGrassmannLog}.}
    
    Even though this experiment addresses the extreme-case behavior, it is of practical importance. In fact, the results of \cite{AbsilEdelmanKoev2006} show that for large-scale $n$ and two subspaces drawn from the uniform distribution on $\Gr(n,p)$, the largest principal angle between the subspaces is with high probability close to $\frac{\pi}{2}$.

 \section{Local Parameterizations of the Grassmann Manifold}
\label{sec:Parameterizations}
    In this section, we construct local parameterizations and coordinate charts of the Grassmannian.
   To this end, we work with the Grassmann representation as orthogonal projector $P = UU^T$.
    The dimension of $\Gr(n,p)$ is $(n-p)p$.
    Here, we recap how explicit local parameterizations from open subsets of $\R^{(n-p)\times p}$ onto open subsets of $\Gr(n,p)$ (and the corresponding coordinate charts) can be constructed.
    
    The Grassmannian $\Gr(n,p)$ can be parameterized by the so called \emph{normal coordinates} via the exponential map, which was also done in \cite{HelmkeHueperTrumpf2007}. Let $P=UU^T\in \Gr(n,p)$ and $U_\perp$ some orthogonal completion of $U \in \St(n,p)$. By making use of \eqref{eq:GrassmannTangentvectorUBUT}, a parameterization of $\Gr(n,p)$ around $P$ is given via
    \begin{equation*}
    \begin{split}
        \rho &\colon \R^{(n-p) \times p} \to \Gr(n,p),\\
        \rho(B)&:= \Exp^\Gr_P(U_\perp B U^T + U B^T U_\perp^T)\\
        &= \begin{pmatrix} U & U_\perp \end{pmatrix}\expm\left(\begin{pmatrix} 0 & -B^T\\ B & 0 \end{pmatrix}\right)P_0\expm\left(\begin{pmatrix} 0 & B^T\\ -B & 0 \end{pmatrix}\right)\begin{pmatrix} U & U_\perp \end{pmatrix}^T.
    \end{split}
    \end{equation*} 
    
    A different approach that avoids matrix exponentials, and which is also briefly introduced in \cite[Appendix C.4]{HelmkeMoore1994}, works as follows: Let $\mathcal{B} \subset \R^{(n-p)\times p}$ be an open ball around the zero-matrix $0\in \R^{(n-p)\times p}$ for some induced matrix norm $\|\cdot\|$. Consider
    \[
    \varphi\colon \mathcal{B} \rightarrow \R^{n\times n},\ B\mapsto 
    \begin{pmatrix}
    I_p\\
    B
    \end{pmatrix}
    (I_p + B^TB)^{-1} 
    \begin{pmatrix}
    I_p & B^T
    \end{pmatrix}.
    \]
    Note that $B$ is mapped to the orthogonal projector onto $\colspan(\begin{psmallmatrix}
    I_p\\ B
    \end{psmallmatrix})$, so that actually $\varphi(\mathcal{B})\subset \Gr(n,p)$.
    In particular, $\varphi(0) = P_0$.
    Let $P\in \Gr(n,p)$ be written block-wise as 
    $P = \begin{psmallmatrix}
        A & B^T\\
        B & C
        \end{psmallmatrix}
    $. Next, we show that the image of $\varphi$ is the set of such projectors $P$ with an invertible $p \times p$-block $A$ and that $\varphi$ is a bijection onto its image. To this end, assume that $A\in\R^{p\times p}$ has full rank $p$.
    Because $P$ is idempotent, it holds
    \[
    P = 
    \begin{pmatrix}
        A & B^T\\
        B & C
        \end{pmatrix} 
        =
        \begin{pmatrix}
        A^2 + B^TB & AB^T + B^TC\\
        BA + CB & BB^T + C^2
        \end{pmatrix}
    = P^2.
    \]
    As a consequence, $(I_p + A^{-1}B^TBA^{-1})^{-1} = A(\overbrace{A^2 + B^TB}^{=A})^{-1}A = A$.
    Moreover, since $p = \rank P = \rank A = \rank \begin{psmallmatrix}
        A \\
        B
        \end{psmallmatrix}$, the blocks 
    $\begin{psmallmatrix}
        B^T\\
        C
    \end{psmallmatrix}
    $
    can be expressed as a linear combination 
    $\begin{psmallmatrix}
        A\\
        B
        \end{psmallmatrix}X =
    \begin{psmallmatrix}
        B^T\\
        C
    \end{psmallmatrix}    
    $
    with $X\in\R^{p\times (n-p)}$.
    This shows that $X = A^{-1}B^T$ and $C = BA^{-1}B^T$.
    In summary,
    \begin{eqnarray*}
        P &=& 
        \begin{pmatrix}
            A & B^T\\
            B & BA^{-1}B^T
        \end{pmatrix}
        =
        \begin{pmatrix}
            I_p       \\
            BA^{-1} \\ 
        \end{pmatrix}
        A
        \begin{pmatrix}
            I_p & A^{-1}B^T
        \end{pmatrix}\\
        &=&
        \begin{pmatrix}
            I_p       \\
            BA^{-1} \\ 
        \end{pmatrix}
        (I_p + A^{-1}B^TBA^{-1})^{-1}
        \begin{pmatrix}
            I_p & A^{-1}B^T
        \end{pmatrix}
        = \varphi(BA^{-1}).
    \end{eqnarray*}
    Let 
    $
    \psi:  
    \begin{psmallmatrix}
        A & B^T\\
        B & C
    \end{psmallmatrix}
    \mapsto BA^{-1}$.
    Then, for any $B\in \mathcal{B}$, $(\psi \circ \varphi)(B) = B$ so that $\psi\circ \varphi = \id|_{\mathcal{B}}$.
    Conversely, for any $P\in \Gr(n,p)$ with full rank upper $(p\times p)$-diagonal block $A$,
    $(\varphi \circ \psi)(P) = P$.
    Therefore,
    $\varphi:\mathcal{B} \rightarrow \varphi(\mathcal{B})$ is a local parameterization
    around $0\in \R^{(n-p)\times p}$ and
    $x:= \psi|_{\varphi(\mathcal{B})}: \varphi(\mathcal{B}) \to \mathcal{B}$
    is the associated coordinate chart $x = \varphi^{-1}$.
    With the group action $\Phi$, we can move this local parameterization to obtain local parameterizations
    around any other point of $\Gr(n,p)$ via $\varphi_Q(B) := Q\varphi(B)Q^T$, which (re)establishes the fact that 
    $\Gr(n,p)$ is an embedded $(n-p)p$-dimensional submanifold of $\R^{n\times n}$.  

    The tangent space at $P$ is the image $\colspan(\D\varphi_P)$ for a suitable parameterization $\varphi$ around $P$.
    At $P_0$, we obtain
    \[
    T_{P_0}\Gr(n,p) =\{\D\varphi_{P_0}(B)\mid B\in\R^{(n-p)\times p} \}
    =\{ \begin{pmatrix}
        0 & B^T\\
        B &0
    \end{pmatrix}\mid B\in\R^{(n-p)\times p} \},
    \]
    in consistency with \eqref{eq:GrassmannTangentvectorB}.
    
    In principle, $\varphi$ and $\psi$ can be used as a replacement for the 
    Riemannian exp- and log-mappings in data processing procedures. For example, for a set of data points contained in 
    $\varphi(\mathcal{B})\subset \Gr(n,p)$, Euclidean interpolation can be performed on the coordinate images in $\mathcal{B}\subset \R^{(n-p)\times p}$.
    Likewise, for an objective function $f:\Gr(n,p)\supset \mathcal{D}\to \R$ with domain $\mathcal{D}\subset \varphi(\mathcal{B})$, the
    associated function $f\circ \varphi: \R^{(n-p)\times p}\supset \varphi^{-1}(\mathcal{D})\to \R$ can be optimized relying entirely on
    standard Euclidean tools; no evaluation of neither matrix exponentials nor matrix logarithms is required.
    Yet, these parameterizations do not enjoy the metric special properties of the Riemannian normal coordinates.
    Another reason to be wary of interpolation in coordinates is that the values on the Grassmannian will never leave $\varphi(\mathcal{B})$, and this can be very unnatural for some data sets. Furthermore, the presence of a domain $\mathcal{D}$ can be unnatural, as $\varphi(\mathcal{B})$ is an open subset of $\Gr(n,p)$, whereas the whole Grassmannian is compact, a desirable property for optimization. If charts are a switched, then information gathered by the solver may lose interest. Nevertheless, working in charts can be a successful approach~\cite{Usevich2014}.
    
\section{Jacobi Fields and Conjugate Points}
\label{sec:conjugatepoints}
    In this section, we describe Jacobi fields vanishing at one point and the conjugate locus of the Grassmannian. \emph{Jacobi fields} are vector fields along a geodesic fulfilling the Jacobi equation \eqref{eq:JacobiEquation}. They can be viewed as vector fields pointing towards another ``close-by'' geodesic, see for example \cite[\revcomm{Chapter 10}]{Lee2018riemannian}. The \emph{conjugate points} of $P$ are all those $F \in \Gr(n,p)$ such that there is a non-zero Jacobi field along a (not necessarily minimizing) geodesic from $P$ to $F$, which vanishes at $P$ and $F$. The set of all conjugate points of $P$ is the \emph{conjugate locus} of $P$. In general, there are not always multiple distinct (possibly non-minimizing) geodesics between two conjugate points, but on the Grassmannian there are. The conjugate locus on the Grassmannian was first treated in \cite{Wong1968}, but the description there is not complete. This is for example pointed out in \cite{Sakai1977} and \cite{Berceanu1997}. The latter gives a description of the conjugate locus in the complex case, which we show can be transferred to the real case.
    
    Jacobi fields and conjugate points are of interest when variations of geodesics are considered. They arise for example in geodesic regression \cite{fletcher2013geodesic} and curve fitting problems on manifolds \cite{BergmannGousenbourger2018}.

\subsection{Jacobi Fields}
\label{subsec:jacobifields}
    A \emph{Jacobi field} is smooth vector field $J$ along a geodesic $\gamma$ satisfying the ordinary differential equation 
    \begin{equation}
    \label{eq:JacobiEquation}
        D_t^2J+R(J,\dot\gamma)\dot\gamma=0,
    \end{equation}
    called \emph{Jacobi equation}. Here $R(\cdot,\cdot)$ is the curvature tensor and $D_t$ denotes the covariant derivative along the curve $\gamma$. This means that for every extension $\hat{J}$ of $J$, which is to be understood as a smooth vector field on a neighborhood of the image of $\gamma$ that coincides with $J$ on $\gamma(t)$ for every $t$, it holds that $(D_t J)(t) = \nabla_{\dot\gamma(t)}\hat{J}$.  For a detailed introduction see for example \cite[Chapter 10]{Lee2018riemannian}. A Jacobi field is the variation field of a variation through geodesics. That means intuitively that $J$ points from the geodesic $\gamma$ to a ``close-by'' geodesic, and, by linearity and scaling, to a whole family of such close-by geodesics. Jacobi fields that vanish at a point can be explicitly described via \cite[Proposition 10.10]{Lee2018riemannian}, which states that the Jacobi field $J$ along the geodesic $\gamma$, with $\gamma(0)=p$ and $\dot\gamma(0)=v$, and initial conditions $J(0)=0 \in T_pM$ and $D_t J(0)=w \in T_v(T_pM) \cong T_pM$ is given by
    \begin{equation}
        \label{eq:vanishingJacobigeneral}
        J(t)=\D(\exp_p)_{tv}(tw).
    \end{equation}
    The concept is visualized in Figure~\ref{fig:Jacobifield}.
    \begin{figure}[ht]
        \centering
        \includegraphics{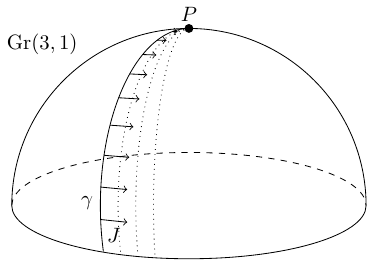}
        \caption{The Jacobi field $J$ points from the geodesic $\gamma$ towards close-by geodesics (dotted) and vanishes at $P$. Note that $J(t) \in T_{\gamma(t)}\Gr(3,1)$ is a tangent vector and not actually the offset vector between points on the respective geodesics in $\R^3$. Nevertheless, $J$ is the variation field of a variation of $\gamma$ through geodesics, c.f. \cite[Proposition 10.4]{Lee2018riemannian}.}
        \label{fig:Jacobifield}
    \end{figure}
    
    By making use of the derivative of the exponential mapping derived in Proposition\nobreakspace\ref{prop:derivative_Grassmannexp}, we can state the following proposition for Jacobi fields vanishing at a point on the Grassmannian.
    
    \begin{proposition}
    \label{prop:Jacobifields}
        Let $P=UU^T\in \Gr(n,p)$ and let $\Delta_1,\Delta_2 \in T_P\Gr(n,p)$ be two tangent vectors, where the singular values of $(\Delta_1)^\hor_U$ are mutually distinct and non-zero. Define the geodesic $\gamma$ by $\gamma(t):=\Exp_P^\Gr(t\Delta_1)$. Furthermore, let $(t\Delta_1)^\hor_U=Q(t\Sigma)V^T$ and $(t(\Delta_1+s\Delta_2))^\hor_U=Q(s)(t\Sigma(s))V(s)^T$
        be given via the compact SVDs of the horizontal lifts, i.e., $Q(s) \in \St(n,p)$, $\Sigma(s)=\diag(\sigma_1(s),\dots,\sigma_p(s))$ and $V(s) \in \O(p)$, as well as $Q(0)=Q, \Sigma(0)=\Sigma$ and $V(0)=V$.
        Finally, define
        \begin{equation*}
            Y(t) := UV \cos(t\Sigma) + Q\sin(t\Sigma) \in \St(n,p)
        \end{equation*} 
        and
        \begin{equation*}
            \Gamma(t):= U \dot{V}\cos(t\Sigma) - tUV\sin(t\Sigma)\dot{\Sigma} + \dot{Q}\sin(t\Sigma)+ t Q\cos(t\Sigma)\dot{\Sigma} \in T_{Y(t)}\St(n,p).\footnote{As in the case of the derivative of the Grassmann exponential, the matrices $\dot Q=\frac{\D Q}{\D s}(0), \dot \Sigma =\frac{\D \Sigma}{\D s}(0)$ and $\dot V=\frac{\D V}{\D s}(0)$ can be calculated by Algorithm~\ref{alg:SVDdiff}.}
        \end{equation*}
        Then the Jacobi field $J$ along $\gamma$ fulfilling $J(0)=0$ and $D_tJ(0)=\Delta_2$ is given by
        \begin{equation*}
            J(t)=\Gamma(t)Y(t)^T+Y(t)\Gamma(t)^T \in T_{\gamma(t)}\Gr(n,p).
        \end{equation*}
        The horizontal lift of $J(t)$ to $Y(t)$ is accordingly given by 
        \begin{equation*}
            \left(J(t)\right)^\hor_{Y(t)}=\Gamma(t) + Y(t)\Gamma(t)^TY(t) = (I_n-Y(t)Y(t)^T)\Gamma(t).
        \end{equation*}
        It is the variation field of the variation of $\gamma$ through geodesics given by $\Xi(s,t):=\Exp_P^\Gr(t(\Delta_1+s\Delta_2))$. 
        \begin{proof}
            The proof works analogously to the one of Proposition\nobreakspace\ref{prop:derivative_Grassmannexp}, since according to \eqref{eq:vanishingJacobigeneral}
            \begin{equation*}
                J(t)=\D (\Exp_P^\Gr)_{t\Delta_1}(t\Delta_2)=\frac{\D}{\D s}\Big\vert_{s=0} \Exp_P^\Gr(t(\Delta_1 + s\Delta_2)).
            \end{equation*}
        \end{proof}
    \end{proposition}

\subsection{Conjugate Locus}
    In the following, let $p \leq \frac{n}{2}$. The reason for this restriction is that for $p > \frac{n}{2}$ there are automatically principal angles equal to zero, yet these do not contribute to the conjugate locus, as one can see by switching to the orthogonal complement. We will see that the conjugate locus $\Conj_P$ of $P \in \Gr(n,p)$ is given by all $F \in \Gr(n,p)$ such that at least two principal angles between $P$ and $F$ coincide, or there is at least one principal angle equal to zero if $p<\frac{n}{2}$. This obviously includes the case of two or more principal angles equal to $\frac{\pi}{2}$. In the complex case, the conjugate locus also includes points with one principal angle of $\frac{\pi}{2}$, as is shown in \cite{Berceanu1997}. Only in the cases of principal angles of $\frac{\pi}{2}$ is there a nontrivial Jacobi field vanishing at $P$ and $F$ along a \emph{shortest} geodesic. It can be calculated from the variation of geodesics as above. In the other cases, the shortest geodesic is unique, but we can smoothly vary longer geodesics from $P$ to $F$. This variation is possible because of the periodicity of sine and cosine and the indeterminacies of the SVD.
    \begin{theorem}
    \label{thm:conjugatelocus}
        Let $P \in \Gr(n,p)$ where $p \leq \frac{n}{2}$. The conjugate locus $\Conj_P$ of $P$ consists of all points $F \in \Gr(n,p)$ with at least two identical principal angles or, when $p < \frac{n}{2}$, at least one zero principal angle between $F$ and $P$.
        \begin{proof}
          Let $P$ and $F$ have $r=j-i+1$ repeated principal angles $\sigma_i = \dots = \sigma_j$. Obtain $\Delta^\hor_U=\hat{Q}\Sigma R^T$ by Algorithm \ref{alg:modgrasslog}. Define $\Sigma'$ by adding $\pi$ to one of the repeated angles. Then for every $D \in \O(r)$ and $\tilde{D}:=\diag(I_{i-1},D,I_{p-j})$, the curve
         \begin{equation*}
            \gamma_D(t)=\revcomm{\pi^\SG\Big(}UR\tilde{D}\cos(t\Sigma')\tilde{D}^TR^T + \hat{Q}\tilde{D}\sin(t\Sigma')\tilde{D}^TR^T\revcomm{\Big)}
         \end{equation*}
         is a geodesic from $P$ to $\gamma_D(1)=F$\revcomm{, with projection $\pi^\SG$ from~\eqref{eq:piSG}}. Since for $0<t<1$ the matrix $\cos(t\Sigma')$ does not have the same number of repeated diagonal entries as $\cos(t\Sigma)$, not all curves $\gamma_D$ coincide. Then we can choose an open interval $\mathcal{I}$ around $0$ and a smooth curve $D \colon \mathcal{I} \to \O(r)$ with $D(0) = I_r$ such that $\Gamma(s,t) = \gamma_{D(s)}(t)$ is a variation through geodesics as defined in \cite[Chap. 10, p. 284]{Lee2018riemannian}. The variation field $J$ of $\Gamma$ is a Jacobi field along $\gamma_{D(0)}$ according to \cite[Theorem 10.1]{Lee2018riemannian}. Furthermore, $J$ is vanishing at $t=0$ and $t=1$, as $\gamma_{D(s)}(1)=\gamma_{D(\tilde{s})}(1)$ for all $s,\tilde s \in \mathcal{I}$ by Proposition \ref{prop:ambiguitySVD}, and likewise for $t=0$. Since $J$ is not constantly vanishing, $P$ and $F$ are conjugate along $\gamma_{D(0)}$ by definition.
         
         When $p<\frac{n}{2}$ and there is at least one principal angle equal to zero, there is some additional freedom of variation. Let the last $r$ principal angles between $P$ and $F$ be $\sigma_{p-r+1}=\dots=\sigma_p=0$. Obtain $\Delta^\hor_U=\hat{Q}\Sigma R^T$ by Algorithm~\ref{alg:modgrasslog}. Since $p<\frac{n}{2}$, $\hat{Q}$ can be chosen such that $U^T\hat{Q}=0$, and there is at least one unit vector $\hat{q} \in \R^n$, such that $\hat{q}$ is orthogonal to all column vectors in $U$ and in $\hat{Q}$. Let $\hat{Q}_\perp$ be an orthogonal completion of $\hat{Q}$ with $\hat{q}$ as its first column vector. Define $\Sigma'$ as the matrix $\Sigma$ with $\pi$ added to the $(p-r+1)$th diagonal entry. Then for every $W \in \O(2)$,
        \begin{equation*}
            \gamma_{W}(t)=\revcomm{\pi^\SG\Bigg(}UR\cos(t\Sigma')+\begin{pmatrix} \hat{Q} & \hat{Q}_\perp\end{pmatrix}\diag(I_{p-r},W,I_{n-p+r-2})\begin{pmatrix}\sin(t\Sigma')\\ 0\end{pmatrix}\revcomm{\Bigg)}
        \end{equation*}
        is a geodesic from $P$ with $\gamma_{W}(1)=F$. With an argument as above, $P$  and $F$ are conjugate along $\gamma_{I_2}$.
        
        There are no other points in the conjugate locus than those with repeated principal angles (or one zero angle in case of $p<\frac{n}{2}$), as the SVD is unique (up to order of the singular values) for matrices with no repeating and no zero singular values. As every geodesic on the Grassmannian is of the form \eqref{eq:GrassmannExpSt}, the claim can be shown by contradiction.
        \end{proof}
    \end{theorem}
    By construction, the length of $\gamma_{D(0)}$ between $P$ and $F$ is longer than the length of the shortest geodesic, since $\norm{\Sigma}_F < \norm{\Sigma'}_F$. The same is true for the case of a zero angle. It holds that the cut locus $\Cut_P$ is no subset of the conjugate locus $\Conj_P$, since points with just one principal angle equal to $\frac{\pi}{2}$ are not in the conjugate locus. Likewise the conjugate locus is no subset of the cut locus. The points in the conjugate locus that are conjugate along a \emph{minimizing} geodesic however are also in the cut locus, as those are exactly those with multiple principal angles equal to $\frac{\pi}{2}$. 
    \begin{remark}
        The (incomplete) treatment in \cite{Wong1968} covered only the cases of at least two principal angles equal to $\frac{\pi}{2}$ or principal angles equal to zero, but not the cases of repeated arbitrary principal angles. We can nevertheless take from there that for $p>\frac{n}{2}$ we need at least $2p-n+1$ principal angles equal to zero, instead of just one as for $p<\frac{n}{2}$. Points with repeated (nonzero) principal angles are however always in the conjugate locus, as the proof of Theorem \ref{thm:conjugatelocus} still holds for them.
    \end{remark}

\section{Conclusion}
\label{sec:Conclusion}
In this work, we have collected the facts and formulae that we deem most important for Riemannian computations on the Grassmann manifold.
This includes in particular explicit formulae and algorithms for computing local coordinates, the Riemannian normal coordinates (the Grassmann exponential and logarithm mappings), the Riemannian connection, the parallel transport of tangent vectors and the sectional curvature.
All these concepts may appear as building blocks 
or tools for the theoretical analysis of, e.g., optimization problems, interpolation problems and, more generally speaking, data processing problems such as data averaging or clustering.

We have treated the Grassmannian both as a quotient manifold of the orthogonal group and the Stiefel manifold, and as the space of orthogonal projectors of fixed rank and have exposed (and exploited) the connections between these view points.
While concepts from differential geometry arise naturally in the theoretical considerations, care has been taken that the final formulae are purely matrix-based and thus are fit for immediate use in algorithms.
At last, the paper features an original approach to computing the Grassmann logarithm, which simplifies the theoretical analysis, extends its operational domain and features improved numerical properties.
Eventually, this tool allowed us to conduct a detailed investigation of shortest curves to cut points as well as studying the conjugate points on the Grassmannian by basic matrix-algebraic means. These findings are more explicit and more complete than the previous results in the research literature.

\appendix
\section{Basics from Riemannian Geometry}
\label{sec:RiemannGeo}
    For the reader's convenience, we recap some fundamentals from Riemannian geometry. 
    Concise introductions can be found in \cite[Appendices C.3, C.4, C.5]{HelmkeMoore1994}, \cite{Gallier2011} and \cite{AbsilMahonySepulchre2008}.
    For an in-depth treatment, see for example
    \cite{DoCarmo2013riemannian, KobayashiNomizu1996, Lee2018riemannian}.

An {\em $n$-dimensional differentiable manifold} $\mcM$ is a topological space $\mcM$ such that for every point $p\in \mcM$,
there exists a so-called {\em coordinate chart} $x:\mcM\supset \mathcal{D}_p \rightarrow \R^n$ that bijectively maps an open neighborhood $\mathcal{D}_p\subset \mcM$ of a location $p$ to an open neighborhood $D_{x(p)}\subset \R^n$ around $x(p)\in \R^n$ 
with the additional property that the {\em coordinate change} 
$$x\circ \tilde{x}^{-1}: \tilde{x}(\mathcal{D}_p\cap\tilde{\mathcal{D}}_p) \rightarrow x(\mathcal{D}_p\cap\tilde{\mathcal{D}}_p)$$
of two such charts
$x,\tilde{x}$ is a diffeomorphism, where their domains of definition overlap, see \cite[Fig. 18.2, p. 496]{Gallier2011}.
This enables to transfer the most essential tools from calculus
to manifolds.
An {\em n-dimensional submanifold of $\R^{n+d}$} is a subset $\mcM\subset \R^{n+d}$ that can be locally smoothly straightened, i.e.\revcomm{,} satisfies the local $n$-slice condition \cite[Thm. 5.8]{Lee2012smooth}.
\begin{theorem}[{\cite[Prop. 18.7, p. 500]{Gallier2011}}]
\label{thm:regurbild}
 Let $h: \R^{n+d}\supset \Omega \rightarrow \R^{d}$ be differentiable and $c_0\in \R^d$ be defined such that 
 the differential $Dh_p\in \R^{d\times (n+d)}$ has maximum possible rank $d$ at every point $p\in \Omega$ with $h(p) = c_0$.
 Then, the preimage
 \[
  h^{-1}(c_0) = \{p\in \Omega \mid \hspace{0.1cm} h(p) = c_0\}
 \]
 is an $n$-dimensional submanifold of $\R^{n+d}$.
\end{theorem}
This theorem establishes the Stiefel manifold 
$\St(n,p) = \left\lbrace U\in\R^{n\times p}\ \middle| \ U^TU=I\right\rbrace$ as an embedded submanifold of $\R^{n\times p}$, since $\St(n,p) = F^{-1}(I)$ for $F:U\mapsto U^TU$.

\paragraph{Tangent Spaces}
 The {\em tangent space} of a submanifold $\mcM$ at a point $p\in \mcM$, in symbols $T_p\mcM$, is the space of 
 velocity vectors of differentiable curves $c:t \mapsto c(t)$ passing through $p$, i.e.,
 \[ 
  T_p\mcM = \{\dot{c}(t_0)\mid \hspace{0.1cm}  c:I\rightarrow \mcM,   \hspace{0.1cm} c(t_0)=p\}.
 \]
%------------------------------------------------------
%
%
The tangent space is a vector space of the same dimension $n$ as the manifold $\mcM$.
\paragraph{Geodesics and the Riemannian Distance Function}
\label{sec:diffgeo_geodesics}
Riemannian metrics measure the lengths and angles
between tangent vectors. Eventually, this allows to measure 
the lengths of curves on a manifold and the Riemannian distance
between two manifold locations.

  A {\em Riemannian metric} on $\mcM$ is a family $(\revcomm{g_p(\cdot,\cdot)})_{p\in\mcM}$ of inner products $\revcomm{g_p(\cdot,\cdot)}: T_p\mcM\times T_p\mcM\rightarrow \R$ that is smooth in variations of the base point $p$\revcomm{, or more precisely, a smooth covariant 2-tensor field, c.f.~\cite[Chapter 2]{Lee2018riemannian}}.
  The {\em length} of a tangent vector $v\in T_p\mcM$ is $\|v\|_p := \sqrt{\revcomm{g_p(v,v)}}$.
  The length of a curve $c:[a,b] \rightarrow \mcM$ is defined as $$L(c) = \int_a^b \|\dot{c}(t)\|_{c(t)} dt  = \int_a^b \sqrt{\revcomm{g_{c(t)}(\dot{c}(t),\dot{c}(t))}} dt.$$\\
  A curve is said to be {\em parameterized by the arc length}, if $L(c|_{[a,t]}) = t-a$ for all $t\in [a,b]$. Obviously, {\em unit-speed curves} with $\|\dot{c}(t)\|_{c(t)}\equiv 1$ are parameterized by the arc length.
  Constant-speed curves with $\|\dot{c}(t)\|_{c(t)}\equiv \nu_0$ are parameterized proportional to the arc length.
  The {\em Riemannian distance} between two points $p,q\in \mcM$ with respect to a given metric is
  \begin{equation}
\label{eq:RiemannDist}
  \dist_{\mcM}(p,q) = \inf\{L(c)\mid c:[a,b]\rightarrow \mcM \mbox{ piecewise smooth, } c(a)=p, c(b)=q\},
  \end{equation}
  where, by convention, $\inf\{\emptyset\} =\infty$. 
A shortest path between $p,q\in \mcM$ is a curve $c$ that connects $p$ and $q$ such that 
$L(c) = \dist_{\mcM}(p,q)$.
Candidates for shortest curves between points 
are called {\em geodesics} and are
characterized by a differential equation:
 A differentiable curve $c:[a,b]\rightarrow \mcM$ is a geodesic (w.r.t. to a given Riemannian metric),
 if the {\em  covariant derivative} of its velocity vector field vanishes, i.e., 
 \begin{equation}
 \label{eq:geodesic}
  \frac{D\dot{c}}{dt}(t) = 0 \quad \forall t\in [a,b].
 \end{equation}
Intuitively, the covariant derivative can be thought of
as the standard derivative (if it exists) followed by a 
point-wise projection onto the tangent space.
In general, a covariant derivative, also known as a {\em linear connection},
is a bilinear mapping $(X,Y) \mapsto \nabla_XY$ that maps two vector fields $X,Y$ to a third vector field
$\nabla_XY$ in such a way that it can be interpreted as the directional derivative of $Y$ in the direction of $X$, \cite[\S 4, \S 5]{Lee2018riemannian}.
Of importance is the {\em Riemannian connection} or {\em Levi-Civita connection} that is compatible with a Riemannian metric \cite[Thm 5.3.1]{AbsilMahonySepulchre2008}, \cite[Thm 5.10]{Lee2018riemannian}. 
It is determined uniquely by the Koszul formula
\begin{eqnarray*}
   \label{eq:koszul}
   2\revcomm{g(\nabla_XY, Z)}  &=& X(\revcomm{g( Y, Z )} ) + Y(\revcomm{g( Z,X)} ) - Z(\revcomm{g(X, Y )})\\
   \nonumber
   && - \revcomm{g( X, [Y,Z] )} - \revcomm{g( Y, [X,Z] )} + \revcomm{g( {Z}, {[X,Y]})}
\end{eqnarray*}
and is used to define the {\em Riemannian curvature tensor}
$$
(X,Y,Z)\mapsto R(X,Y)Z=\nabla_X\nabla_YZ-\nabla_Y\nabla_XZ-\nabla_{[X,Y]}Z.\footnote{In these formulae, $[X,Y] = X(Y)-Y(X)$ is the Lie bracket of two vector fields.}
$$
A Riemannian manifold is flat if and only if it is locally isometric to the Euclidean space, which holds if and only if the Riemannian curvature tensor vanishes identically \cite[Thm. 7.10]{Lee2018riemannian}.

\paragraph{Lie Groups and Orbits}
\label{app:LieGroups}
    A \emph{Lie group} is a smooth manifold that is also a group with smooth multiplication and inversion. A {\em matrix Lie group} $G$ is a subgroup of the general linear group $GL(n,\C)$ that is closed in $GL(n,\C)$ (but not necessarily in the ambient space $\C^{n\times n}$). Basic examples include  $GL(n,\R)$ and the orthogonal group $\O(n)$.
    Any matrix Lie group $G$ is automatically an embedded submanifold of $\C^{n\times n}$ \cite[Corollary 3.45]{Hall_Lie2015}. The tangent space $T_IG$ of $G$ at the identity $I\in G$ has a special role. When endowed with the bracket operator or {\em matrix commutator} $[V,W] = VW-WV$ for  $V,W \in T_IG$, the tangent space becomes an algebra, called the {\em Lie algebra} associated with the Lie group $G$, see \cite[\S 3]{Hall_Lie2015}. As such, it is denoted by $\mathfrak{g} = T_IG$.
    For any $A\in G$, the function ``left-multiplication with $A$'' is a diffeomorphism
    $L_A\colon G\to G,\ L_A(B) = AB$; its differential at a point $B\in G$ is the isomorphism
    $\D (L_A)_B\colon T_BG\to T_{L_A(B)}G,\ \D (L_A)_B(V) = AV$. Using this observation at $B=I$ shows that the
    tangent space at an arbitrary location $A\in G$ is given by the translates (by left-multiplication) of the 
    tangent space at the identity \cite[\S 5.6, p. 160]{godement2017introduction},
    \begin{equation}
    \label{eq:tangspaceshifts}
      T_AG = T_{L_A(I)}G =  A \mathfrak{g} = \left\{\Delta = AV\in \R^{n\times n}|\hspace{0.2cm} V\in \mathfrak{g}\right\}.
    \end{equation}
    
    A \emph{smooth left action} of a Lie group $G$ on a manifold $M$ is a smooth map $\phi \colon G \times M \to M$ fulfilling $\phi(g_1,\phi(g_2,p))=\phi(g_1g_2,p)$ and $\phi(e,p)=p$ for all $g_1,g_2 \in G$ and all $p \in M$, where $e \in G$ denotes the identity element. One often writes $\phi(g,p) = g \cdot p$. For each $p \in M$, the \emph{orbit of $p$} is defined as
    \begin{equation}
        G \cdot p := \{ g \cdot p \mid g \in G \},
    \end{equation}
    and the \emph{stabilizer of $p$} is defined as
    \begin{equation}
        G_p := \{ g \in G \mid g \cdot p = p\}.
    \end{equation}
    For a detailed introduction see for example \cite[\revcomm{Chapters 7 \& 21}]{Lee2012smooth}. We need the following well known result, see for example \cite[Section 2.1]{HelmkeMoore1994}, where the quotient manifold $G/G_p$ refers to the set $\{g G_p \mid g \in G\}$ endowed with the unique manifold structure that turns the quotient map $g\mapsto g G_p$ into a submersion.
    \begin{proposition}
        \label{prop:orbitmanifold}
        Let $G$ be a compact Lie group acting smoothly on a manifold $M$. Then for any $p \in M$, the orbit $G \cdot p$ is an embedded submanifold of $M$ that is diffeomorphic to the quotient manifold $G/G_p$.
        \begin{proof}
            The continuous action of a compact Lie group is always proper, \cite[Corollary 21.6]{Lee2012smooth}. Therefore \cite[Proposition 3.41]{Alexandrino2015} shows the claim. 
        \end{proof}
    \end{proposition}

\section{Matrix Analysis Necessities}
\label{app:MatrixBasics}
Throughout, we consider the matrix space $\R^{m\times n}$ as a Euclidean vector space with the standard metric 
\begin{equation}
 \label{eq:EuclideanMetric}
 \langle A,B\rangle_0 = \tr(A^TB).
\end{equation}
Unless noted otherwise, the singular value decomposition (SVD) of a matrix $X\in\R^{m\times n}$
is understood to be the compact SVD
\[
    X = U\Sigma V^T, \quad U\in \R^{m\times n}, \Sigma, V \in \R^{n\times n}.
\]
The SVD is not unique.
    \begin{proposition}[Ambiguity of the Singular Value Decomposition]\cite[Theorem 3.1.1']{HornJohnson1991}
    \label{prop:ambiguitySVD}
        Let $X \in \R^{m \times n}$ have a (full) SVD $X=U \Sigma V^T$ with singular values in descending order and $\rank(X)=r$. Let $\sigma_1>\dots >\sigma_k>0$ be the distinct nonzero singular values with respective multiplicity $\mu_1,\dots,\mu_k$. Then $X=\tilde{U}\Sigma \tilde{V}^T$ is another SVD if and only if $\tilde{U}=U\diag(D_1,\dots,D_k,W_1)$ and $\tilde{V}=V\diag(D_1,\dots,D_k,W_2)$, with $D_i \in \O(\mu_i)$, $W_1 \in \O(m-r)$, and $W_2 \in \O(n-r)$ arbitrary.
    \end{proposition}
%
%%
%%%
\paragraph{Differentiating the Singular Value Decomposition}
\label{app:diffSVD}
Let $p\leq n\in \N$ and suppose that $t \mapsto Y(t)\in \R^{n\times p}$ is a differentiable matrix curve around $t_0\in \R$.
If the singular values of $Y(t_0)$ are mutually distinct and non-zero, then
the singular values and both the left and the right singular vectors 
depend differentiable on $t \in [t_0 -\delta t, t_0+\delta  t]$
for $\delta t$ small enough. 

Let $ t \mapsto Y(t) = U( t)\Sigma( t) V( t)^T \in \R^{n\times p}$, 
where $U(t)\in \St(n,p)$, $V( t)\in O(p)$ and $\Sigma( t)\in\R^{p\times p}$ diagonal and positive definite.
Let $u_j$ and $v_j$, $j=1,\ldots,p$ denote the columns of $U( t_0)$ and $V( t_0)$, respectively.
For brevity, write $Y = Y(t_0), \dot{Y} = \frac{\D}{\D t}\big\vert_{t=t_0}Y(t)$, likewise for the other matrices that feature in the SVD.
The derivatives of the matrix factors of the SVD can be calculated with Alg. \ref{alg:SVDdiff}. A proof can for example be found in \cite{HayBorggaardPelletier2009,DieciEirola1999}.
\begin{algorithm}
\caption{Differentiating the SVD}
\label{alg:SVDdiff}
\begin{algorithmic}[1]
  \Require{Matrices $Y, \dot{Y}\in\R^{n\times p}$, (compact) SVD $Y = U\Sigma V^T$.}
  \State{ $\dot{\sigma}_j = (u_j)^T \dot{Y} v_j\mbox{ for }j = 1,\ldots, p$}
  \State{ $\dot{V} =V \Gamma, \mbox{ where } \Gamma_{ij} =
            \left\{
                \begin{array}{ll}
                 \frac{ \sigma_i (u_i^T \dot{Y} v_j) +  \sigma_j(u_j^T\dot{Y} v_i) }{(\sigma_j + \sigma_i)(\sigma_j - \sigma_i)}, & i\neq j\\
                0,          & i=j
                \end{array}
            \right.  \mbox{ for }i,j = 1,\ldots, p$}
  \State{ $\dot{U} = \left(\dot{Y} V + U(\Sigma \Gamma - \dot{\Sigma}) \right) \Sigma^{-1}.$}
  \Ensure{$\dot{U}, \dot\Sigma=\diag(\dot\sigma_1,\ldots,\dot\sigma_m), \dot{V}$}
\end{algorithmic}
\end{algorithm}

\paragraph{Differentiating the QR-Decomposition}
\label{app:diffQR}
Let $t\mapsto Y(t)\in \R^{n\times r}$ be a differentiable matrix function with Taylor expansion
$Y(t_0+ h) = Y(t_0) + h \dot{Y}(t_0) + \mathcal{O}(h^2)$.
Following \cite[Proposition 2.2]{WalterLehmannLamour2012}, the QR-decomposition is characterized via the following set of matrix equations.
\[
 Y(t)  = Q(t)R(t),\quad Q^T(t)Q(t) = I_r, \quad 0 = P_L\odot R(t).
\]
In the latter, 
$P_L =
    \begin{psmallmatrix}
        0      & \cdots      &\cdots & 0\\
        1      & \ddots &  & \vdots\\
        \vdots & \ddots &\ddots &\vdots \\
        1      & \cdots &1      & 0
    \end{psmallmatrix}
$
and `$\odot$' is the element-wise matrix product so that
$P_L\odot R$ selects the strictly lower triangle of the square matrix $R$.
For brevity, we write $Y= Y(t_0),\ \dot Y = \frac{\D}{\D t}\big\vert_{t=t_0}Y(t)$, likewise for $Q(t)$, $R(t)$.
By the product rule
\[
 \dot{Y} = \dot{Q} R + Q\dot{R}, \quad 0 =\dot{Q}^TQ + Q^T\dot{Q}, \quad 0  = P_L\odot \dot{R}.
\]
According to \cite[Proposition 2.2]{WalterLehmannLamour2012}, the derivatives $\dot{Q}, \dot{R}$
can be obtained from Alg. \ref{alg:QRdiff}.
The trick is to compute $X = Q^T\dot{Q}$ first and then use this to compute
$\dot{Q} = QQ^T\dot{Q} + (I_n-QQ^T)\dot{Q}$ by exploiting that $Q^T\dot{Q}$ is skew-symmetric
and that $\dot{R}R^{-1}$ is upper triangular.
\begin{algorithm}
\caption{Differentiating the QR-decomposition, \cite[Proposition 2.2]{WalterLehmannLamour2012}}
\label{alg:QRdiff}
\begin{algorithmic}[1]
  \Require{\revcomm{M}atrices $T, \dot{T}\in\R^{n\times r}$, (compact) QR-decomposition $T=QR$.}
  \State{ $L:= P_L\odot(Q^T\dot{T}R^{-1})$}
  \State{ $X = L-L^T$} \hfill \Comment{Now, $X = Q^T\dot{Q}$}
  \State{ $\dot{R} =  Q^T\dot{T} - XR$}
  \State{ $\dot{Q} = (I_n-QQ^T)\dot{T}R^{-1} + QX$}
  \Ensure{$\dot{Q}, \dot{R}$}
\end{algorithmic}
\end{algorithm}

%%
%%%
\paragraph{Matrix Exponential and the Principal Matrix Logarithm}
The matrix exponential and the principal matrix logarithm are defined by
\begin{equation}
 \label{eq:MatrxiExpLog}
 \exp_m(X):=\sum_{j=0}^\infty{\frac{X^j}{j!}}, \quad \log_m(I+X):=\sum_{j=1}^\infty{(-1)^{j+1}\frac{X^j}{j}}.
\end{equation}
The latter is well-defined for matrices that have no eigenvalues on $\R^-$.

\section{\revcomm{Computational Complexity}}
\revcommblock{
    For the benefit of the reader, we include  Table~\ref{tab:flops} of the floating point operation (FLOP) counts of some of the most commonly used formulas in this handbook. Note that the FLOP count of the SVD and other operations depends on the specific implementation. Furthermore, we counted $\sin(\cdot)$, $\cos(\cdot)$, $\sqrt{\cdot}$ etc. for scalars as one flop for simplicity.
\begin{table}[H]
    \begin{tabular}{l<{\hspace{-1em}}r>{\hspace{-1em}}l>{\hspace{-1em}}c}
        \toprule
        \bfseries{Operation} & \multicolumn{2}{c}{\hspace{-5em}\bfseries{Formula}} & \bfseries{FLOPS}\\
        \midrule
        Riem. metric & $g_{UU^T}^\Gr(\Delta_1,\Delta_2) =$&$\tr((\Delta^\hor_{1,U})^T\Delta^\hor_{2,U})$ & $2np^2 + p$\\
        Riem. gradient & $(\grad \bar{f})_U=$&$(I-UU^T)\grad^{\eucl} \bar{f}_U$ & $4np^2 + np$\\
        Riem. exponential & $\Exp_U^\Gr(t\Delta^\hor_U)=$&$U\tilde{V}\cos(t\tilde{\Sigma})\tilde{V}^T + \tilde{Q}\sin(t\tilde{\Sigma})\tilde{V}^T$ & $\sim 6np^2+6p^3+p$\\
        Parallel transport & $\left(\mathbb{P}_\Delta(\gamma_\Gamma(t))\right)^\hor_{U(t)}=$& \eqref{eq:ParallelTransportHorLift} & $\sim 5np^2 + 4np + p^2 + 4p$\\
        Riem. logarithm & $(\Log_{UU^T}^\Gr(YY^T))^\hor_U=$& Alg.~\ref{alg:modgrasslog} & $\sim 8np^2 + 2np + p^3 + p^2 + 2p$\\
        \bottomrule
    \end{tabular}
    \caption{\revcommblock{Floating point operation (FLOP) counts for some of the most commonly used formulas in this handbook, working with Stiefel representatives and assuming $n \gg p$.}}
    \label{tab:flops}
\end{table}
}

\section*{Acknowledgments}
This work was initiated when the first author was at UCLouvain for a research visit, hosted by the third author. 

\section*{Funding and/or Conflicts of Interest/Competing Interests}
The third author was supported by the Fonds de la Recherche Scientifique -- FNRS and the Fonds Wetenschappelijk Onderzoek -- Vlaanderen under EOS Project no 30468160.

Conflict of Interest: The authors declare that they have no conflict of interest.

\bibliographystyle{plainurl}
\bibliography{Grassmannbib}

\begin{thebibliography}{10}

\bibitem{AbsilEdelmanKoev2006}
P.-A. Absil, A.~Edelman, and P.~Koev.
\newblock On the largest principal angle between random subspaces.
\newblock {\em Linear Algebra and its Applications}, 414(1):288 -- 294, 2006.
\newblock \href {https://doi.org/10.1016/j.laa.2005.10.004}
  {\path{doi:10.1016/j.laa.2005.10.004}}.

\bibitem{AbsilMahonySepulchre2004}
P.-A. Absil, R.~Mahony, and R.~Sepulchre.
\newblock {R}iemannian geometry of {G}rassmann manifolds with a view on
  algorithmic computation.
\newblock {\em Acta Applicandae Mathematica}, 80(2):199--220, 2004.
\newblock \href {https://doi.org/10.1023/B:ACAP.0000013855.14971.91}
  {\path{doi:10.1023/B:ACAP.0000013855.14971.91}}.

\bibitem{AbsilMahonySepulchre2008}
P.-A. Absil, R.~Mahony, and R.~Sepulchre.
\newblock {\em Optimization Algorithms on Matrix Manifolds}.
\newblock Princeton University Press, Princeton, New Jersey, 2008.
\newblock URL: \url{http://press.princeton.edu/titles/8586.html}.

\bibitem{AfsariTronVidal2013}
B.~Afsari, R.~Tron, and R.~Vidal.
\newblock On the convergence of gradient descent for finding the {R}iemannian
  center of mass.
\newblock {\em SIAM Journal on Control and Optimization}, 51(3):2230--2260,
  2013.
\newblock \href {https://doi.org/10.1137/12086282X}
  {\path{doi:10.1137/12086282X}}.

\bibitem{Alexandrino2015}
M.~M. Alexandrino and R.~G. Bettiol.
\newblock {\em Lie Groups and Geometric Aspects of Isometric Actions}.
\newblock Springer International Publishing, Cham, 2015.
\newblock \href {https://doi.org/10.1007/978-3-319-16613-1_2}
  {\path{doi:10.1007/978-3-319-16613-1_2}}.

\bibitem{Alimisisetal2021}
F.~Alimisis, A.~Orvieto, G.~Becigneul, and A.~Lucchi.
\newblock Momentum improves optimization on {R}iemannian manifolds.
\newblock In Arindam Banerjee and Kenji Fukumizu, editors, {\em Proceedings of
  The 24th International Conference on Artificial Intelligence and Statistics},
  volume 130 of {\em Proceedings of Machine Learning Research}, pages
  1351--1359. PMLR, 13--15 Apr 2021.

\bibitem{AlimisisVandereycken2023}
F.~Alimisis and B.~Vandereycken.
\newblock Geodesic convexity of the symmetric eigenvalue problem and
  convergence of {R}iemannian steepest descent, 2023.
\newblock \href {http://arxiv.org/abs/2209.03480} {\path{arXiv:2209.03480}}.

\bibitem{AmsallemFarhat2008}
D.~Amsallem and C.~Farhat.
\newblock Interpolation method for adapting reduced-order models and
  application to aeroelasticity.
\newblock {\em AIAA Journal}, 46(7):1803--1813, 2008.
\newblock \href {https://doi.org/10.2514/1.35374} {\path{doi:10.2514/1.35374}}.

\bibitem{Balzano2015}
L.~Balzano and S.~J. Wright.
\newblock Local convergence of an algorithm for subspace identification from
  partial data.
\newblock {\em Foundations of Computational Mathematics}, 15(5):1279--1314,
  2015.
\newblock \href {https://doi.org/10.1007/s10208-014-9227-7}
  {\path{doi:10.1007/s10208-014-9227-7}}.

\bibitem{BatziesHueperMachadoLeite2015}
E.~Batzies, K.~H\"{u}per, L.~Machado, and F.~Silva~Leite.
\newblock Geometric mean and geodesic regression on {G}rassmannians.
\newblock {\em Linear Algebra Appl.}, 466:83--101, 2015.
\newblock \href {https://doi.org/10.1016/j.laa.2014.10.003}
  {\path{doi:10.1016/j.laa.2014.10.003}}.

\bibitem{Berceanu1997}
S.~Berceanu.
\newblock On the geometry of complex {G}rassmann manifold, its noncompact dual
  and coherent states.
\newblock {\em Bull. Belg. Math. Soc. Simon Stevin}, 4(2):205--243, 1997.
\newblock \href {https://doi.org/10.36045/bbms/1105731655}
  {\path{doi:10.36045/bbms/1105731655}}.

\bibitem{BergmannGousenbourger2018}
R.~Bergmann and P.-Y. Gousenbourger.
\newblock A {V}ariational {M}odel for {D}ata {F}itting on {M}anifolds by
  {M}inimizing the {A}cceleration of a {B}{\'e}zier {C}urve.
\newblock {\em Frontiers in Applied Mathematics and Statistics}, 4:59, 2018.
\newblock \href {https://doi.org/10.3389/fams.2018.00059}
  {\path{doi:10.3389/fams.2018.00059}}.

\bibitem{Bishop77}
R.~L. Bishop.
\newblock Decomposition of cut loci.
\newblock {\em Proc. Amer. Math. Soc.}, 65(1):133--136, 1977.
\newblock \href {https://doi.org/10.2307/2042008} {\path{doi:10.2307/2042008}}.

\bibitem{BorisenkoNikolaevskii1991}
A.~A. Borisenko and Yu.~A. Nikolaevski\u{\i}.
\newblock Grassmann manifolds and {G}rassmann image of submanifolds.
\newblock {\em Uspekhi Mat. Nauk}, 46(2(278)):41--83, 240, 1991.
\newblock \href {https://doi.org/10.1070/RM1991v046n02ABEH002742}
  {\path{doi:10.1070/RM1991v046n02ABEH002742}}.

\bibitem{boumal2014thesis}
N.~Boumal.
\newblock {\em Optimization and estimation on manifolds}.
\newblock PhD thesis, Universit\'e catholique de Louvain, 2014.

\bibitem{boumal2015rtrmcextended}
N.~Boumal and P.-A. Absil.
\newblock Low-rank matrix completion via preconditioned optimization on the
  {G}rassmann manifold.
\newblock {\em Linear Algebra and its Applications}, 475:200--239, 2015.
\newblock \href {https://doi.org/10.1016/j.laa.2015.02.027}
  {\path{doi:10.1016/j.laa.2015.02.027}}.

\bibitem{Chakraborty2018}
\revcomm{R.} Chakraborty and \revcomm{B.~C.} Vemuri.
\newblock {Statistics on the Stiefel manifold: Theory and applications}.
\newblock {\em The Annals of Statistics}, 47(1):415 -- 438, 2019.
\newblock \href {https://doi.org/10.1214/18-AOS1692}
  {\path{doi:10.1214/18-AOS1692}}.

\bibitem{CriscitielloBoumal2023}
C.~Criscitiello and N.~Boumal.
\newblock Curvature and complexity: Better lower bounds for geodesically convex
  optimization, 2023.
\newblock \href {http://arxiv.org/abs/2306.02959} {\path{arXiv:2306.02959}}.

\bibitem{DieciEirola1999}
L.~Dieci and T.~Eirola.
\newblock On smooth decompositions of matrices.
\newblock {\em SIAM Journal on Matrix Analysis and Applications},
  20(3):800--819, 1999.
\newblock \href {https://doi.org/10.1137/S0895479897330182}
  {\path{doi:10.1137/S0895479897330182}}.

\bibitem{DoCarmo2013riemannian}
M.~P. do~Carmo.
\newblock {\em {R}iemannian Geometry}.
\newblock Mathematics: Theory \& Applications. Birkh{\"a}user Boston, 1992.

\bibitem{EdelmanAriasSmith1999}
A.~Edelman, T.~A. Arias, and S.~T. Smith.
\newblock The geometry of algorithms with orthogonality constraints.
\newblock {\em SIAM Journal on Matrix Analysis and Applications},
  20(2):303--353, April 1998.
\newblock \href {https://doi.org/10.1137/S0895479895290954}
  {\path{doi:10.1137/S0895479895290954}}.

\bibitem{fletcher2013geodesic}
P\revcomm{.}~T. Fletcher.
\newblock Geodesic regression and the theory of least squares on {R}iemannian
  manifolds.
\newblock {\em International journal of computer vision}, 105(2):171--185,
  2013.
\newblock \href {https://doi.org/10.1007/s11263-012-0591-y}
  {\path{doi:10.1007/s11263-012-0591-y}}.

\bibitem{Gallier2011}
J.~H. Gallier.
\newblock {\em Geometric Methods and Applications: For Computer Science and
  Engineering}.
\newblock Texts in {A}pplied {M}athematics. Springer, New York, 2011.
\newblock \href {https://doi.org/10.1007/978-1-4419-9961-0}
  {\path{doi:10.1007/978-1-4419-9961-0}}.

\bibitem{Gallivan_etal2003}
K.~A. Gallivan, A.~Srivastava, X.~Liu, and P.~Van~Dooren.
\newblock Efficient algorithms for inferences on {G}rassmann manifolds.
\newblock In {\em IEEE Workshop on Statistical Signal Processing}, pages
  315--318, 2003.
\newblock \href {https://doi.org/10.1109/SSP.2003.1289408}
  {\path{doi:10.1109/SSP.2003.1289408}}.

\bibitem{Gillespie2019}
M.~Gillespie.
\newblock {\em Variations on a Theme of {S}chubert Calculus}, pages 115--158.
\newblock Springer International Publishing, Cham, 2019.
\newblock \href {https://doi.org/10.1007/978-3-030-05141-9_4}
  {\path{doi:10.1007/978-3-030-05141-9_4}}.

\bibitem{godement2017introduction}
R.~Godement and U.~Ray.
\newblock {\em Introduction to the Theory of Lie Groups}.
\newblock Universitext. Springer International Publishing, 2017.
\newblock \href {https://doi.org/10.1007/978-3-319-54375-8}
  {\path{doi:10.1007/978-3-319-54375-8}}.

\bibitem{gousenbourger2019}
\revcomm{P.-Y.} Gousenbourger, \revcomm{E.} Massart, and P.-A. Absil.
\newblock Data fitting on manifolds with composite b{\'e}zier-like curves and
  blended cubic splines.
\newblock {\em Journal of Mathematical Imaging and Vision}, 61(5):645--671,
  2019.

\bibitem{Hairer06gni}
E.~Hairer, C.~Lubich, and G.~Wanner.
\newblock {\em Geometric numerical integration}, volume~31 of {\em Springer
  Series in Computational Mathematics}.
\newblock Springer-Verlag, Berlin, second edition, 2006.
\newblock \href {https://doi.org/10.1007/3-540-30666-8}
  {\path{doi:10.1007/3-540-30666-8}}.

\bibitem{Hall_Lie2015}
B.~C. Hall.
\newblock {\em Lie Groups, Lie Algebras, and Representations: {A}n Elementary
  Introduction}.
\newblock Springer Graduate texts in Mathematics. Springer--Verlag, New York --
  Berlin -- Heidelberg, 2nd edition, 2015.
\newblock \href {https://doi.org/10.1007/978-3-319-13467-3}
  {\path{doi:10.1007/978-3-319-13467-3}}.

\bibitem{HayBorggaardPelletier2009}
A.~Hay, J.~T. Borggaard, and D.~Pelletier.
\newblock Local improvements to reduced-order models using sensitivity analysis
  of the proper orthogonal decomposition.
\newblock {\em Journal of Fluid Mechanics}, 629:41--72, 2009.
\newblock \href {https://doi.org/10.1017/S0022112009006363}
  {\path{doi:10.1017/S0022112009006363}}.

\bibitem{helgason2001differential}
S.~Helgason.
\newblock {\em Differential Geometry, Lie Groups, and Symmetric Spaces}.
\newblock Crm Proceedings \& Lecture Notes. American Mathematical Society,
  2001.

\bibitem{HelmkeHueperTrumpf2007}
U.~Helmke, K.~H\"{u}per, and J.~Trumpf.
\newblock Newton's method on {G}ra{\ss}mann manifolds, 2007.
\newblock \href {http://arxiv.org/abs/0709.2205} {\path{arXiv:0709.2205}}.

\bibitem{HelmkeMoore1994}
U.~Helmke and J.~B. Moore.
\newblock {\em Optimization and Dynamical Systems}.
\newblock Communications \& Control Engineering. Springer--Verlag, London,
  1994.
\newblock \href {https://doi.org/10.1007/978-1-4471-3467-1}
  {\path{doi:10.1007/978-1-4471-3467-1}}.

\bibitem{Higham:2008:FM}
N.~J. Higham.
\newblock {\em Functions of Matrices: {Theory} and Computation}.
\newblock Society for Industrial and Applied Mathematics, Philadelphia, PA,
  USA, 2008.
\newblock \href {https://doi.org/10.1137/1.9780898717778}
  {\path{doi:10.1137/1.9780898717778}}.

\bibitem{HornJohnson1991}
R.~A. Horn and C.~R. Johnson.
\newblock {\em Topics in matrix analysis}.
\newblock Cambridge University Press, Cambridge, 1991.
\newblock \href {https://doi.org/10.1017/CBO9780511840371}
  {\path{doi:10.1017/CBO9780511840371}}.

\bibitem{KobayashiNomizu1996}
S.~Kobayashi and K.~Nomizu.
\newblock {\em Foundations of differential geometry. {V}ol. {I \& II}}.
\newblock Wiley Classics Library. John Wiley \& Sons, Inc., New York, 1996.
\newblock Reprint of the 1963 original.

\bibitem{KochLubich2007}
O.~Koch and C.~Lubich.
\newblock Dynamical low‐rank approximation.
\newblock {\em SIAM J. Matrix Analysis Applications}, 29:434--454, 2007.
\newblock \href {https://doi.org/10.1137/050639703}
  {\path{doi:10.1137/050639703}}.

\bibitem{Kozlov}
E.~S. Kozlov.
\newblock Geometry of real {G}rassmann manifolds. {Part I - VI}.
\newblock {\em Journal of Mathematical Sciences}, 2000 -- 2001.

\bibitem{lai2020}
Z.~Lai, L.-H. Lim, and K.~Ye.
\newblock Simpler {G}rassmannian optimization, 2020.
\newblock \href {http://arxiv.org/abs/2009.13502} {\path{arXiv:2009.13502}}.

\bibitem{Lee2012smooth}
J.~M. Lee.
\newblock {\em Introduction to Smooth Manifolds}.
\newblock Graduate Texts in Mathematics. Springer New York, 2012.
\newblock \href {https://doi.org/10.1007/978-1-4419-9982-5}
  {\path{doi:10.1007/978-1-4419-9982-5}}.

\bibitem{Lee2018riemannian}
J.~M. Lee.
\newblock {\em Introduction to {R}iemannian Manifolds}, volume 176 of {\em
  Graduate Texts in Mathematics}.
\newblock Springer, Cham, 2018.
\newblock \href {https://doi.org/10.1007/978-3-319-91755-9}
  {\path{doi:10.1007/978-3-319-91755-9}}.

\bibitem{Leichtweiss1961}
K.~Leichtweiss.
\newblock Zur {R}iemannschen {G}eometrie in {G}rassmannschen
  {M}annigfaltigkeiten.
\newblock {\em Math. Z.}, 76:334--366, 1961.
\newblock \href {https://doi.org/10.1007/BF01210982}
  {\path{doi:10.1007/BF01210982}}.

\bibitem{li2019}
C.~Li, X.~Wang, J.~Wang, and J.-C. Yao.
\newblock Convergence {A}nalysis of {G}radient {A}lgorithms on {R}iemannian
  {M}anifolds {W}ithout {C}urvature {C}onstraints and {A}pplication to
  {R}iemannian {M}ass, 2019.
\newblock \href {http://arxiv.org/abs/1910.02280} {\path{arXiv:1910.02280}}.

\bibitem{MachadoSalavessa1985}
A.~Machado and I.~Salavessa.
\newblock Grassmannian manifolds as subsets of {E}uclidean spaces.
\newblock In {\em Differential geometry ({S}antiago de {C}ompostela, 1984)},
  volume 131 of {\em Res. Notes in Math.}, pages 85--102. Pitman, Boston, MA,
  1985.

\bibitem{Lui2012}
Y.~Man~Lui.
\newblock Advances in matrix manifolds for computer vision.
\newblock {\em Image and Vision Computing}, 30(6--7):380--388, 2012.
\newblock \href {https://doi.org/10.1016/j.imavis.2011.08.002}
  {\path{doi:10.1016/j.imavis.2011.08.002}}.

\bibitem{Minh:2016:AAR:3029338}
H.~Q. Minh and V.~Murino.
\newblock {\em Algorithmic Advances in {R}iemannian Geometry and Applications:
  For Machine Learning, Computer Vision, Statistics, and Optimization}.
\newblock Advances in Computer Vision and Pattern Recognition. Springer
  International Publishing, Cham, 2016.
\newblock \href {https://doi.org/10.1007/978-3-319-45026-1}
  {\path{doi:10.1007/978-3-319-45026-1}}.

\bibitem{Nguyen2012}
T.~S. Nguyen.
\newblock A real time procedure for affinely dependent parametric model order
  reduction using interpolation on {G}rassmann manifolds.
\newblock {\em International Journal for Numerical Methods in Engineering},
  93(8):818--833, 2013.
\newblock \href {https://doi.org/10.1002/nme.4408}
  {\path{doi:10.1002/nme.4408}}.

\bibitem{SonStykel2015}
T.~S. Nguyen and T.~Stykel.
\newblock Model order reduction of parameterized circuit equations based on
  interpolation.
\newblock {\em Adv. Comput. Math}, 41:1321--1342, 2015.
\newblock \href {https://doi.org/10.1007/s10444-015-9418-z}
  {\path{doi:10.1007/s10444-015-9418-z}}.

\bibitem{ONeill1983}
B.~O'Neill.
\newblock {\em Semi-{R}iemannian geometry - With applications to relativity},
  volume 103 of {\em Pure and Applied Mathematics}.
\newblock Academic Press, New York, 1983.

\bibitem{QiuZhangLi2005}
L.~Qiu, Y.~Zhang, and C.-K. Li.
\newblock Unitarily {I}nvariant {M}etrics on the {G}rassmann {S}pace.
\newblock {\em SIAM Journal on Matrix Analysis and Applications},
  27(2):507--25, 2005.
\newblock \href {https://doi.org/10.1137/040607605}
  {\path{doi:10.1137/040607605}}.

\bibitem{Rahman_etal2005}
I.~U. Rahman, I.~Drori, V.~C. Stodden, D.~L. Donoho, and P.~Schr\"oder.
\newblock Multiscale representations for manifold-valued data.
\newblock {\em SIAM Journal on Multiscale Modeling and Simulation},
  4(4):1201--1232, 2005.
\newblock \href {https://doi.org/10.1137/050622729}
  {\path{doi:10.1137/050622729}}.

\bibitem{Rentmeesters2013}
Q.~Rentmeesters.
\newblock {\em Algorithms for data fitting on some common homogeneous spaces}.
\newblock PhD thesis, Universit\'{e} {C}atholique de Louvain, Louvain, Belgium,
  2013.
\newblock URL: \url{http://hdl.handle.net/2078.1/132587}.

\bibitem{Eigenvalue_Problems_Saad}
Y.~Saad.
\newblock {\em Numerical Methods for Large Eigenvalue Problems}.
\newblock Algortihms and Architectures for Advanced Scientific Computing.
  {M}anchester {U}niversity {P}ress, Manchester, {UK}, 1992.

\bibitem{Sakai1977}
T.~Sakai.
\newblock On cut loci of compact symmetric spaces.
\newblock {\em Hokkaido Math. J.}, 6(1):136--161, 1977.
\newblock \href {https://doi.org/10.14492/hokmj/1381758555}
  {\path{doi:10.14492/hokmj/1381758555}}.

\bibitem{sakai1996riemannian}
T.~Sakai.
\newblock {\em Riemannian Geometry}.
\newblock Fields Institute Communications. American Mathematical Soc., 1996.
\newblock \href {https://doi.org/10.1090/mmono/149}
  {\path{doi:10.1090/mmono/149}}.

\bibitem{Srivastavaetal2005}
A.~Srivastava and X.~Liu.
\newblock Tools for application-driven linear dimension reduction.
\newblock {\em Neurocomputing}, 67:136 -- 160, 2005.
\newblock Geometrical Methods in Neural Networks and Learning.
\newblock \href {https://doi.org/10.1016/j.neucom.2004.11.036}
  {\path{doi:10.1016/j.neucom.2004.11.036}}.

\bibitem{RiemannInComputerVision}
A.~Srivastava and P.~K. Turaga.
\newblock {\em {R}iemannian computing in computer vision}.
\newblock Springer International Publishing, 2015.
\newblock \href {https://doi.org/10.1007/978-3-319-22957-7}
  {\path{doi:10.1007/978-3-319-22957-7}}.

\bibitem{Usevich2014}
K.~Usevich and I.~Markovsky.
\newblock Optimization on a {G}rassmann manifold with application to system
  identification.
\newblock {\em Automatica}, 50(6):1656 -- 1662, 2014.
\newblock \href {https://doi.org/10.1016/j.automatica.2014.04.010}
  {\path{doi:10.1016/j.automatica.2014.04.010}}.

\bibitem{WalterLehmannLamour2012}
S.~F. Walter, L.~Lehmann, and R.~Lamour.
\newblock On evaluating higher-order derivatives of the {QR} decomposition of
  tall matrices with full column rank in forward and reverse mode algorithmic
  differentiation.
\newblock {\em Optimization Methods and Software}, 27(2):391--403, 2012.
\newblock \href {https://doi.org/10.1080/10556788.2011.610454}
  {\path{doi:10.1080/10556788.2011.610454}}.

\bibitem{Wong1967}
Y.-C. Wong.
\newblock Differential geometry of {G}rassmann manifolds.
\newblock {\em Proceedings of the National Academy of Sciences of the United
  States of America}, 57:589--594, 1967.
\newblock \href {https://doi.org/10.1073/pnas.57.3.589}
  {\path{doi:10.1073/pnas.57.3.589}}.

\bibitem{Wong1968}
Y.-C. Wong.
\newblock Conjugate loci in {G}rassmann manifolds.
\newblock {\em Bull. Amer. Math. Soc.}, 74:240--245, 1968.
\newblock \href {https://doi.org/10.1090/S0002-9904-1968-11903-2}
  {\path{doi:10.1090/S0002-9904-1968-11903-2}}.

\bibitem{Wong1968b}
Y.-C. Wong.
\newblock Sectional curvatures of {G}rassmann manifolds.
\newblock {\em Proceedings of the National Academy of Sciences of the United
  States of America}, 60(1):75--79, 1968.
\newblock \href {https://doi.org/10.1073/pnas.60.1.75}
  {\path{doi:10.1073/pnas.60.1.75}}.

\bibitem{WuChen1988}
G.~L. Wu and W.~H. Chen.
\newblock A matrix inequality and its geometric applications.
\newblock {\em Acta Math. Sinica}, 31(3):348--355, 1988.

\bibitem{YeLim2016}
K.~Ye and L.-H. Lim.
\newblock Schubert varieties and distances between subspaces of different
  dimensions.
\newblock {\em SIAM J. Matrix Anal. Appl.}, 37(3):1176--1197, 2016.
\newblock \href {https://doi.org/10.1137/15M1054201}
  {\path{doi:10.1137/15M1054201}}.

\bibitem{ZhangBalzano2016}
D.~Zhang and L.~Balzano.
\newblock {\em Global Convergence of a {G}rassmannian Gradient Descent
  Algorithm for Subspace Estimation}, pages 1460--1468.
\newblock Proceedings of the 19th International Conference on Artificial
  Intelligence and Statistics, AISTATS, Cadiz, Spain, May 2016.

\bibitem{ZhangSra2016}
H.~Zhang and S.~Sra.
\newblock First-order methods for geodesically convex optimization.
\newblock In V.~Feldman, A.~Rakhlin, and O.~Shamir, editors, {\em 29th Annual
  Conference on Learning Theory}, volume~49 of {\em Proceedings of Machine
  Learning Research}, pages 1617--1638, Columbia University, New York, New
  York, USA, 23--26 Jun 2016. PMLR.

\bibitem{Zimmermann2014}
R.~Zimmermann.
\newblock A locally parametrized reduced order model for the linear frequency
  domain approach to time-accurate computational fluid dynamics.
\newblock {\em SIAM Journal on Scientific Computing}, 36(3):B508--B537, 2014.
\newblock \href {https://doi.org/10.1137/130942462}
  {\path{doi:10.1137/130942462}}.

\bibitem{ZimmermannPeherstorferWillcox_2017}
R.~Zimmermann, B.~Peherstorfer, and K.~Willcox.
\newblock Geometric subspace updates with applications to online adaptive
  nonlinear model reduction.
\newblock {\em {SIAM} Journal on Matrix Analysis and Application},
  39(1):234--261, 2018.
\newblock \href {https://doi.org/10.1137/17M1123286}
  {\path{doi:10.1137/17M1123286}}.

\bibitem{zimmermann2019modelreduction}
\revcomm{R.} Zimmermann.
\newblock Manifold interpolation.
\newblock In Peter Benner, Stefano Grivet-Talocia, Alfio Quarteroni, Gianluigi
  Rozza, Wil Schilders, and Luís~Miguel Silveira, editors, {\em Volume 1
  System- and Data-Driven Methods and Algorithms}, pages 229--274. De Gruyter,
  Berlin, Boston, 2021.
\newblock \href {https://doi.org/10.1515/9783110498967-007}
  {\path{doi:10.1515/9783110498967-007}}.

\bibitem{ZimmermannHermite_2020}
R\revcomm{.} Zimmermann.
\newblock {H}ermite interpolation and data processing errors on {R}iemannian
  matrix manifolds.
\newblock {\em SIAM Journal on Scientific Computing}, 42(5):A2593--A2619, 2020.
\newblock \href {https://doi.org/10.1137/19M1282878}
  {\path{doi:10.1137/19M1282878}}.

\end{thebibliography}
\end{document}